\documentclass[a4paper,reqno]{amsart}

\usepackage[utf8]{inputenc}
\usepackage[english]{babel}
\usepackage[T1]{fontenc}
\usepackage{microtype}
\usepackage{amsmath,amssymb,amsthm}
\usepackage{thmtools,thm-restate}
\usepackage{mathtools}
\usepackage{upgreek}
\usepackage{orcidlink}
\usepackage{hyperref}
\hypersetup{hidelinks}
\usepackage[nameinlink]{cleveref}
\usepackage{suffix}
\usepackage{nccmath}
\usepackage{nicefrac}
\usepackage{enumitem}

\usepackage[hang,flushmargin]{footmisc}
\crefname{equation}{equation}{equation}

\title[{Boundary delay}]{Evolutionary boundary delay equations}
\author[B.~Aigner]{Bernhard Aigner\orcidlink{0009-0009-8252-162X}}
\thanks{supported by the state of Saxony via a graduate student stipend}
\address[B.~Aigner]{Institut f\"{u}r Angewandte Analysis\\
  TU Bergakademie Freiberg\\
  Institut für Angewandte Analysis\\
  Pr\"{u}ferstr. 9, 09599 Freiberg\\
  Germany}
\email[B.~Aigner]{bernhard.aigner@doktorand.tu-freiberg.de}
\date{\today}


\DeclarePairedDelimiterX{\norm}[1]{\lVert}{\rVert}{#1}
\DeclarePairedDelimiterX{\abs}[1]{\lvert}{\rvert}{#1}
\DeclarePairedDelimiterX{\set}[1]{\{}{\}}{#1}
\DeclarePairedDelimiterX{\dset}[2]{\{}{\}}{#1\,\delimsize\vert\,\mathopen{} #2}
\DeclarePairedDelimiterX{\scprod}[2]{(}{)}{#1\delimsize| #2}
\DeclarePairedDelimiterX{\dualprod}[2]{\langle}{\rangle}{#1,#2}
\newcommand{\e}{\mathrm{e}}
\newcommand{\iu}{\mathrm{i}}
\newcommand{\dd}{\mathrm{d}}
\newcommand{\dx}[1][x]{\,\dd{}#1}
\newcommand{\argdot}{\cdot}
\renewcommand{\Re}{\operatorname{Re}}

\newcommand{\N}{\mathbb{N}}

\newcommand{\leb}{\mathrm{L}}
\newcommand{\sobh}{\mathrm{H}}
\newcommand{\sobw}{\mathrm{W}}
\newcommand{\flt}{\mathcal{L}_{\rho}}

\definecolor{bacolor}{rgb}{0.0,0.0,0.8}
\definecolor{mwcolor}{rgb}{1.0, 0.0, 0.5}

\newtheorem*{definition*}{Definition}
\newtheorem*{theorem*}{Theorem}
\newtheorem*{proposition*}{Proposition}
\newtheorem*{lemma*}{Lemma}
\newtheorem*{corollary*}{Corollary}
\newtheorem*{remark*}{Remark}

\newtheorem{definition}{Definition}[section]
\newtheorem{theorem}[definition]{Theorem}
\newtheorem{proposition}[definition]{Proposition}
\newtheorem{lemma}[definition]{Lemma}

\newtheorem{remark}[definition]{Remark}
\newtheorem{example}[definition]{Example}
\newtheorem{assumptions}[definition]{Assumptions}

\begin{document}

\begin{abstract}
  Recent well-posedness results for evolutionary partial differential equations with state-dependent inhomogeneity are extended to a larger problem class incorporating nonautonomous material behaviour. This generalization facilitates the formulation of a framework able to accommodate delayed boundary value problems related to evolutionary partial differential equations. The results permit the ad hoc treatment of Dirichlet- and Neumann-like boundary conditions with state-dependent delay. Complex boundary conditions involving nonautonomous delay in the material law and state-dependent delay in the forcing term can be accommodated by use of extended state spaces. This approach provides the first systematic treatment addressing well-posedness of several classical boundary conditions involving state-dependent delay. The viability and versatility of the theory is showcased by applications to parabolic and hyperbolic partial differential equations with Dirichlet, Neumann, Robin, Wentzell--Robin and Leontovich boundary conditions.
\end{abstract}

\maketitle

\section{Introduction}
\label{sec:Intro}
The ultimate goal of this article is to provide local existence and uniqueness results for boundary value problems associated to evolution equations
\begin{equation}
  \label{eq:EvolEq}
  \big[\partial_{t}M + N + A\big] u = F(\argdot, u_{(\argdot)})
\end{equation}
with nonautonomous and state-dependent delay in the boundary condition. In this context, $M, N \in \mathcal{L}\big(\leb^2_{\rho}(\mathbb{R};H)\big)$ are (nonautonomous) bounded linear operators over a Hilbert (product) space $H = H_{0}\!\times\! H_{1}$ and the equation is understood as an equation on an exponentially weighted Lebesgue-space $\leb^2_{\rho}(\mathbb{R};H) = \leb^2\big((\mathbb{R},\e^{-2\rho \argdot}\dx[t]);H\big)$. $u_{(t)}$ denotes the history of the state $u$, i.e.,
\begin{equation*}
  u_{(t)}\colon [-h,0] \to H, \quad s\mapsto u(t + s),
\end{equation*}
where $h>0$ is assumed to be a finite backward time horizon. $M$ and $N$ need to adhere to the structural assumptions of the framework of evolutionary equations, cf.~\cref{subsec:Basics} and $A\colon H\supseteq \operatorname{dom}(A) \to H$ is a purely spatial operator of one of the following forms:
\begin{enumerate}[label = (\roman*), leftmargin = 4ex]
  \item\label{it:BdyTriples} $A$ is the extension of a skew-symmetric operator of the form
        \begin{equation*}
          \mathring{A} = \begin{pmatrix} 0 & \mathring{D} \\ \mathring{G} & 0 \end{pmatrix},
        \end{equation*}
        where $\mathring{D}\colon H_{1}\supseteq \operatorname{dom}(\mathring{D})\to H_{0}$ and $\mathring{G}\colon H_{0}\supseteq \operatorname{dom}(\mathring{G})\to H_{1}$ are closed, densely defined, linear operators satisfying
        \begin{equation*}
          \mathring{D} \subseteq D\coloneq - \mathring{G}^{\ast} \quad\text{and}\quad \mathring{G} \subseteq G\coloneq - \mathring{D}^{\ast}.
        \end{equation*}
        This block operator setup for $\mathring{A}$ is inspired by the definition of {\em abstract boundary data spaces}, cf.~\cite{Trostorff2014}, and is related to boundary space concepts characterizing which extensions of a skew-symmetric operator $\mathring{A}$ define a maximal accretive realization $A$. In the classical setting of equal deficiency indices of $\mathring{A}$ those boundary conditions can be parametrized utilizing the concept of boundary triples, cf.~\cite{Behrndt2020}. Similarly to the semigroup setting of abstract Cauchy problems (i.e.\ $M = 1$ and $N = 0$ in \cref{eq:EvolEq}), such a maximal accretive realization assures well-posedness of the associated evolution equation appealing to Picard's theorem, cf.~\cite[thm.~6.2.1]{STW2022}.
\item\label{it:DivGrad} $A$ is of the form
        \begin{equation*}
          A = \begin{pmatrix} 0 & -C^{\ast} \\ C & 0 \end{pmatrix},
    \end{equation*}
        where $C\colon H_{0}\supseteq \operatorname{dom}(C) \to H_{1}$ is a closed, densely defined, linear operator generating an {\em abstract div-grad system}, cf.~\cite{PicardSeidlerTrostorffWaurick2016} and \cref{subsec:ExtendedState}. Since $A$ is skew-selfadjoint by definition, the associated evolution problem \eqref{eq:EvolEq} is well-posed. In contrast to approach \ref{it:BdyTriples}, the concept of abstract $\operatorname{div}$-$\operatorname{grad}$ systems can be used to pose an evolutionary \cref{eq:EvolEq} on an extended state space that allows the incorporation of complex boundary conditions, e.g.\ Robin and Wentzell--Robin boundary conditions for $\operatorname{div}$-$\operatorname{grad}$ systems or Leontovich boundary conditions for $\operatorname{curl}$-$\operatorname{curl}$ systems. A technically more sophisticated application relating to a $\operatorname{curl}$-$\operatorname{Grad}$ system, where $\operatorname{Grad}$ denotes the symmetric gradient, is the recent article \cite{Buchinger2024} pertaining to a thermo-piezo-electromagnetic system, cf.~\cite{Doherty2024}.
\end{enumerate}
As a key first step, recent results developed for evolutionary equations with state-dependent delay in \cite{AignerWaurick2026} are extended to the nonautonomous case \eqref{eq:EvolEq}. The generalization is of independent interest and in particular facilitates nonautonomous delay in the material law. The employed solution theory makes use of the theory of evolutionary equations in the sense of Picard, cf.~\cite{STW2022}, in particular of the nonautonomous solution theory developed in \cite[ch.~16]{STW2022} and \cite{TrostorffWaurick2021} by S.~Trostorff and M.~Waurick.\\
The structure of the article is as follows:\\
In \cref{sec:EvolEq} on evolutionary equations, a brief account on basic facts required for successive parts is provided; in particular, the setting of \cite{AignerWaurick2026} is recapitulated. Consecutively, two main results from \cite{AignerWaurick2026} are generalized to the nonautonomous case. The section concludes with an application of these new results to relevant examples in the literature addressing heat and wave equations with nonautonomous delay.\\
In \cref{sec:BdyDelay} on boundary value problems, a comprehensive framework to treat boundary conditions for evolutionary equations with state-dependent boundary delay is developed. Initially, a reformulation for inhomogeneous Dirichlet- and Neumann boundary conditions with state-dependent delay using approach \ref{it:BdyTriples} is provided; corresponding applications are presented in \cref{app:AbstractBD}. Consecutively, in \cref{subsec:ExtendedState}, the approach \ref{it:DivGrad} is elaborated upon: After a brief introduction to abstract $\operatorname{div}$-$\operatorname{grad}$ systems, an extended state space approach is used to model boundary conditions for $\operatorname{div}$-$\operatorname{grad}$ and $\operatorname{curl}$-$\operatorname{curl}$ systems with nonautonomous and state-dependent delay, generalizing several results in the literature. Finally, the versatility of the results is demonstrated by applications to heat, wave and Maxwell's equations.

\section{Evolutionary equations}
\label{sec:EvolEq}
In this section, the solution theory for evolution equations of the form \eqref{eq:EvolEq} is expanded upon. We commence by formally introducing the spaces and operators involved.

\subsection{Basic concepts}
\label{subsec:Basics}
The solution theory for \cref{eq:EvolEq} is formulated on weighted Bochner--Lebesgue and Bochner--Sobolev spaces:
\begin{definition}[exponentially weighted spaces]
  \label{def:LebesgueSobolev}
  For $\rho \in \mathbb{R}$, $I\subseteq \mathbb{R}$ an open interval (finite or infinite) and a Hilbert space $H$ let
  \begin{align*}
    \leb^2_{\rho}(I;H)&\coloneq \overline{\mathring{\mathcal{C}}^{\infty}\bigl(I;H\bigr)}^{\norm{\argdot}_{2,\rho}}\text{,}\\
    \norm{u}_{2,\rho}&\coloneq \Bigl(\medint\int_{I} \norm{u(t)}^{2}_{H}\e^{-2\rho t} \dx[t]\Bigr)^{\nicefrac{1}{2}}\text{,}
  \end{align*}
  where $\mathring{\mathcal{C}}^{\infty}\bigl(I;H\bigr)$ is the space of smooth functions with values in $H$ and compact support. The corresponding first order Sobolev space is given via
  \begin{align*}
    \sobh^{1}_{\rho}(I;H) &\coloneq \dset{u \in \leb^{2}_{\rho}(I;H)}{u' \in \leb^{2}_{\rho}(I;H)},\\
    \norm{u}_{\sobh^{1}_{\rho}}&\coloneq \big(\norm{u}_{2,\rho}^{2} + \norm{u'}_{2,\rho}^{2}\big)^{\nicefrac{1}{2}},
  \end{align*}
  where $u'$ denotes the distributional derivative.
\end{definition}
Higher order versions $\sobh^{k}_{\rho}(I;H)$ are defined accordingly. One advantage of the exponentially weighted space $\leb^{2}_{\rho}(\mathbb{R};H)$ on $I = \mathbb{R}$ is that the time derivative can be defined as the inverse of the antiderivative/Volterra operator:
\begin{definition}
  Let $\rho > 0$ and $H$ a Hilbert space. Let
  \begin{equation*}
    I_{\rho}\colon \leb^2_{\rho}(\mathbb{R};H) \to \leb^2_{\rho}(\mathbb{R};H), \quad \varphi \mapsto \Big( t\mapsto \medint\int_{-\infty}^{t}\varphi(s)\dx[s]\Big).
  \end{equation*}
  We define $\partial_{t,\rho}\coloneq I_{\rho}^{-1}$.
\end{definition}
That $I_{\rho}$ is an injective, bounded linear operator is verified in \cite[ch.~3]{AignerWaurick2026}, which enables the prior definition of $\partial_{t,\rho}$. The operator $\partial_{t,\rho}$ agrees with the distributional derivative, cf.~\cite[prop.~4.1.1]{STW2022}. Frequently the index $\rho$ will be omitted if $\rho$ is clear from the context or is assumed to be large enough; in that case we simply write $\partial_{t}$.
\begin{remark}[norm on $\sobh^{1}_{\rho}(\mathbb{R};H)$]
  Let $\rho \in \mathbb{R}$, $H$ a Hilbert space and $k \in \N$. In the case $I = \mathbb{R}$, $\sobh^{k}_{\rho}(\mathbb{R};H) =  \operatorname{dom}(\partial_{t,\rho}^{k})$ and $\norm{\argdot}_{\sobh^{k}_{\rho}}\coloneq \norm{\partial_{t,\rho}^{k}\argdot}_{\leb^2_{\rho}}$ provides an equivalent norm for $\rho >0$, which will be preferred in this work.
\end{remark}
The topical \cref{eq:EvolEq} will be studied on $\leb^2_{\rho}(\mathbb{R};H)$ or $\sobh^{k}_{\rho}(\mathbb{R};H)$. After the introduction of the spaces for \cref{eq:EvolEq}, next come the operators $M$, $N$ and $A$: First, $A\colon H\supseteq \operatorname{dom}(A)\to H$ is a maximal accretive, densely defined, linear operator on the Hilbert space $H$. $M$ and $N$ can be nonautonomous operators acting on $\leb^2_{\rho}(\mathbb{R};H)$ or $\sobh^{k}_{\rho}(\mathbb{R};H)$, but since it is a crucial ingredient in the solution theory to be able to vary the weight parameter $\rho$ as needed, the introduction of $M$ and $N$ begs a proper definition:
\begin{definition}[{\cite[sec.~4.2\&16.3]{STW2022}}]
  Let $H, H_{0}$ and $H_{1}$ be Hilbert spaces.
  \begin{itemize}[leftmargin = 4ex]
    \item The space of {\em simple functions with compact support} is
          \begin{equation*}
            \mathring{S}(\mathbb{R};H)\coloneq \dset{f\colon \mathbb{R}\to H}{f\;\text{simple},\; \mathrm{spt}(f)\;\text{compact}}.
          \end{equation*}
    \item A function
          \begin{equation*}
            F\colon \mathring{S}(\mathbb{R};H_{0})\to \bigcap_{\rho\geq \rho_{0}}\leb^2_{\rho}(\mathbb{R};H_{1})
          \end{equation*}
          is {\em uniformly Lipschitz-continuous} if there exists $\rho_{0}\in \mathbb{R}$ such that for all $\rho \geq \rho_{0}$, the map $F$ considered as a map in $\leb^2_{\rho}(\mathbb{R};H_{0}) \!\times \!\leb^2_{\rho}(\mathbb{R};H_{1})$ is Lipschitz-continuous and a finite Lipschitz-constant can be found uniformly w.r.t.\ $\rho$.
    \item A map
          \begin{equation*}
            M\colon \mathring{S}(\mathbb{R};H) \to \bigcap_{\rho\geq \rho_{0}}\leb^2_{\rho}(\mathbb{R};H)
          \end{equation*}
          is {\em evolutionary at $\rho_{0} \in \mathbb{R}$} if it is linear and uniformly Lipschitz-continuos at $\rho_{0}$, i.e.\ for all $\rho\geq \rho_{0}$, the map
          \begin{equation*}
            M\colon \leb^2_{\rho}(\mathbb{R};H)\supseteq \mathring{S}(\mathbb{R};H) \to \leb^2_{\rho}(\mathbb{R};H)
          \end{equation*}
          is linear and continuous. Its continuous extension to $\leb^2_{\rho}(\mathbb{R};H)$ is denoted $M^{\rho}$.
    \item Let $\rho_{0}\in \mathbb{R}$. The set of all evolutionary maps is defined as
          \begin{equation*}
            S_{\mathrm{ev}}(H,\rho_{0}) \coloneq \dset[\Big]{M\colon \mathring{S}(\mathbb{R};H) \to \bigcap_{\rho\geq \rho_{0}}\leb^2_{\rho}(\mathbb{R};H)}{M\,\text{is evolutionary at}\ \rho_{0}}.
          \end{equation*}
  \end{itemize}
\end{definition}
Lastly, a reminder on the following important notion:
\begin{definition}
  \label{def:Causality}
  Let $H_{0}$, $H_{1}$ be Hilbert spaces, $\rho \in \mathbb{R}$ and $F\colon \leb^2_{\rho}(\mathbb{R};H_{0}) \to \leb^2_{\rho}(\mathbb{R};H_{1})$. $F$ is called {\em causal} if
  \begin{equation*}
    \forall a \in \mathbb{R}\;\forall f,g \in \leb^2_{\rho}(\mathbb{R};H_{0})\colon f\vert_{(-\infty,a]}=g\vert_{(-\infty,a]}\implies F(f)\vert_{(-\infty,a]} = F(g)\vert_{(-\infty,a]}.
  \end{equation*}
\end{definition}

\subsubsection{Well-posedness in $\leb^2_{\rho}$}
The main result in the theory for nonautonomous evolutionary equations is the following generalization of the autonomous version of Picard's theorem \cite[thm.~6.2.1]{STW2022}:
\begin{theorem}[Well-posedness in $\leb^2_{\rho}$, {\cite[thm.~16.3.1]{STW2022}}]
  \label{th:Picard}
  Let $\rho_{0} \in \mathbb{R}$ and $M,M',N \in S_{\mathrm{ev}}(H,\rho_{0})$ such that $\Re M^{\rho} \geq 0$ for all $\rho \geq \rho_{0}$. Let $A\colon H \supseteq \operatorname{dom}(A)  \to H$ be maximal accretive. Assume that there exists $c>0$ such that
  \begin{enumerate}[label = (\roman*), leftmargin = 5ex]
    \item\label{it:Commutator} $M^{\rho_{0}}\partial_{t,\rho_{0}} \subseteq \partial_{t,\rho_{0}}M^{\rho_{0}} - (M')^{\rho_{0}}$.
    \item\label{it:Coercivity} for all $\rho \geq \rho_{0}$ and $\varphi \in \operatorname{dom}(\partial_{t,\rho})$,
          \begin{equation*}
            \Re \dualprod[\big]{\varphi}{(\partial_{t,\rho}M^{\rho} + N^{\rho})\varphi}_{\leb^2_{\rho}(\mathbb{R};H)} \geq c \norm{\varphi}^{2}_{\leb^2_{\rho}(\mathbb{R};H)}.
          \end{equation*}
  \end{enumerate}
  Then for all $\rho \geq \max \{\rho_{0},0\}$, $\rho \neq 0$, the operator
  \begin{equation*}
    \partial_{t,\rho}M^{\rho} + N^{\rho} + A \colon \leb^2_{\rho}(\mathbb{R};H)\supseteq \sobh^{1}_{\rho}(\mathbb{R};H)\cap \leb^2_{\rho}\big(\mathbb{R};\operatorname{dom}(A)\big) \to \leb^2_{\rho}(\mathbb{R};H)
  \end{equation*}
  is closable and its closure is continuously invertible.\\
  Moreover, let $S_{\rho} \in \mathcal{L}\big(\leb^2_{\rho}(\mathbb{R};H)\big)$ denote the inverse of this closure. Then
  \begin{itemize}[leftmargin = 4ex]
    \item $\norm{S_{\rho}}_{\mathcal{L}(\leb^2_{\rho}(\mathbb{R};H))} \leq \nicefrac{1}{c}$,
    \item $S_{\rho}$ is eventually independent of $\rho$ and
    \item $S_{\rho}$ is causal.
  \end{itemize}
\end{theorem}
This theorem establishes well-posedness of problem \eqref{eq:EvolEq} for right-hand sides $f \in \leb^2_{\rho}(\mathbb{R};H)$ by virtue of providing a solution operator $S_{\rho} = \big(\overline{\partial_{t}M+N+A}\big)^{-1} \in \mathcal{L}\big(\leb^2_{\rho}(\mathbb{R};H)\big)$. We employ a corresponding notion of weak solutions:
\begin{definition}
  \label{def:weakSol}
  For $\rho > s_{\mathrm{b}}(M)$ and a right-hand side $g \in \leb^2_{\rho}(0,\infty;H)$, we refer to $u=S_{\rho}g$ as a {\em solution} to
  \begin{equation*}
    \big[\partial_{t}M + N + A\big] u = g.
  \end{equation*}
\end{definition}

\begin{remark}\phantom{.}
  \begin{itemize}[leftmargin = 4ex]
    \item The condition $\Re M^{\rho} \geq 0$ for all $\rho \geq \rho_{0}$ in \cref{th:Picard} is only required for causality of the solution operator $S_{\rho}$.
    \item Condition \ref{it:Commutator} is always satisfied in the autonomous case appealing to tensor product structure; $M \partial_{t,\rho} \subseteq \partial_{t,\rho}M$ and thus $M' = 0$.
    \item Condition \ref{it:Coercivity} in the autonomous case is satisfied in particular when $M$ is selfadjoint and the conditions $M\vert_{\operatorname{ran}(M)} \geq c_{0}>0$ and $\Re N\vert_{\ker M}\geq c_{1}>0$ are satisfied.
  \end{itemize}
\end{remark}

\subsubsection{Well-posedness in $\sobh^{k}_{\rho}$}
It will turn out that well-posedness in $\leb^2_{\rho}(\mathbb{R};H)$ assured by \cref{th:Picard} does not suffice for the general nonautonomous case of \cref{eq:EvolEq}. In addition, well-posedness in $\sobh^{k}_{\rho}(\mathbb{R};H)$ will be required for $k \in \{1,2\}$. Existence and uniqueness of solutions with higher regularity can be verified following the approach from \cite[sec.~3]{TrostorffWaurick2021} verbatim by exchanging $\sobh^{\nicefrac{1}{2}}_{\rho}(\mathbb{R};H)$ there with $\sobh^{k}_{\rho}(\mathbb{R};H)$ here. On a technical level, the approach mostly relies on algebraic transformations and the arguments hold true irrespective of the considered exponent.\footnote{In fact, the statement holds true for $\sobh^{\alpha}_{\rho}(\mathbb{R};H)$ for any $\alpha \in \mathbb{R}_{>0}$.} The result is the following existence and uniqueness result.

\begin{theorem}[Well-posedness in $\sobh^{k}_{\rho}$]
  \label{th:PicardHigh}
  Let $k \in \mathbb{N}$, $\rho_{0} \in \mathbb{R}$ and $A\colon \operatorname{dom}(A)\subseteq H \to H$ be maximal accretive. Let $M,M',N \in S_{\mathrm{ev}}(H,\rho_{0})$  with $\Re M^{\rho} \geq 0$ for all $\rho \geq \rho_{0}$. Assume that there exists $c>0$ such that
  \begin{enumerate}[label = (\roman*), leftmargin = 5ex]
    \item $M^{\rho_{0}}\vert_{\sobh^{k}_{\rho}}, (M')^{\rho_{0}}\vert_{\sobh^{k}_{\rho}},N^{\rho_{0}}\vert_{\sobh^{k}_{\rho}} \in \mathcal{L}\big(\sobh^{k}_{\rho}(\mathbb{R};H)\big)$.
    \item $M^{\rho_{0}}\partial_{t,\rho_{0}} \subseteq \partial_{t}M^{\rho_{0}} - (M')^{\rho_{0}}$.
    \item for all $\rho \geq \rho_{0}$ and $\varphi \in \operatorname{dom}\big(\partial_{t,\rho}^{k+1}\big)$ one has
          \begin{equation*}
            \Re \dualprod[\big]{\varphi}{(\partial_{t,\rho}M^{\rho} + N^{\rho})\varphi}_{\sobh^{k}_{\rho}(\mathbb{R};H)} \geq c \norm{\varphi}^{2}_{\sobh^{k}_{\rho}(\mathbb{R};H)}.
          \end{equation*}
  \end{enumerate}
  Then for all $\rho \geq \max \{\rho_{0},0\}$, $\rho \neq 0$, the operator
  \begin{equation*}
    \partial_{t,\rho}M^{\rho} + N^{\rho} + A \colon \sobh^{k}_{\rho}(\mathbb{R};H)\supseteq \sobh^{k+1}_{\rho}(\mathbb{R};H)\cap \sobh^{k}_{\rho}\big(\mathbb{R};\operatorname{dom}(A)\big) \to \sobh^{k}_{\rho}(\mathbb{R};H)
  \end{equation*}
  is closable and its closure is continuously invertible. Moreover, let $S_{\rho} \in \mathcal{L}\big(\sobh^{k}_{\rho}(\mathbb{R};H)\big)$ denote the inverse of this closure. Then
  \begin{itemize}[leftmargin = 4ex]
    \item $\norm{S_{\rho}}_{\mathcal{L}(\sobh^{k}_{\rho}(\mathbb{R};H))} \leq \nicefrac{1}{c}$,
    \item $S_{\rho}$ is eventually independent of $\rho$ and
    \item $S_{\rho}$ is causal.
  \end{itemize}
\end{theorem}
Succinctly put, one only needs to replace the coercivity estimate in $\leb^2_{\rho}(\mathbb{R};H)$ with one in $\sobh^{k}_{\rho}(\mathbb{R};H)$.

\begin{proof}
  All statements follow by the same arguments as in \cite[sec.~3]{TrostorffWaurick2021}; the only missing statement there is the causality. For this, we first observe that for $f \in \sobh^{k}_{\rho}(\mathbb{R};H)$,
  \begin{equation*}
    \norm{S_{\rho}f}_{\leb^2_{\rho}}
    = \norm{I_{\rho}^{k}S_{\rho}f}_{\sobh^{k}_{\rho}}
    \leq \tfrac{1}{\rho^{k}} \norm{S_{\rho}f}_{\sobh^{k}_{\rho}}
    \leq \tfrac{1}{c\rho^{k}} \norm{f}_{\sobh^{k}_{\rho}}
    = \tfrac{1}{c\rho^{k}}\norm{\partial_{t,\rho}^{k}f}_{\leb^2_{\rho}},
  \end{equation*}
  where the fact $\norm{I_{\rho}} = \tfrac{1}{\rho}$ was used as well as the estimate $\norm{S_{\rho}}_{\mathcal{L}(\sobh^{k}_{\rho})}\leq \tfrac{1}{c}$. Now let $f \in \sobh^{k}_{\rho}(\mathbb{R};H)$ with $\mathrm{spt}(f) \subseteq (\alpha,\infty)$. A similar calculation as in the proof of \cite[lem.~4.2.5(ii)]{STW2022} shows
  \begin{align*}
    \medint\int_{-\infty}^{a} \norm{(S_{\rho}f)(s)}^{2}\e^{2\rho(a -s)}\dx[s]
    &= \e^{2\rho a}\norm{\chi_{(-\infty,a]}S_{\rho}f}_{\leb^2_{\rho}}^{2}\\
    &\leq \e^{2\rho a}\norm{S_{\rho}f}_{\leb^2_{\rho}}^{2}\\
    &\leq \e^{2\rho a}\tfrac{1}{c^{2}\rho^{2k}}\norm{\partial_{t,\rho}^{k}f}_{\leb^{2}_{\rho}}^{2}\\
    &= \tfrac{1}{c^{2}\rho^{2k}} \medint\int_{\mathbb{R}} \norm{\partial_{t,\rho}^{k}f (s)}^{2} \e^{2\rho (a-s)}\dx[s].
  \end{align*}
  The assumption $\mathrm{spt}(f) \subseteq (\alpha,\infty)$ in particular implies that $\mathrm{spt}(\partial_{t,\rho}^{k}f) \subseteq (\alpha,\infty)$ and thus an application of the monotone convergence theorem implies
  \begin{equation*}
    \medint\int_{-\infty}^{a} \norm{(S_{\rho}f)(s)}^{2}\e^{2\rho(a -s)}\dx[s] \leq  \tfrac{1}{c^{2}\rho^{2k}} \medint\int_{a}^{\infty} \norm{\partial_{t,\rho}^{k}f (s)}^{s} \e^{2\rho (a-s)}\dx[s] \xrightarrow[]{\rho \to \infty} 0.
  \end{equation*}
  Consequently, $\mathrm{spt}(S_{\rho}f)\subseteq (\alpha,\infty)$.
\end{proof}

\newpage
\subsection{Evolutionary equations with state-dependent delay}
\label{subsec:EvolEqSDD}

\subsubsection{Key notions for well-posedness}
\label{subsubsec:DefintionsSDD}
To start off, some basic terminology and tools necessitated by the presence of state-dependent delay in the inhomogeneity of \cref{eq:EvolEq}, i.e.
\begin{equation*}
  [\partial_{t}M + N + A]u = F(\argdot, u_{(\argdot)}),
\end{equation*}
is recalled. Some key notions and findings of \cite{AignerWaurick2026} are recapitulated before the central results \cite[thm.~4.1\&4.2]{AignerWaurick2026} are subsequently generalized to the nonautonomous case. The framework from \cite{AignerWaurick2026} requires several assumptions on the right-hand side of \cref{eq:EvolEq}. For existence and uniqueness a Lipschitz-property has to be assumed.
\begin{definition}
  \label{def:AlmLC}
  A function $G\colon [0,\infty) \!\times\! \leb^{2}(-h,0;H) \to H$ is {\em almost uniformly Lipschitz-continuous} if it is continuous and for all $\alpha >0$ there exists $L_{\alpha}>0$ s.t.
  \begin{equation*}
    \forall t\in [0,\infty)\,\forall \phi,\psi \in V_{\alpha}\colon \norm{G(t,\phi) - G(t,\psi)}_{H}\leq L_{\alpha}\norm{\phi - \psi}_{\sobh^{1}(-h,0;H)},
  \end{equation*}
  where we set $V_{\alpha}\coloneq \dset{u \in \sobh^{1}(-h,0;H)}{\norm{u'}_{\infty}\leq \alpha}\subseteq \leb^{2}(-h,0;H)$.
\end{definition}
Already in the  ODE case (i.e.\ $M = 1$, $N = 0 = A$), Lipschitz-continuity of prehistories is required for uniqueness of solutions, cf.~\cite[ex.~4.3]{Waurick2023}. This fact is reflected in the definition of  $V_{\alpha}$. Note that $\bigcup_{\alpha >0}V_{\alpha} = \mathrm{W}^{1,\infty}(-h,0;H)$. The method to show existence and uniqueness of a local solution will rely on an application of the contraction mapping principle. More specifically, one aims to solve a fixed point problem on $V_{\alpha}$ first and tries to recover a local solution of the original problem. A key ingredient for the latter in \cite{AignerWaurick2026} was regularity theory for the solution operator $S_{\rho}$ provided by the autonomous version of \cref{th:Picard}, cf.~\cite[thm.~6.2.1]{STW2022}. In particular, the fixed point argument in the corresponding proofs of \cite[thm.~4.1\&4.2]{AignerWaurick2026} require that the regularity of the unknown can be pulled through the solution operator.
\begin{definition}
  \label{def:RegPres}
  $F$ is called {\em regularity preserving} if for $u \in \sobh^{1}_{\rho}(-h,\infty;H)$ with $\norm{u'}_{\infty}<\infty$, the function $t\mapsto F\bigl(t, u_{(t)}\bigr)$ defines an element of $\sobh^{1}_{\rho}(0,\infty;H)$.
\end{definition}

\subsubsection{A note on initial value problems}
\label{subsubsec:InitialValues}
Throughout this section, let $A\colon H\supseteq \operatorname{dom}(A)\to H$ be a maximal accretive, linear operator and let $M,M'$ and $N$ satisfy the assumptions of \cref{th:Picard} for all $\rho>0$ large enough. In the context of initial conditions, the problem one aims to solve is the informal problem
\begin{equation*}
  \left\{
    \begin{aligned}
      \big[\partial_{t}M(t) + N(t) + A\big]u(t) &= F\big(t,u_{(t)}\big),\quad t>0,\\
      u_{(0)} &= \Phi,
    \end{aligned}
  \right.
\end{equation*}
for some $\Phi \in \sobw^{1,\infty}(-h,0;H)\subseteq \sobh^{1}(-h,0;H)$. However, the first equation is not well-defined, because evolutionary equations need to be posed on the entire real axis. One can remedy this problem by transitioning to a distributional formulation in $\sobh^{-1}_{\rho}(\mathbb{R};H) \coloneq \big(\sobh^{1}(\mathbb{R};H)\big)'$, cf.~\cite{Trostorff2018}, resulting in an equation
\begin{equation*}
  \big[\partial_{t}M + N + A\big]v = f + F\big(\argdot, (v+Z)_{(\argdot)}\big) \qquad \text{in}\ \sobh^{-1}(\mathbb{R};H),
\end{equation*}
where one makes the ansatz $u = v + Z$ with $v$ having support only on $[0,\infty)$ and $Z$ being a continuation of the initial prehistory. By additionally requiring that the functions on the right-hand side are in $\leb^2_{\rho}(\mathbb{R};H)$, the problem at hand takes the form
\begin{equation}
  \label{eq:EvolEqSDD}
  \big[\partial_{t}M + N + A\big]v = f + F\big(\argdot, (v + Z)_{(\argdot)}\big).
\end{equation}
This is the general class of equations studied in \cite{AignerWaurick2026}. As argued there, the motivation for the transition to this case is that the focus of this investigation lies on the well-posedness aspects of the problem and not in the challenging task of identifying the ``correct'' spaces for initial data for (generalized) differential algebraic equations. We simply point out that for simple enough cases of $M$ and $N$ (e.g.\ autonomous $M$ and $N$) it is no challenge to identify a suitable $f$, which encodes the interaction of the initial prehistory with the operator on the left, cf.~\cite[sec.~5\&6]{AignerWaurick2026} and \cite{Trostorff2018} for the general formulation of initial value problems within the framework of evolutionary equations. Solvability of \cref{eq:EvolEq} will implicitly be  understood to refer to solvability of \cref{eq:EvolEqSDD}.

\subsubsection{Fixed point problem}
Provided that the assumptions of \cref{th:Picard} hold, the solution operator $S_{\rho}$ can be applied to \cref{eq:EvolEqSDD} and the resulting equation can be used to pose a fixed point problem for the following map:
\begin{equation}
  \label{eq:FPP}
  \Gamma_{\rho}\colon \leb^2_{\rho}(\mathbb{R};H) \to \leb^2_{\rho}(\mathbb{R};H), \quad v \mapsto S_{\rho}\big[f + F(\argdot, (v + Z)_{(\argdot)})\big].
\end{equation}
One can modify this map by considering a projected version, utilizing the metric projection $\pi_{\alpha}\colon \leb^2(-h,0;H)\to V_{\alpha}$ onto the subset $V_{\alpha}$ from \cref{def:AlmLC}. Note that $\pi_{\alpha}$ exists, because $V_{\alpha}$ is a closed and convex subset of $\sobh^{1}(-h,0;H)$, cf.~\cite[lem.~3.3]{AignerWaurick2026}. The modified fixed point map reads
\begin{align}
  \begin{aligned}
    \label{eq:FPPmod}
    \Gamma_{\rho,\alpha}\colon \, \leb^2_{\rho}(0,\infty;H)&\rightarrow \leb^2_{\rho}(0,\infty;H)\text{,}\\
    v &\mapsto S_{\rho}f + S_{\rho}F\bigl(\argdot,\pi_{\alpha}((v+Z)_{(\argdot)})\bigr).
  \end{aligned}
\end{align}
Correspondingly, as a local solution of \cref{eq:EvolEqSDD} we will understand a fixed point of $\Gamma_{\rho,\alpha}$ that locally around $t=0$ does not require the metric projection:
\begin{definition}
  A {\em local} solution $v \in \leb^2_{\rho}(0,\infty;H)$ is a fixed point of $\Gamma_{\rho,\alpha}$ for some $\alpha>0$, that satisfies $\pi_{\alpha}\bigl((v+Z)_{(t)}\bigr)=(v+Z)_{(t)}$ for $t\in [0,T]$ up to some $T>0$.
\end{definition}

\subsubsection{Assumptions for well-posedness}
Finally, we recall the assumptions for well-posedness from \cite{AignerWaurick2026}, which will be used in the nonautonomous case as well.
\begin{assumptions}
  \label{assumptions}
  \phantom{.}
  \begin{enumerate}[label=(\Alph*), leftmargin=5ex]
    \item \label{ass:A} Let $G\colon [0,\infty) \!\times \!\leb^2(-h,0;H)\to H$
          \begin{enumerate}[label=\roman*), leftmargin=5ex]
            \item be almost uniformly Lipschitz-continuous (cf.~\cref{def:AlmLC}),
            \item regularity preserving (cf.~\cref{def:RegPres})
            \item and satisfy $G(\argdot,0)\in \leb^{2}_{\rho}(0,\infty;H)$.
          \end{enumerate}
    \item \label{ass:B} Let $\Phi \in \sobh^{1}(-h,0;H)$ with $\norm{\Phi'}_{\infty}<\infty$ and let $Z$ be an $\sobh^{1}_{\rho}(-h,\infty;H)$-extension of $\Phi$ satisfying
          \begin{enumerate}[label=\roman*), leftmargin=5ex]
            \item $Z\vert_{(-h,0]}=\Phi$ and
            \item $Z\vert_{[0,\infty)}\in \mathcal{C}^{1}(0,\infty;H)$ with $\norm{Z'(0)}_{H}\leq \norm{\Phi'}_{\infty}$.
          \end{enumerate}
  \end{enumerate}
\end{assumptions}

The developed well-posedness theory in \cite{AignerWaurick2026} relies on the inhomogeneity carrying state-dependent delay having a specific form. Let
\begin{align*}
  \widetilde{F}(v)(t) &\coloneq F\bigl(t, (v + Z)_{(t)}\bigr)\text{,}\\
  \widetilde{F}_{\alpha}(v)(t) &\coloneq F\bigl(t, \pi_{\alpha}\bigl((v + Z)_{(t)}\bigr)\bigr)\text{.}
\end{align*}
The cases considered in \cite{AignerWaurick2026} were
\begin{enumerate}[leftmargin = 4ex]
  \item $\widetilde{F}(v)=I_{\rho}\widetilde{G}(v)$.\\
        Here, $F\bigl(t,(v + Z)_{(t)}\bigr)=I_{\rho}G\bigl(\argdot,(v+Z)_{(\argdot)}\bigr)(t)$ with $G$ satisfying assumption \ref{ass:A}, which gives rise to the fixed point problem
        \begin{equation}
          \label{eq:FPPIntOut}
          v = S_{\rho}\bigl[f + I_{\rho}G\bigl(\argdot,(v+Z)_{(\argdot)}\bigr)\bigr]
          \qquad\text{in}\,\leb^2_{\rho}(0,\infty;H)\text{.}
        \end{equation}
  \item $\widetilde{F}(v)=\widetilde{G}(I_{\rho}v)$.\\
        Here, $F\bigl(t,(v + Z)_{(t)}\bigr)=G\bigl(t,(I_{\rho}v + Z)_{(t)}\bigr)$ with $G$ satisfying assumption \ref{ass:A}, which gives rise to the fixed point problem
        \begin{equation}
          \label{eq:FPPIntIns}
          v = S_{\rho}\bigl[f + G\bigl(\argdot,(I_{\rho}v+Z)_{(\argdot)}\bigr)\bigr]
          \qquad\text{in}\,\leb^2_{\rho}(0,\infty;H)\text{.}
        \end{equation}
\end{enumerate}
In either case, it should be noted that the right-hand side of \eqref{eq:FPPmod} has support only on $[0,\infty)$, because $f$ and $G$ have and the solution operator $S_{\rho}$ is causal.

\subsection{Well-posedness for the non-autonomous case}
\label{subsec:WellPosedness}
We study the two cases mentioned above. Since the proofs of the respective well-posedness results are only slight modifications of the corresponding proofs from \cite{AignerWaurick2026}, nothing but the necessary changes in the proofs are pointed out here. However, because the results are key to the investigation in this article, full proofs are provided in \cref{app:WellPosedness}.

\subsubsection*{The case $\widetilde{F}(v)=I_{\rho}\widetilde{G}(v)$}
\begin{theorem}[local existence and uniqueness for \cref{eq:FPPIntOut}]
  \label{th:localExistOut}
  Let $G$ satisfy assumption \ref{ass:A} and let $\Phi$ and $Z$ satisfy assumption \ref{ass:B}. Let $A$, $M$, $M'$ and $N$ satisfy the assumptions of \cref{th:Picard} and \cref{th:PicardHigh} for $k = 1$ and $k=2$. Let $f \in \sobh_{\tilde{\rho}}^{2}(0,\infty;H)$ for some $\tilde{\rho} >0$ satisfy $f(0) = 0$ and $G(0,\Phi) = - f'(0)$. Then \cref{eq:FPPIntOut} has a unique local solution $u\in \sobh^{2}(0,T;H)$.
\end{theorem}

\begin{proof}
  Let $\alpha > \norm{\Phi'}_{\infty}$ and $\rho > \max\bigl\{\tfrac{L_{\alpha}^{2}}{2 c^{2}},\tilde{\rho}\bigr\}$, where $L_{\alpha}$ is the Lipschitz-constant of $G$ associated to the parameter $\alpha$, cf.~\cref{def:AlmLC}, and $c$ is the minimum of the coercivity bounds from \cref{th:PicardHigh} (for $k =1$ and $k = 2$) and \cref{th:Picard}.\\
  We aim to show the existence of a fixed point of
  \begin{multline*}
    v = S_{\rho}\bigl[f + I_{\rho}G\bigl(\argdot,\pi_{\alpha}((v+Z)_{(\argdot)})\bigr)\bigr] = S_{\rho}\bigl[f + I_{\rho}\widetilde{G}_{\alpha}(v)\bigr]\eqcolon \Gamma_{\rho,\alpha}(v)\text{,}\\
    \text{for}\ v\in \sobh^{1}_{\rho}(0,\infty;H)\ \text{satisfying}\ v(0)=0
  \end{multline*}
  by showing that $\Gamma_{\rho,\alpha}$ satisfies the requirements of the contraction mapping principle. This can be achieved as in the proof of \cite[thm.~4.1]{AignerWaurick2026} by replacing the regularity statements $\partial_{t}S_{\rho} \supseteq S_{\rho}\partial_{t}$ there with well-posedness in $\sobh^{1}_{\rho}(\mathbb{R};H)$ appealing to \cref{th:PicardHigh} here. The contraction mapping principle then provides a solution $w$ of
  \begin{equation*}
    w = \Gamma_{\rho,\alpha}(w)\text{,} \quad w\in \sobh^{1}_{\rho}(0,\infty;H)\;\text{with}\ w(0)=0\text{.}
  \end{equation*}
  To obtain a local solution of \cref{eq:FPPIntOut}, it needs to be argued that at least up to some $T>0$, $\pi_{\alpha}\bigl((w + Z)_{(\argdot)}\bigr) = (w + Z)_{(\argdot)}$. This is done by a continuity argument:\\
  First recall that $f\in \sobh^{2}_{\rho}(0,\infty;H)$, and because $G$ is regularity preserving, $\widetilde{G}_{\alpha}(w) \in \sobh^{1}_{\rho}(0,\infty;H)$ and consequently $I_{\rho}\widetilde{G}_{\alpha}(w) \in \sobh^{2}_{\rho}(0,\infty;H)$. By assumption, the sum $f + I_{\rho}\widetilde{G}_{\alpha}(w) \in \sobh^{2}_{\rho}(\mathbb{R};H)$. Well-posedness in $\sobh^{2}_{\rho}(\mathbb{R};H)$ appealing to \cref{th:PicardHigh} for $k =2$ now assures that $w \in \sobh^{2}_{\rho}(\mathbb{R};H)$. Causality of $S_{\rho}$ implies that $\mathrm{spt}(w) \subseteq [0,\infty)$. Because $w \in \sobh^{2}_{\rho}(\mathbb{R};H)$, in particular $w'(t) = 0$ for $t<0$ and hence $w' \in \sobh^{1}_{\rho}(0,\infty;H)$ with $w'(0) = 0$.\\
  Since $w = 0$ on $(-\infty,0]$ and $\Phi \in V_{\alpha}$, we only need to show that $w' + Z'$ is bounded in $\norm{\argdot}_{\infty}$-norm by $\alpha$ up to some positive time $T>0$. Since $Z'$ is bounded on any interval $[0,c]$, $c>0$, and $\norm{Z'(0)}_{H}< \alpha$ by assumption \ref{ass:B} ii), $\norm{Z'}_{\mathcal{C}(0,\tilde{T};H)}<\alpha$ holds up to some positive $\tilde{T}>0$. Since $w'$ is continuous and $w'(0)=0$, up to some positive $0<T\leq \tilde{T}$, $w'+Z'$ remains bounded in $\norm{\argdot}_{\infty}$-norm by $\alpha$. This establishes existence of a local solution.\\
  Uniqueness of solutions is shown verbatim as in the proof of \cite[thm.~4.1]{AignerWaurick2026}.
\end{proof}

\subsubsection*{The case $\widetilde{F}(v)=\widetilde{G}(I_{\rho}v)$}\phantom{.}
\begin{theorem}[local existence and uniqueness for \cref{eq:FPPIntIns}]
  \label{th:localExistIns}
  Let $G$ satisfy assumption \ref{ass:A} and let $\Phi$ and $Z$ satisfy  assumption \ref{ass:B}. Let $M$, $M'$, $N$ and $A$ satisfy the assumptions of \cref{th:Picard} and \cref{th:PicardHigh} for $k=1$. Let $f\in \sobh^{1}_{\tilde{\rho}}(0,\infty;H)$ for some $\tilde{\rho}>0$ satisfy $f(0) = -G\bigl(0, \Phi\bigr)$. Then \cref{eq:FPPIntIns} has a unique local solution $u \in \sobh^{1}(0,T;H)$.
\end{theorem}
Here, because of the similarity to the prvious case, a sketch of the proof is omitted entirely and the details deferred to \cref{app:WellPosedness}. Instead we mention that in both cases $\widetilde{F}(v) = I_{\rho}\widetilde{G}(v)$ and $\widetilde{F}(v) = \widetilde{G}(I_{\rho}v)$, we can show the existence of a maximal existence interval and a blowup criterion for the maximal time $T_{\max}$. The reason for that is that the proofs of \cref{th:localExistIns} and \cref{th:localExistOut} rely on the contraction mapping principle and are essentially generalizations of the ODE proof for \cite[thm.~4.1]{Waurick2023}, cf.~\cref{rmk:maxExistIntOut} and \cref{rmk:maxExistIntIns}.\\
For a first application, we can generalize the two examples \cite[ex.~6.1\&6.3]{AignerWaurick2026} to the nonautonomous case, as treated in \cite[ch.~16]{STW2022} without delay.
\begin{example}[Heat equation]
Let $\Omega \subseteq \mathbb{R}^{d}$ be open and let the heat conductivity $a\colon \mathbb{R} \!\times \!\Omega \to \mathbb{R}^{d\times d}$ satisfy:
\begin{enumerate}[label = (\roman*), leftmargin = 5ex]
  \item $a$ is bounded and measurable.
  \item There exists $c>0$ such that
        \begin{equation*}
          \Re a(t,x) \geq c \quad \text{for a.a.}\ (t,x)\in \mathbb{R} \!\times \!\Omega.
        \end{equation*}
  \item $a$ is Lipschitz-continuous w.r.t.\ the temporal variables uniformly in space; i.e.\ there exists $L\geq 0$ such that
        \begin{equation*}
          \norm[\big]{a(s,x) - a(t,x)}_{\mathbb{R}^{d\times d}} \leq L \abs{t-s} \quad \forall s,t \!\in \!\mathbb{R}\;\forall x \!\in \!\Omega.
        \end{equation*}
  \item The time derivative $a_{t}\colon \mathbb{R} \!\times \!\Omega \to \mathbb{R}^{d\times d}$ is bounded and measurable.
\end{enumerate}
Then the nonautonomous system describing heat conduction is given as
\begin{equation*}
  \left\{
  \begin{aligned}
    \partial_{t,\rho}\theta(t,x) + \mathring{\operatorname{div}} q(t,x) &= Q(t,x) &\quad\text{for a.a.}\ (t,x) \in \mathbb{R} \!\times \!\Omega,\\
    q(t,x) &= -a(t,x) \operatorname{grad} \theta(t,x) &\quad\text{for a.a.}\ (t,x) \in \mathbb{R} \!\times \!\Omega,
  \end{aligned}
  \right.
\end{equation*}
where $\theta$ denotes the heat distribution, $Q$ a given heat source and $q$ is the heat flux.
\begin{itemize}[leftmargin = 4ex]
  \item $\operatorname{grad}\colon \leb^2(\Omega)\supseteq \sobh^{1}(\Omega) \to \leb^2(\Omega)^{d}$ denotes the weak gradient  and
  \item $\mathring{\operatorname{div}}\coloneq \operatorname{grad}^{\ast}$ the weak divergence with Dirichlet boundary conditions.
\end{itemize}
We show that the coercivity estimate required for \cref{th:PicardHigh} for $k =1$ is satisfied. Using $(a^{-1})_{t} = - a^{-1}a_{t}a^{-1}$ (obtained from differentiating $I = aa^{-1}$ using the product rule) we estimate
\begin{align*}
  \Re\dualprod[\big]{a^{-1}\varphi}{\varphi}_{\sobh^{1}_{\rho}}
  &= \Re\dualprod[\big]{\tfrac{\dd}{\dx[t]}(a^{-1}\varphi)}{\varphi'}_{\leb^2_{\rho}}\\
  &= \Re\dualprod[\big]{(a^{-1})_{t}\varphi + a^{-1}\varphi'}{\varphi'}_{\leb^2_{\rho}}\\
  &= \Re\dualprod[\big]{- a^{-1}a_{t}a^{-1} \varphi + a^{-1}\varphi'}{\varphi'}_{\leb^2_{\rho}}\\
  &\geq - \norm{a^{-1}}_{\mathcal{L}(\leb^2(\Omega)^{d\times d})}^{2}\norm{a_{t}}_{\mathcal{L}(\leb^2(\Omega)^{d\times d})}\norm{\varphi}_{\leb^2_{\rho}}\norm{\varphi'}_{\leb^2_{\rho}} + \tfrac{c}{\norm{a}^{2}}\norm{\varphi'}_{\leb^2_{\rho}}^{2}\\
  &\geq \Big(\tfrac{c}{\norm{a}^{2}} - \tfrac{\norm{a_{t}}}{\rho c^{2}}\Big)\norm{\varphi}_{\sobh^{1}_{\rho}}^{2},
\end{align*}
where we used that $\varphi = I_{\rho}\varphi'$ with $\norm{I_{\rho}} = \tfrac{1}{\rho}$ and $\Re a^{-1} \geq \tfrac{c}{\norm{a}^{2}}$ by \cite[prop.~6.2.3]{STW2022}. This estimate can always be satisfied for large enough $\rho$. The remaining assumptions for \cref{th:Picard} and \cref{th:PicardHigh} can be checked easily, cf.~\cite[ch.~16.1]{STW2022}. Transitioning to an initial value problem with a continuation $Z$ of a prehistory $\Phi\in \sobw^{1,\infty}(-h,0;H)$ satisfying assumption \ref{ass:B} and a function $G$ satisfying asssumption \ref{ass:A}, well-posedness of the problem
\begin{equation*}
  \left[\partial_{t,\rho}\begin{pmatrix} 1 & 0 \\ 0 & 0 \end{pmatrix}
  + \begin{pmatrix} 0 & 0 \\ 0 & a^{-1} \end{pmatrix}
  + \begin{pmatrix} 0 & \mathring{\operatorname{div}} \\ \operatorname{grad} & 0 \end{pmatrix}\right]
  \begin{pmatrix} \theta \\ q \end{pmatrix}
  = \begin{pmatrix} f + G(\argdot, (I_{\rho}\theta + Z)_{(\argdot)}) \\ 0 \end{pmatrix}
\end{equation*}
follows from \cref{th:localExistIns} provided that $f \in \sobh^{1}_{\rho}\big(\mathbb{R};\leb^{2}(\Omega)\big)$ and $f(0) = -G(0,\Phi)$.
\end{example}

\begin{example}[Maxwell's equations]
Let $\Omega \subseteq \mathbb{R}^{3}$ be open and let $\epsilon,\mu,\sigma \colon \mathbb{R} \!\times \!\Omega \to \mathbb{R}^{3\times 3}$ denoting dielectricity, magnetic permeability and electric conductivity respectively satisfy
\begin{enumerate}[label = (\roman*), leftmargin = 5ex]
  \item $\epsilon,\mu,\sigma$ are bounded and measurable.
  \item $\epsilon_{t}$, $\epsilon_{tt}$, as well as $\mu_{t}$, $\mu_{tt}$ and $\sigma_{t}$ exist pointwise almost everywhere and are bounded and measurable functions.
  \item Assume $\epsilon(t,x)^{\intercal} = \epsilon (t,x)$ and $\mu(t,x)^{\intercal} = \mu(t,x)$ for a.a.\ $(t,x)\in \mathbb{R} \!\times \!\Omega$.
  \item There exist $c, \rho_{1}>0$ such that for all $\rho \geq \rho_{1}$
        \begin{equation*}
          \qquad\mu(t,x)\geq c \quad \text{and} \quad \rho \epsilon(t,x) + \epsilon_{t}(t,x) + \Re \sigma(t,x) \geq c \quad \text{for a.a.}\ (t,x)\in \mathbb{R} \!\times \!\Omega.
        \end{equation*}
  \item\label{it:Condition} There exist $d, \rho_{2}>0$ such that for all $\rho \geq \rho_{2}$ and a.a.\ $(t,x)\in \mathbb{R} \!\times \!\Omega$
        \begin{equation*}
          \qquad \rho\mu(t,x) + \mu_{t}(t,x)\geq d \quad \text{and} \quad \rho \epsilon(t,x) + 2\epsilon_{t}(t,x) + \Re \sigma(t,x) \geq d.
        \end{equation*}
\end{enumerate}
Then the nonautonomous Maxwell equations are given as
\begin{equation*}
  \left[\partial_{t}\begin{pmatrix} \epsilon & 0 \\ 0 & \mu \end{pmatrix}
  + \begin{pmatrix} \sigma & 0 \\ 0 & 0 \end{pmatrix}
  + \begin{pmatrix} 0 & -\operatorname{curl} \\ \mathring{\operatorname{curl}} & 0 \end{pmatrix}\right]
  \begin{pmatrix} E \\ H \end{pmatrix}
  = \begin{pmatrix} j_{0} \\ 0 \end{pmatrix}
\end{equation*}
for some external current $j_{0}\colon \mathbb{R} \!\times\!\Omega \to \mathbb{R}^{3}$. Note that the two divergence conditions usually complementing Maxwell's equations can be realized via suitable initial conditions, cf.~\cite[rmk.~6.2.9]{STW2022}. Well-posedness of the problem in $\leb^2_{\rho}\big(\mathbb{R};\leb^2(\Omega)^{3} \!\times\! \leb^2(\Omega)^{3}\big)$ is assured by \cref{th:Picard} for $\rho \geq \rho_{1}$ without condition \ref{it:Condition}, cf.~\cite[sec.~16.1]{STW2022}.\\
For well-posedness in $\sobh^{1}_{\rho}$ and state-dependent delay in the right-hand side we require condition \ref{it:Condition}. A similar calculation as for the heat equation shows
\begin{align*}
  \Re \dualprod[\big]{(\partial_{t}\epsilon + \sigma)\varphi}{\varphi}_{\sobh^{1}_{\rho}}
  &= \Re \dualprod[\big]{\tfrac{\dd}{\dx[t]}(\epsilon_{t} + \sigma)\varphi}{\varphi'}_{\leb^2_{\rho}}
  \geq \Big(d - \tfrac{\norm{\epsilon_{tt}}+\norm{\sigma'}}{\rho}\Big)\norm{\varphi}_{\sobh^{1}_{\rho}}^{2},\\
  \Re \dualprod{\mu \varphi}{\varphi}_{\sobh^{1}_{\rho}} &\geq d\norm{\varphi}_{\sobh^{1}_{\rho}}^{2},
\end{align*}
which assures the coercivity condition for \cref{th:PicardHigh} for $\rho > \min \{\rho_{1},\rho_{2}\}$; hence, well-posedness in $\sobh^{1}_{\rho}$ is assured.
Again, after a transition to an initial value problem for continuations $Z_{1}$ and $Z_{2}$ of initial prehistories $\Phi, \Psi\in \sobw^{1,\infty}\big(-h,0;\leb^{2}(\Omega)^{3}\big)$ satisfying assumption \ref{ass:B} and functions $G_{1}$, $G_{2}$ satisfying assumption \ref{ass:A}, \cref{th:localExistIns} assures local well-posedness of the problem
\begin{multline*}
  \left[\partial_{t}\begin{pmatrix} \epsilon & 0 \\ 0 & \mu \end{pmatrix}
  + \begin{pmatrix} \sigma & 0 \\ 0 & 0 \end{pmatrix}
  + \begin{pmatrix} 0 & -\operatorname{curl} \\ \mathring{\operatorname{curl}} & 0 \end{pmatrix}\right]
  \begin{pmatrix} E \\ H \end{pmatrix}\\
  = \begin{pmatrix} f_{1} + G_{1}\big(\argdot, (I_{\rho}E + Z_{1})_{(\argdot)}, (I_{\rho}H + Z_{2})_{(\argdot)}\big) \\
    f_{2} + G_{2}\big(\argdot, (I_{\rho}E + Z_{1})_{(\argdot)}, (I_{\rho}H + Z_{2})_{(\argdot)}\big) \end{pmatrix}
\end{multline*}
provided that
\begin{itemize}[leftmargin = 4ex]
  \item $f_{1},f_{2} \in \sobh^{1}_{\rho}\big(-h,0;\leb^{2}(\Omega)^{3}\big)$,
  \item $f_{1}(0) = -G_{1}(0,\Phi,\Psi)$ and
  \item $f_{2}(0) = -G_{2}(0,\Phi,\Psi)$.
\end{itemize}
\end{example}

The last examples generalize two examples from \cite{AignerWaurick2026} to include nonautonomous coefficients in the evolutionary equation. More importantly however, the nonautonomous solution theory permits the direct inclusion of nonautonomous delay in the material law. To that end, we first ought to take a closer look at shift operators.

\subsection{Nonautonomous delay}
\label{subsec:Shift}
To study nonautonomous delay of the form $u\big(t - \tau(t)\big)$ with a delay functional $\tau \colon \mathbb{R}\to [0,h]$, the following assumptions are made:
\begin{assumptions}
  \label{ass:Tau}
  Let $\tau\colon \mathbb{R}\to [0,h]$ satisfy $\tau \in \sobw^{1,\infty}(\mathbb{R})$ and $\tau' \leq a < 1$.
\end{assumptions}
These are typical assumptions made on delay functionals in several examples, as we will see later. In particular, they imply that $t\mapsto t - \tau(t)$ is injective. The \cref{ass:Tau} allow us to accommodate a nonautonomous shift operator.
\begin{definition}[Nonautonomous shift]
  \label{def:Shift}
  Let $\tau\colon \mathbb{R}\to \mathbb{R}$, $\rho \in \mathbb{R}$ and $H$ be a Hilbert space. Let
  \begin{equation*}
    S_{\tau}\colon \leb^{2}_{\rho}(\mathbb{R};H) \to \leb^{2}_{\rho}(\mathbb{R};H),\quad
    \varphi \mapsto \big(t\mapsto \varphi(t - \tau(t))\big).
  \end{equation*}
\end{definition}
In principle, such a shift operator can be defined on any function space, but for the task at hand the setting of exponentially weighted $\leb^{2}$-spaces suffices.

\begin{lemma}
  \label{th:Shift}
  Let $\tau$ satisfy \cref{ass:Tau}. Then $S_{\tau} \in \mathcal{L}\big(\leb^{2}_{\rho}(\mathbb{R};H)\big)$ with
  \begin{equation*}
    \norm{S_{\tau}} \leq \frac{\e^{-\rho \inf (\tau)}}{\sqrt{1-a}}.
  \end{equation*}
\end{lemma}

\begin{proof}
  Let $\varphi \in \leb^{2}_{\rho}(\mathbb{R};H)$. We estimate
  \begin{align*}
    \norm{S_{\tau}\varphi}_{\leb^{2}(\mathbb{R};H)}^{2}
    &= \medint\int_{-\infty}^{\infty} \norm[\big]{\varphi\big(t-\tau(t)\big)}_{H}^{2}\e^{-2\rho t}\dx[t]\\
    &= \medint\int_{-\infty}^{\infty} \e^{-2\rho \tau(t)}\norm[\big]{\varphi\big(t-\tau(t)\big)}_{H}^{2}\e^{-2\rho (t-\tau(t))}\dx[t]\\
    &\leq \e^{-2\rho \inf(\tau)} \medint\int_{-\infty}^{\infty}\norm{\varphi(s)}_{H}^{2}\e^{-2\rho s} \tfrac{1}{1-\tau'\big((1-\tau)^{-1}(s)\big)}\dx[s].
  \end{align*}
  In the last line the substitution $s = t - \tau(t)$ was used, which gives
  \begin{equation*}
    \dx[s] = \big[1 - \tau'(t)\big]\dx[t] = \big[1 - \tau'\big((1-\tau)^{-1}(s)\big)\big]\dx[t].
  \end{equation*}
  Since $\tau' \leq a < 1$, we conclude
  \begin{equation*}
    \norm{S_{\tau}\varphi}_{\leb^{2}_{\rho}(\mathbb{R};H)} \leq \tfrac{\e^{-\rho \inf(\tau)}}{\sqrt{1-a}} \norm{\varphi}_{\leb^{2}_{\rho}(\mathbb{R};H)}.\qedhere
  \end{equation*}
\end{proof}

To verify coercivity estimates, we will make use of the following observations:

\begin{lemma}
  \label{th:Derivative}
  Let $\rho >0$ and $\varphi \in \sobh^{1}_{\rho}(\mathbb{R};H)$. Then,
  \begin{equation*}
    \Re \dualprod{\varphi'}{\varphi}_{\leb^{2}_{\rho}(\mathbb{R};H)} = \rho \norm{\varphi}_{\leb^{2}_{\rho}(\mathbb{R};H)}^{2}.
  \end{equation*}
\end{lemma}

\begin{proof}
  A direct calculation shows
  \begin{align*}
    2\Re \dualprod{\varphi'}{\varphi}_{\leb^{2}_{\rho}(\mathbb{R};H)}
    &= 2\Re \medint\int_{-\infty}^{\infty} \dualprod{\varphi'(t)}{\varphi(t)}_{H}\e^{-2\rho t}\dx[t]\\
    &= \medint\int_{-\infty}^{\infty} \big(\tfrac{\dd}{\dx[t]}\norm{\varphi(t)}_{H}^{2}\big)\e^{-2\rho t}\dx[t]\\
    &= 2\rho\medint\int_{-\infty}^{\infty} \norm{\varphi(t)}_{H}^{2}\e^{-2\rho t}\dx[t]
    = 2\rho \norm{\varphi}_{\leb^{2}_{\rho}(\mathbb{R};H)}^{2}.\qedhere
  \end{align*}
\end{proof}

This estimate can be generalized in the following way:
\begin{lemma}
  \label{th:DerivativeT}
  Let $\Omega \subseteq \mathbb{R}^{d}$ be bounded and open and let $T\colon \mathbb{R} \!\times \!\Omega \to \mathbb{R}^{d\times d}$
  \begin{itemize}[leftmargin = 4ex]
    \item be a measurable, bounded function of symmetric matrices,
    \item satisfy $T \geq c >0$ and
    \item be Lipschitz-continuous w.r.t.\ the temporal variables uniformly in space.
  \end{itemize}
  Then for $\rho >0$ and $\varphi \in \sobh^{1}_{\rho}\big(\mathbb{R};\leb^{2}(\Omega)^{d}\big)$,
  \begin{equation*}
    \Re \dualprod{T\varphi}{\varphi'}_{\leb^{2}_{\rho}(\mathbb{R};H)} \geq \big(\rho c - \tfrac{\norm{T_{t}}}{2}\big)\norm{\varphi}_{\leb^{2}_{\rho}(\mathbb{R};\leb^{2}(\Omega)^{d})}^{2}.
  \end{equation*}
\end{lemma}

\begin{proof}
  Using the fact that on $\leb^{2}_{\rho}(\mathbb{R};H)$ the adjoint $(\partial_{t,\rho})^{\ast} = -\partial_{t,\rho} + 2\rho$, cf.~\cite[cor.~3.2.6]{STW2022}, we can calculate for $\varphi \in \sobh^{1}_{\rho}\big(\mathbb{R};\leb^{2}(\Omega)^{d}\big)$,
  \begin{align*}
    \Re \dualprod{T\varphi}{\varphi'}_{\leb^{2}_{\rho}}
    &= \Re \dualprod[\big]{(-\partial_{t} + 2\rho)T\varphi}{\varphi}_{\leb^{2}_{\rho}}\\
    &= - \Re \dualprod{T_{t}\varphi + T\varphi'}{\varphi}_{\leb^{2}_{\rho}} + 2\rho \Re \dualprod{T\varphi}{\varphi}_{\leb^{2}_{\rho}}.
  \end{align*}
  The observation that symmetry of $T$ yields
  \begin{equation*}
    \Re \dualprod{T\varphi'}{\varphi}_{\leb^{2}_{\rho}} = \Re \dualprod{\varphi'}{T\varphi}_{\leb^{2}_{\rho}} = \Re \dualprod{T\varphi}{\varphi'}_{\leb^{2}_{\rho}},
  \end{equation*}
  allows the estimate
  \begin{align*}
    2\Re \dualprod{T\varphi}{\varphi'}_{\leb^{2}_{\rho}}
    &= \Re \dualprod[\big]{(2\rho - T_{t})\varphi}{\varphi}_{\leb^{2}_{\rho}}\\
    &\geq \big(2\rho c - \norm{T_{t}}\big)\norm{\varphi}_{\leb^{2}_{\rho}}^{2}.\qedhere
  \end{align*}
\end{proof}
To show the viability of the theory, we will apply it to examples discussed in \cite{Pignotti2016,Pignotti2022,Pignotti2023} on perturbed wave equations and in the process substantially generalize the well-posedness results there.

\subsubsection{Nonautonomous delay in the velocity}\label{subsubsec:Wave1}\phantom{.}\\
In \cite{Pignotti2016}, the motivating example is the system
\begin{equation}
  \label{eq:PignottiWave1}
  \left\{
  \begin{aligned}
    u_{tt}(x,t) - \Delta u(x,t) &+ b_{1}(t)u_{t}(x,t) + b_{2}(t)u_{t}(x, t\!-\!\tau) &\phantom{.}\\
    \qquad&= -\abs{u}^{p}(x,t)u(x,t) &(x,t)\in \Omega \!\times \!(0,\infty),\\
  u(x,t) &=0 &(x,t)\in\partial\Omega \!\times \!(0,\infty),\\
  u(x,0) &= u_{0}(x) &x\in\Omega,\\
  u_{t}(x,0) &= u_{1}(x) &x\in\Omega,
  \end{aligned}
  \right.
\end{equation}
where $\Delta$ denotes the Dirichlet--Laplacian on a smooth, bounded domain $\Omega\subseteq \mathbb{R}^{d}$ with initial data $(u_{0},u_{1}) \in \mathring{\sobh}^{1}(\Omega) \!\times\! \leb^2(\Omega)$ and $b_{1},b_{2}\in \leb^{\infty}(0,\infty)$ are nonnegative functions such that $b_{1}(t)b_{2}(t) =0$ for all $t>0$.\\
We generalize the scenario and only demand
\begin{itemize}[leftmargin = 4ex]
  \item $\Omega \subseteq \mathbb{R}^{d}$ to be an open, bounded Lipschitz domain.
  \item that $T\colon \mathbb{R} \!\times \!\Omega \to \mathbb{R}^{d\times d}$ is a measurable, bounded function of symmetric matrices satisfying $T \geq c >0$. Additionally, assume that $T_{t}$ and $T_{tt}$ exists almost everywhere as bounded, measurable functions.
  \item that $\tau$ satisfies \cref{ass:Tau}.
\end{itemize}
We pose the wave equation as a system in its Dirac formulation, i.e.\
\begin{equation*}
  \bigg[\partial_{t}\underbrace{\begin{pmatrix} T^{-1} & 0 \\ 0 & 1 \end{pmatrix}}_{\eqcolon M}
  + \underbrace{\begin{pmatrix} 0 & 0 \\ 0 & b_{1} + b_{2}S_{\tau} \end{pmatrix}}_{\eqcolon N}
  - \underbrace{\begin{pmatrix} 0 & \mathring{\operatorname{grad}} \\ \operatorname{div} & 0 \end{pmatrix}}_{\eqcolon -A}\bigg]
  \begin{pmatrix} u \\ v \end{pmatrix}
  = \begin{pmatrix} f \\ g \end{pmatrix},
\end{equation*}
where $S_{\tau}$ denotes the shift operator from \cref{def:Shift}. Note that the \cref{ass:Tau} are more general than in \cite{Pignotti2016}, where $\tau > 0$ was constant. However, we forgo the challenge the right-hand side in \cref{eq:PignottiWave1} poses and consider $(f,g)\in \leb^{2}_{\rho}\big(\mathbb{R};\leb^{2}(\Omega)^{d}\big) \!\times\! \leb^{2}_{\rho}\big(\mathbb{R};\leb^{2}(\Omega)\big)$. To apply well-posedness \cref{th:localExistIns}, we first have to verify its assumptions:
\begin{itemize}[leftmargin = 4ex]
  \item $M$ is selfadjoint by definition.
  \item The commutator condition in \cref{th:Picard} is satisfied, because $T_{t}$ exists almost everywhere and is a bounded, measurable function.
  \item $A$ is skew-selfadjoint by definition (note that $\mathring{\operatorname{grad}}\coloneq \operatorname{div}^{\ast}$ is the gradient with Dirichlet boundary condition\footnote{The adjoint in the definition is taken as the adjoint of $\operatorname{div}\colon \leb^2(\Omega)^{d}\supseteq \operatorname{dom}(\operatorname{div})\to \leb^2(\Omega)$.}).
  \item For well-posedness in $\leb^2_{\rho}\big(\mathbb{R};\leb^2(\Omega)^{d} \!\times\! \leb^2(\Omega)\big)$ two coercivity estimates need to be verified. First, we estimate
        \begin{align*}
          \Re\dualprod[\big]{\tfrac{\dd}{\dx[t]}\big(T^{-1}\varphi\big)}{\varphi}_{\leb^2_{\rho}}
          &= \Re \dualprod[\big]{(T^{-1})_{t}\varphi}{\varphi}_{\leb^{2}_{\rho}} + \Re \dualprod{T^{-1}\varphi'}{\varphi}_{\leb^{2}_{\rho}}\\
          &\geq - \norm{(T^{-1})_{t}}\norm{\varphi}_{\leb^{2}_{\rho}}^{2} + \big(\rho \tfrac{c}{\norm{T^{-1}}^{2}} - \tfrac{\norm{(T^{-1})_{t}}}{2}\big)\norm{\varphi}_{\leb^{2}_{\rho}}^{2}\\
          &\geq \big(\rho \tfrac{c}{\norm{T^{-1}}^{2}} - \tfrac{3\norm{T^{-1}}^{2}\norm{T_{t}}}{2}\big)\norm{\varphi}_{\leb^{2}_{\rho}}^{2},
        \end{align*}
        where \cref{th:DerivativeT} was used for $T^{-1}$, noting that $\Re T^{-1} \geq \tfrac{c}{\norm{T}^{2}}$, cf.~\cite[prop.~6.2.3]{STW2022}.
        Second, we estimate using \cref{th:Shift} and \cref{th:Derivative},
        \begin{align*}
          \quad\Re \dualprod[\big]{(\partial_{t} \!+ \!(b_{1} \!+ \!b_{2}S_{\tau}))\varphi}{\varphi}_{\leb^2_{\rho}}
          &= \Re \dualprod{\varphi'}{\varphi}_{\leb^2_{\rho}} \!+ \!\Re\dualprod[\big]{(b_{1} \!+ \!b_{2}S_{\tau})\varphi}{\varphi}_{\leb^2_{\rho}}\\
          &= \rho \norm{\varphi}_{\leb^{2}_{\rho}}^{2} \!+ \!\Re \dualprod{b_{1}\varphi}{\varphi}_{\leb^{2}_{\rho}} \!+ \!\Re \dualprod{b_{2}S_{\tau}\varphi}{\varphi}_{\leb^{2}_{\rho}}\\
          &\geq \rho \norm{\varphi}_{\leb^{2}_{\rho}}^{2} \!+ \!\inf b_{1} \norm{\varphi}_{\leb^{2}_{\rho}}^{2} \!- \!\norm{b_{2}}_{\infty}\norm{S_{\tau}\varphi}_{\leb^{2}_{\rho}}\norm{\varphi}_{\leb^{2}_{\rho}}\\
          &\geq \Big(\rho \!+ \!\inf b_{1} \!- \!\norm{b_{2}}_{\infty}\tfrac{\e^{-\rho \inf (\tau)}}{\sqrt{1-a}}\Big)\norm{\varphi}_{\leb^{2}_{\rho}}^{2}.
        \end{align*}
        Hence, coercivity in $\leb^2_{\rho}\big(\mathbb{R};\leb^{2}(\Omega)^{d} \!\times\! \leb^{2}(\Omega)\big)$ is assured for $\rho$ large enough, because $\inf (\tau)\geq 0$.
\end{itemize}
These facts thus ascertain well-posedness of the equation
\begin{equation*}
  \bigg[\partial_{t}\begin{pmatrix} T^{-1} & 0 \\ 0 & 1 \end{pmatrix}
  +\begin{pmatrix} 0 & 0 \\ 0 & b_{1} + b_{2}S_{\tau} \end{pmatrix}
  - \begin{pmatrix} 0 & \mathring{\operatorname{grad}} \\ \operatorname{div} & 0 \end{pmatrix}\bigg]
\begin{pmatrix} u \\ v \end{pmatrix}
= \begin{pmatrix} f \\ g \end{pmatrix}
\end{equation*}
for $f \in \leb^2_{\rho}\big(\mathbb{R};\leb^2(\Omega)^{d}\big)$ and $g \in \leb^2_{\rho}\big(\mathbb{R};\leb^2(\Omega)\big)$ appealing to \cref{th:Picard}. To accommodate a state-dependent right-hand side, the assumptions of \cref{th:localExistIns} need to be verified. This begs additional assumptions on the coefficients. More specifically, $b_{1},b_{2} \in \sobw^{1,\infty}(\mathbb{R})$.\\
To verify the assumptions of \cref{th:PicardHigh}, a coercivity estimate in $\sobh^{1}_{\rho}\big(\mathbb{R};\leb^2(\Omega)^{d} \!\times\! \leb^2(\Omega)\big)$ needs to be established. Since $T_{t}$ and $T_{tt}$ exist, so do the first and second derivatives of $T^{-1}$. For $\sobh^{1}_{\rho}$-coercivity of the first component we calculate
\begin{align*}
  \Re \dualprod[\big]{\tfrac{\dd}{\dx[t]}(T^{-1}\varphi)}{\varphi}_{\sobh^{1}_{\rho}}
  &= \Re \dualprod[\big]{\tfrac{\dd}{\dx[t]}(T^{-1}\varphi)}{(-\partial_{t} + 2\rho)\varphi'}_{\leb^{2}_{\rho}}\\
  &= \Re \dualprod[\big]{(T^{-1})_{t}\varphi + T^{-1}\varphi'}{(-\partial_{t} + 2\rho)\varphi'}_{\leb^{2}_{\rho}}\\
  &= - \Re \dualprod[\big]{(-\partial_{t} + 2\rho)\big[(T^{-1})_{t}\varphi + T^{-1}\varphi']}{\varphi'}_{\leb^{2}_{\rho}} \\
  &\qquad+ \Re 2\rho \dualprod[\big]{(T^{-1})_{t}\varphi + T^{-1}\varphi'}{\varphi'}_{\leb^{2}_{\rho}}\\
  &= \Re \dualprod[\big]{(T^{-1})_{tt}\varphi + 2(T^{-1})_{t}\varphi' + T^{-1}\varphi''}{\varphi'}_{\leb^{2}_{\rho}}\\
  &\geq -\norm{(T^{-1})_{tt}}\norm{\varphi}_{\leb^{2}_{\rho}}\norm{\varphi'}_{\leb^{2}_{\rho}} - 2\norm{(T^{-1})_{t}}\norm{\varphi'}_{\leb^{2}_{\rho}}^{2} \\
  &\qquad+ \big(\rho\tfrac{c}{\norm{T^{-1}}^{2}} - \tfrac{\norm{T_{t}}}{2}\big)\norm{\varphi'}_{\leb^{2}_{\rho}}\\
  &\geq \big[\rho \tfrac{c}{\norm{T^{-1}}^{2}} - \big(2 \norm{(T^{-1})_{t}}+\tfrac{1}{2}\norm{T_{t}}\big) - \tfrac{1}{\rho}\norm{(T^{-1})_{tt}}\big]\norm{\varphi}_{\sobh^{1}_{\rho}},
\end{align*}
where we again used \cref{th:DerivativeT} for $T^{-1}$. For the second component we calculate
\begin{align*}
  &\Re \dualprod[\big]{\varphi' \!+ \!(b_{1} \!+ \!b_{2}S_{\tau})\varphi}{\varphi}_{\sobh^{1}_{\rho}(\mathbb{R};\leb^2(\Omega))}\\
  &\quad= \Re \dualprod[\big]{\varphi'' \!+ \!(b_{1}' \!+ \!b_{2}'S_{\tau})\varphi \!+ \!\big(b_{1} \!+ \!b_{2}(1 \!- \!\tau')S_{\tau}\big)\varphi'}{\varphi'}_{\leb^2_{\rho}(\mathbb{R};\leb^2(\Omega))}\\
  &\quad\geq \rho \norm{\varphi'}^{2} \!- \!\big[\norm{b_{1}'}_{\infty}\norm{\varphi} \!+ \!\norm{b_{2}'}_{\infty}\norm{S_{\tau}\varphi}\big]\norm{\varphi'}
  \!+ \!\big[\inf b_{1} \!- \!\norm{b_{2}(1 \!- \!\tau')}_{\infty} \tfrac{\e^{-\rho \inf (\tau)}}{1-a}\big]\norm{\varphi'}^{2}\\
  &\quad\geq \big[\rho \!+ \!\inf (b_{1}) \!- \!\norm{b_{2}(1 \!- \!\tau')}_{\infty} \tfrac{\e^{-\rho \inf (\tau)}}{\sqrt{1-a}} - \tfrac{1}{\rho}\norm{b_{1}'}_{\infty} \!- \!\tfrac{1}{\rho}\norm{b_{2}'}_{\infty}\tfrac{\e^{-\rho \inf (\tau)}}{\sqrt{1-a}}\big]\norm{\varphi'}^{2}.
\end{align*}
Hence, coercivity in $\sobh^{1}_{\rho}$ is established for $\rho$ large enough. Consequently, the system
\begin{equation*}
  \bigg[\partial_{t}\begin{pmatrix} T^{-1} & 0 \\ 0 & 1 \end{pmatrix}
  +\begin{pmatrix} 0 & 0 \\ 0 & b_{1} + b_{2}S_{\tau} \end{pmatrix}
  - \begin{pmatrix} 0 & \mathring{\operatorname{grad}} \\ \operatorname{div} & 0 \end{pmatrix}\bigg]
\begin{pmatrix} u \\ v \end{pmatrix}
= \begin{pmatrix} 0 \\ F(\argdot, u_{(t)}, v_{(t)}) \end{pmatrix}
\end{equation*}
is well-posed after transitioning to a suitable formulation for initial values according to \cref{th:localExistIns}. How to obtain a suitable continuation of a given prehistory is expanded on in \cref{app:InitialValues}.

\subsubsection{Nonautonomous delay paired with nonlocal operators in space}\phantom{.}\\
In the articles \cite{Pignotti2022,Pignotti2023}, problem \eqref{eq:PignottiWave1} was more generally framed as
\begin{equation}
  \label{eq:PignottiWave2}
  \left\{
    \begin{aligned}
      u_{tt}(t) + Pu(t) + B_{1}B_{1}^{\ast}u_{t}(t) + k(t)B_{2}B_{2}^{\ast}u_{t}\big(t -\tau(t)\big) &= F\big(u(t)\big), \quad\;\; t > 0,\\
      u(0) &= u_{0}, \\
      u_{t}(0) &= u_{1},\\
      B^{\ast}u_{t}(s) &= g(s), \quad s \in [-h, 0],
    \end{aligned}
  \right.
\end{equation}
where
\begin{itemize}[leftmargin = 4ex]
  \item $P$ is a positive, selfadjoint, coercive operator on a Hilbert space $H$.
  \item $B_{1},B_{2} \in \mathcal{L}(V,H)$, where $V = \operatorname{dom}(P^{\nicefrac{1}{2}})$.
\end{itemize}
Problem \eqref{eq:PignottiWave2} can be formulated as an evolutionary equation in the form
\begin{equation}
  \label{eq:Wave2EvolEq}
  \bigg[\partial_{t}\underbrace{\begin{pmatrix} 1 & 0 \\ 0 & 1 \end{pmatrix}}_{= M}
    + \underbrace{\begin{pmatrix} 0 & 0 \\ 0 & B + C S_{\tau} \end{pmatrix}}_{= N}
    + \underbrace{\begin{pmatrix} 0 & -P^{\frac{1}{2}} \\ P^{\frac{1}{2}} & 0 \end{pmatrix}}_{= A}\bigg]
  \begin{pmatrix} u \\ v \end{pmatrix}
  = \begin{pmatrix} 0 \\ F(\argdot, u_{(\argdot)}, v_{(\argdot)}) \end{pmatrix}
\end{equation}
via the substitution $v = u_{t}$. We generalize the scenario and assume that
\begin{itemize}[leftmargin = 4ex]
  \item $B,C\colon \mathbb{R}\to \mathcal{L}(H)$ are G\^ateaux-differentiable almost everywhere.
  \item $B,B',C,C' \in \leb^{\infty}\big(\mathbb{R};\mathcal{L}(H)\big)$.
  \item $\tau$ adheres to the \cref{ass:Tau}.
\end{itemize}
Again, we argue the coercivity estimates required for \cref{th:Picard} and \cref{th:PicardHigh} for $k=1$: For the $\leb^2_{\rho}$-case we see, again utilizing \cref{th:Derivative} and \cref{th:Shift},
\begin{align*}
  &\Re \dualprod[\big]{\varphi' + (B + CS_{\tau})\varphi}{\varphi}_{\leb^2_{\rho}(\mathbb{R};H)}
  = \rho \norm{\varphi}_{\leb^{2}(\mathbb{R};H)}^{2} + \Re\dualprod[\big]{(B + CS_{\tau})\varphi}{\varphi}_{\leb^2_{\rho}(\mathbb{R};H)}\\
  &\quad\geq \rho\norm{\varphi}_{\leb^2_{\rho}}^{2} - \norm{B}_{\leb^{\infty}(\mathbb{R};\mathcal{L}(H))}\norm{\varphi}_{\leb^2_{\rho}}^{2} - \norm{C}_{\leb^{\infty}(\mathbb{R};\mathcal{L}(H))}\norm{S_{\tau}\varphi}_{\leb^2_{\rho}}\norm{\varphi}_{\leb^2_{\rho}}\\
  &\quad\geq \big[\rho - \norm{B}_{\leb^{\infty}(\mathbb{R};\mathcal{L}(H))}- \norm{C}_{\leb^{\infty}(\mathbb{R};\mathcal{L}(H))}\tfrac{\e^{-\rho \inf (\tau)}}{\sqrt{1-a}}\big]\norm{\varphi}_{\leb^2_{\rho}}^{2}.
\end{align*}
For coercivity in $\sobh^{1}_{\rho}$ we similarly calculate
\begin{align*}
  &\Re \dualprod[\big]{\varphi' + (B + CS_{\tau})\varphi}{\varphi}_{\sobh^{1}_{\rho}}\\
  &\quad= \Re \dualprod[\big]{\varphi'' + \big(B + C(1-\tau')S_{\tau}\big)\varphi' + \big(B' + C'S_{\tau}\big)\varphi}{\varphi'}_{\leb^2_{\rho}}\\
  &\quad\geq \rho \norm{\varphi'}_{\leb^{2}_{\rho}}^{2}
    - \big[\norm{B}_{\leb^{\infty}(\mathbb{R};\mathcal{L}(H))} - \norm{1-\tau'}_{\leb^{\infty}(\mathbb{R})}\norm{C}_{\leb^{\infty}(\mathbb{R};\mathcal{L}(H))}\big]\tfrac{\e^{-\rho \inf(\tau)}}{\sqrt{1-a}}\norm{\varphi'}_{\leb^{2}_{\rho}}^{2} \\
  &\qquad - \big[\norm{B'}_{\leb^{\infty}(\mathbb{R};\mathcal{L}(H))} + \norm{C'}_{\leb^{\infty}(\mathbb{R};\mathcal{L}(H))}\tfrac{\e^{-\rho \inf(\tau)}}{\sqrt{1-a}}\big]\norm{\varphi'}_{L^{2}_{\rho}}^{2}.
\end{align*}
Hence, problem \eqref{eq:Wave2EvolEq} is locally well-posed after transitioning to an adequate formulation for initial values appealing to \cref{th:localExistIns}.\\
In \cite{Pignotti2023}, some difficulty comes from the assumption on the integrability of $k \in \leb^{1}_{\mathrm{loc}}$ in \cref{eq:PignottiWave2}. We chose to forgo this difficulty here and instead permit $B,C$ to be time-dependent and more importantly, the forcing term to admit state-dependent delay. From this perspective, our results substantially generalize the results regarding well-posedness.

\subsubsection{Nonautonomous delay paired with nonlocal operators in time}\phantom{.}\\
In the articles \cite{Pignotti2021,Pignotti2017}, problem \eqref{eq:PignottiWave1} was generalized to contain a convolution,
\begin{equation}
  \label{eq:PignottiWave3}
  \left\{
  \begin{aligned}
    u_{tt}(x,t) - \Delta u(x,t) + \medint\int_{0}^{\infty}\mu(s)\Delta u(x,t-s)& \dx[s] \phantom{.}& \\
    + ku_{t}(x, t-\tau) &= 0 &(x,t)\in \Omega \!\times \!(0,\infty),\\
    u(x,t) &=0 &(x,t)\in \partial\Omega \!\times \!(0,\infty),\\
    u(x,0) &= u_{0}(x) &x\in\Omega,\\
    u_{t}(x,0) &= u_{1}(x) &x\in\Omega,
  \end{aligned}
  \right.
\end{equation}
where $k \in \mathbb{R}$ and the memory kernel $\mu \colon \mathbb{R}_{\geq 0}\to \mathbb{R}_{\geq 0}$ is a locally absolutely continuous function satisfying
\begin{itemize}[leftmargin = 4ex]
  \item $\mu(0) = \mu_{0}>0$,
  \item $\medint\int_{0}^{\infty}\mu(s)\dx[s] = \overline{\mu} < 1$ and
  \item $\mu'(t)\leq \alpha \mu(t)$ for some $\alpha > 0$.
\end{itemize}
Using the substitutions $v = u_{t}$ and $ q = (1-\mu\ast) \operatorname{grad}u$, \cref{eq:PignottiWave3} can be formulated as an evolutionary equation
\begin{equation}
  \label{eq:Wave3EvolEq}
  \bigg[\partial_{t}\underbrace{\begin{pmatrix} 1 & 0 \\ 0 & (1-\mu \ast)^{-1} \end{pmatrix}}_{= M}
  + \underbrace{\begin{pmatrix} k S_{\tau} & 0 \\ 0 & 0 \end{pmatrix}}_{= N}
  - \underbrace{\begin{pmatrix} 0 & \operatorname{div} \\ \mathring{\operatorname{grad}} & 0 \end{pmatrix}}_{= -A}\bigg]
\begin{pmatrix} v \\ q \end{pmatrix}
= \begin{pmatrix} 0 \\ 0 \end{pmatrix}
\end{equation}
posed on $\leb^{2}_{\rho}\big(\mathbb{R};\leb^{2}(\Omega)\!\times \!\leb^{2}(\Omega)^{d}\big)$. We generalize the setting and only assume
\begin{itemize}[leftmargin = 4ex]
  \item $\mu \colon \mathbb{R}_{\geq 0}\to \mathbb{R}_{\geq 0}$ is a locally absolutely continuous function,
  \item $\medint\int_{0}^{\infty}\mu(s)\dx[s] = \overline{\mu} < 1$ and
  \item $\medint\int_{0}^{\infty}\abs{\mu'(s)}\dx[s] < \infty$.
\end{itemize}
We refer to \cref{app:Convolutions} to see that $(1-\mu\ast)^{-1} \in \mathcal{L}\big(\leb^{2}_{\rho}(\mathbb{R};\leb^{2}(\Omega)^{d})\big)$ and that $(1 - \mu\ast)^{-1}f \in \sobh^{1}_{\rho}(\mathbb{R};\leb^{2}(\Omega)^{d})$ for $f \in \leb^{2}(\mathbb{R};\leb^{2}(\Omega)^{d})$. For well-posedness of \cref{eq:Wave3EvolEq}, yet again we verify coercivity conditions:
\begin{enumerate}[label = \arabic*., leftmargin = 4ex]
  \item For the first component of the state, observe
        \begin{equation*}
          \dualprod{\partial_{t}\varphi + kS_{\tau}\varphi}{\varphi}_{\leb^{2}_{\rho}}
          = \dualprod{\varphi'}{\varphi}_{\leb^{2}_{\rho}} + k\dualprod{S_{\tau}\varphi}{\varphi}_{\leb^{2}_{\rho}}
          \geq \rho \norm{\varphi}_{\leb^{2}_{\rho}}^{2} - k\tfrac{\e^{-\rho \inf(\tau)}}{\sqrt{1-a}}\norm{\varphi}_{\leb^{2}_{\rho}}^{2},
        \end{equation*}
        appealing to \cref{th:Shift} and \cref{th:Derivative}.
  \item For the second component of the state, recall
        \begin{equation*}
          \dualprod[big]{\partial_{t}(1-\mu\ast)^{-1}\varphi}{\varphi}_{\leb^{2}_{\rho}}
          \geq (\rho - c)\norm{\varphi}_{\leb^{2}_{\rho}}^{2}
        \end{equation*}
        for some $c>0$ appealing to \cref{th:Convolution}.
\end{enumerate}
Consequently, both estimates are satisfied for $\rho>0$ large enough. If one wishes to supplant the right-hand side of \cref{eq:Wave3EvolEq} with a state-dependent inhomogeneity
\begin{equation*}
  \begin{pmatrix} F(t,u_{(t)}) \\ G(t,u_{(t)}) \end{pmatrix},
\end{equation*}
then addidionally the assumptions of \cref{th:PicardHigh} need to be satisfied for $k=1$. We similarly calculate
\begin{enumerate}[label = \arabic*., leftmargin = 4ex]
  \item for the first component of the state,
        \begin{align*}
          \dualprod{\partial_{t}\varphi + kS_{\tau}\varphi}{\varphi}_{\sobh^{1}_{\rho}}
          &= \dualprod{\varphi''}{\varphi'}_{\leb^{2}_{\rho}} + k\dualprod{\partial_{t}S_{\tau}\varphi}{\varphi'}_{\leb^{2}_{\rho}}\\
          &= \rho \norm{\varphi'}_{\leb^{2}_{\rho}}^{2} - k\dualprod[\big]{(1-\tau')S_{\tau}\varphi'}{\varphi'}_{\leb^{2}_{\rho}}\\
          &\geq \rho \norm{\varphi}_{\sobh^{1}_{\rho}}^{2} - k\norm{1-\tau'}_{\infty}\tfrac{\e^{-\rho \inf(\tau)}}{\sqrt{1-a}}\norm{\varphi}_{\sobh^{1}_{\rho}}^{2},
        \end{align*}
        appealing to \cref{th:Shift} and \cref{th:Derivative} again.
  \item for the second component of the state, making use of the notation of \cref{app:Convolutions}, we calculate
        \begin{align*}
          \dualprod[\big]{\partial_{t}(1-\mu\ast)^{-1}\varphi}{\varphi}_{\sobh^{1}_{\rho}}
          &= \dualprod[\big]{(\iu \mathrm{m} + \rho)^{2}\flt\big((1- \mu\ast)^{-1}\varphi\big)}{(\iu \mathrm{m} + \rho)\flt(\varphi)}_{\leb^{2}}.\\
          \intertext{We can supplant the expression calculated in \cref{th:FLTMu} for $\flt\big((1-\mu\ast)^{-1}\varphi\big)$ to estimate}
          \dualprod[\big]{\partial_{t}(1-\mu\ast)^{-1}\varphi}{\varphi}_{\sobh^{1}_{\rho}}
          &= \dualprod[\bigg]{(\iu \mathrm{m} + \rho)^{2}\frac{\flt(\varphi)}{(1-\sqrt{2\pi})\flt(\mu)}}{(\iu \mathrm{m} + \rho)\flt(\varphi)}_{\leb^{2}}\\
          &= \dualprod[\bigg]{\frac{\flt(\varphi'')}{(1 - \sqrt{2\pi})\flt(\mu)}}{\flt(\varphi')}_{\leb^{2}}\\
          &\geq (\rho - c)\norm{\varphi'}_{\leb^{2}_{\rho}}^{2}
        \end{align*}
        for some $c>0$ appealing to \cref{th:Convolution}.
\end{enumerate}
Hence, well-posedness is assured after passing to a suitable formulation for initial values.

\section{State-dependent boundary conditions}
\label{sec:BdyDelay}

This section aims to present a fairly comprehensive approach to evolutionary equations with state-dependent delay in boundary conditions. We start with abstract Dirichlet and Neumann conditions before tackling complex boundary conditions.

\subsection{State-dependent delay via shifting and lifting}
For some boundary conditions, the associated boundary value problem comprised of an evolutionary equation and the (inhomogeneous) boundary condition can be reduced to a single evolutionary equation
\begin{equation*}
  \bigl[\partial_{t}M(\partial_{t}) + A\bigr]u = F\bigl(\argdot, u_{(\argdot)}\bigr),
\end{equation*}
where $M$ is a material law and $A\colon H\supseteq \operatorname{dom}(A)\to H$ is a maximal accretive operator, as treated in \cite{AignerWaurick2026} or
\begin{equation*}
  \big[\partial_{t}M + N + A\big]u = F\big(\argdot, u_{(\argdot)}\big),
\end{equation*}
treated in the preceding \cref{sec:EvolEq} under the assumptions there, without a supplementary condition. This is accomplished via a shifting an lifting technique, the inspiration for which is the treatment of Dirichlet or Neumann boundary conditions in the case $A = \begin{psmallmatrix} 0 & \operatorname{div} \\ \operatorname{grad} & 0 \end{psmallmatrix}$. However, the technique is presented in a more general framework.

\subsubsection{The setup}
\label{subsec:AbstractBD}
To formulate abstract boundary conditions without explicit traces, we make use of the framework of {\em abstract boundary data spaces} from \cite{Trostorff2014}. The following brief introduction follows \cite{PicardTrostorffWaurick2016,Trostorff2014}. Let $H_{0},H_{1}$ be two Hilbert spaces and $\mathring{G}$ and $\mathring{D}$ be two densely defined operators satsifying
\begin{equation}
  \label{eq:AbstractBD}
  \begin{aligned}
  \mathring{G}&\colon H_{0}\supseteq \operatorname{dom}(\mathring{G}) \to H_{1},
  \qquad &\mathring{G} \subseteq G \coloneq - \mathring{D}^{\ast},\\
  \mathring{D}&\colon H_{1}\supseteq \operatorname{dom}(\mathring{D}) \to H_{0},
  \qquad &\mathring{D} \subseteq D \coloneq - \mathring{G}^{\ast}.
  \end{aligned}
\end{equation}
We make use of the following notation:
\begin{definition}
  Let $H$ be a Hilbert space and let $C \colon H\supseteq \operatorname{dom}(C) \to H$ be a closed, densely defined, linear operator with $0 \in \uprho(C)$. $\sobh^{1}(C)$ denotes the Hilbert space $\operatorname{dom}(C)$ equipped with the inner product
  \begin{equation*}
    \dualprod{\argdot}{\argdot}_{\sobh^{1}(C)} \coloneq \dualprod{C\argdot}{C\argdot}_{H}.
  \end{equation*}
\end{definition}
Trivially, $C \colon \sobh^{1}(C)\to H$ is a unitary operator. Coming back to the setting \eqref{eq:AbstractBD}, for $G$ and $D$ we can define $\abs{G} + \iu \coloneq \sqrt{GG^{\ast}+1}$ via the main branch of the square root and $\abs{D}+\iu$ accordingly.
\begin{theorem}[{\cite[lem.~2.1.16]{PicardMcGhee2011}}]
  Let $G \colon H_{0}\supseteq \operatorname{dom}(G) \to H_{1}$ be a closed, densely defined, linear operator. Then $G \colon \sobh^{1}(\abs{G}+\iu) \to H_{1}$ is bounded.
\end{theorem}
\begin{definition}[{\cite[sec.~5.2]{PicardTrostorffWaurick2016}}]
  Let $\mathring{G},\mathring{D},G,D$ be as in \eqref{eq:AbstractBD}.
  \begin{align*}
    \mathcal{BD}(G)&\coloneq \operatorname{dom}(\mathring{G})^{\perp_{\sobh^{1}(\abs{G}+\iu)}} = \ker (1 - DG),\\
    \mathcal{BD}(D)&\coloneq \operatorname{dom}(\mathring{D})^{\perp_{\sobh^{1}(\abs{D}+\iu)}} = \ker (1 - GD),
  \end{align*}
  are the {\em boundary data spaces}, where the orthogonal complements are taken w.r.t.\ the inner products in $\sobh^{1}(\abs{G}+\iu)$ and $\sobh^{1}(\abs{D}+\iu)$ respectively.
\end{definition}
\begin{remark}[{\cite[rem.~2.11]{Trostorff2014}}]
  According to the projection theorem,
  \begin{alignat*}{2}
    \sobh^{1}(\abs{G}+\iu) &= \sobh^{1}(\abs{\mathring{G}} + \iu) &&\oplus \mathcal{BD}(G),\\
    \sobh^{1}(\abs{D}+\iu) &= \sobh^{1}(\abs{\mathring{D}} + \iu) &&\oplus \mathcal{BD}(D).
  \end{alignat*}
  This can be interpreted as a decomposition result for elements in $\sobh^{1}(\abs{G}+\iu)$ and $\sobh^{1}(\abs{D}+\iu)$ into one part with “vanishing trace” (in $\sobh^{1}(\abs{\mathring{G}} +\iu)$ or $\sobh^{1}(\abs{\mathring{D}}+\iu)$ respectively) and one part carrying the whole information about the behaviour at the boundary.
\end{remark}
We mention in passing that the framework of abstract boundary data spaces allows the characterization of desirable structural properties of the operator $A = \begin{psmallmatrix} 0 & D \\ G & 0 \end{psmallmatrix}$ in terms of properties of the boundary data, cf.~\cite[thm.~3.1 \& cor.~3.8]{Trostorff2014} and is in that sense reminiscent of the theory of boundary triples, cf.~\cite{Behrndt2020} and \cite[rem.~5.4]{PicardTrostorffWaurick2016}.\\
An advantage of this setting is that the role of the traces is simply taken on by the orthogonal projections $\pi_{\mathcal{BD}(D)}$ and $\pi_{\mathcal{BD}(G)}$. Hence, the identity $\operatorname{id}\colon \mathcal{BD}(D) \to  \operatorname{dom}(D)$ serves as bounded right inverse to the projection $\pi_{\mathcal{BD}(D)}$ (same for $G$). This trivial fact mirrors the fact that many traces admit a bounded right inverse, which is a nontrivial fact to verify for a given trace at hand. These inverses can serve as a lift $L$ from the boundary space to the domain of the associated differential operator.\\
In addition, a lift $L$ can be modified to assure that $L\pi_{\mathcal{BD}(D)} d = d$ for one given $d \in \operatorname{dom}(D)$ is satisfied: Since $L\pi_{\mathcal{BD}(D)}d \in \operatorname{dom}(D)$ and subtraction of an element in $\operatorname{dom}(\mathring{D})$ does not change the trace,  one can define a modified lifting
\begin{equation*}
  L\colon \mathcal{BD}(D) \to \operatorname{dom}(D), \quad d \mapsto Ld + (d-L\pi_{\mathcal{BD}(D)}d),
\end{equation*}
which is still a continuous inverse of $\pi_{\mathcal{BD}(D)}$.

\subsubsection{Dirichlet and Neumann boundary conditions}
Abstract Dirichlet and Neumann boundary conditions can now be formulated via a lifting and shifting technique. This approach is already illustrated in \cite[ch.~12.3]{STW2022} and \cite[ch.~11.1]{JacobZwart2012}.\\
Let $L_{G}$ and $L_{D}$ be bounded right inverses of $\pi_{\mathcal{BD}(G)}$ and $\pi_{\mathcal{BD}(D)}$ as described above. We now want a solution of the (nonautonomous) evolution problem
\begin{equation}
  \label{eq:EvolEqDN}
    \left[\partial_{t}M + N + \begin{pmatrix} 0 & D \\ G & 0 \end{pmatrix}\right]
    \begin{pmatrix} u \\ v \end{pmatrix} = \begin{pmatrix} f_{1} \\ f_{2} \end{pmatrix}
\end{equation}
posed on $H = H_{0}\!\times\! H_{1}$ to attain a boundary condition of the form
\begin{equation*}
  \pi_{\mathcal{BD}(D)}v = d \quad \text{or} \quad \pi_{\mathcal{BD}(G)}u = g.
\end{equation*}
Structurally, this is equivalent to asking that the shifted functions $r = u - L_{G}g$ or $r = v - L_{D}d$ attain homogeneous boundary conditions, i.e.\ $u - L_{G}g \in \operatorname{dom} (\mathring{G})$ or $v - L_{D}d \in \operatorname{dom} (\mathring{D})$.

\begin{theorem}
  \label{th:ShiftingLifting}
  Let $(f_{1},f_{2}) \in \leb^2_{\rho}(\mathbb{R};H_{0}) \!\times\! \leb^2_{\rho}(\mathbb{R};H_{1})$, $g \in \leb^{2}_{\rho}\big(\mathbb{R};\mathcal{BD}(G)\big)$ and $d\in \leb^{2}_{\rho}\big(\mathbb{R};\mathcal{BD}(D)\big)$.
  \begin{enumerate}[label = (\roman*), leftmargin = 4ex]
    \item Let $(L_{G}g, 0) \in \operatorname{dom}(\partial_{t}M)$. Then $(u,v)$ is a solution of \cref{eq:EvolEqDN} with boundary condition $u - L_{G}g \in \operatorname{dom}(\mathring{G})$ if and only if
          \begin{equation*}
            \left[\partial_{t}M + N + \begin{pmatrix} 0 & D \\ \mathring{G} & 0 \end{pmatrix}\right]
            \begin{pmatrix} r \\ v \end{pmatrix}
            = \begin{pmatrix} f_{1} \\ f_{2} \end{pmatrix}
            - \left[\partial_{t}M + N + \begin{pmatrix} 0 & D \\ G & 0 \end{pmatrix}\right]
            \begin{pmatrix} L_{G}g \\ 0 \end{pmatrix}
          \end{equation*}
    \item Let $(0, L_{D}d) \in \operatorname{dom}(\partial_{t}M)$. Then $(u,v)$ is a solution of \cref{eq:EvolEqDN} with boundary condition $v - L_{D}d \in \operatorname{dom}(\mathring{D})$ if and only if $(u, v - L_{D}d)$ solves the problem
          \begin{equation*}
            \left[\partial_{t}M + N + \begin{pmatrix} 0 & \mathring{D} \\ G & 0 \end{pmatrix}\right]
            \begin{pmatrix} u \\ r \end{pmatrix}
            = \begin{pmatrix} f_{1} \\ f_{2} \end{pmatrix}
            - \left[\partial_{t}M + N + \begin{pmatrix} 0 & D \\ G & 0 \end{pmatrix}\right]
            \begin{pmatrix} 0 \\ L_{D}d \end{pmatrix}
          \end{equation*}
  \end{enumerate}
\end{theorem}
\begin{proof}
  Algebraic manipulation shows that a solution of the initial problem translates to a solution of the shifted problem and vice versa. The attainment of the boundary condition is easily observed.
\end{proof}

\subsubsection{Lipschitz condition}
\label{subsubsec:LipschitzCond}
We now turn towards state-dependent boundary conditions
\begin{equation*}
  t\mapsto g(t, u_{(t)}, v_{(t)}) \quad \text{or} \quad t\mapsto d(t,u_{(t)},v_{(t)}).
\end{equation*}
In order to apply \cref{th:localExistIns} for the shifted and lifted problem, a Lipschitz condition for the inhomogeneity is required. Since the lifting is a pointwise action and a Lipschitz continuous operator, the properties required for right-hand sides of problem \eqref{eq:EvolEqSDD} translate to properties of $L_{G}g$ and $L_{D}d$ respectively. Specifically, for almost uniform Lipschitz-continuity of the lifted inhomogeneity we have to verify the following estimate uniformly in time,
\begin{multline*}
  \norm[\big]{L_{G}g(t, \varphi^{(1)}, \psi^{(1)}) - L_{G}g(t, \varphi^{(2)}, \psi^{(2)})}_{H_{0}} \\
  \leq C_{\alpha}\big[\norm[\big]{\varphi^{(1)}-\varphi^{(2)}}_{\sobh^{1}(-h,0;H_{0})} + \norm[\big]{\psi^{(1)}-\psi^{(2)}}_{\sobh^{1}(-h,0;H_{1})}\big]
\end{multline*}
for all $(\varphi^{(j)},\psi^{(j)}) \in \sobw^{1,\infty}(-h,0;H_{0}) \!\times\! \sobw^{1,\infty}(-h,0;H_{1})$, $j \in \{1,2\}$ satisfying\footnote{Corresponding to the definition of $V_{\alpha}$ from \cref{def:AlmLC}.}
\begin{equation}
  \label{eq:ValphaReq}
  \max_{j \in \{1,2\}} \big\{\norm{\varphi^{(j)}}_{\infty}, \norm{\psi^{(j)}}_{\infty}\big\} \leq \alpha.
\end{equation}
Using that $L$ is a Lipschitz-continuous operator on its corresponding boundary data space, the requirement for an inhomogeneity $F$ in \cref{def:AlmLC} changes to
\begin{multline}
  \label{eq:Lipschitz1}
  \norm[\big]{g(t,\varphi^{(1)},\psi^{(1)}) - g(t,\varphi^{(2)},\psi^{(2)})}_{\mathcal{BD}(G)} \\\leq C_{\alpha}\big(\norm{\varphi^{(1)}-\varphi^{(2)}}_{\sobh^{1}(-h,0;H_{0})} + \norm{\psi^{(1)} - \psi^{(2)}}_{\sobh^{1}(-h,0;H_{1})}\big)
\end{multline}
for positive constants $C_{\alpha}>0$ for all $\varphi^{(1)},\varphi^{(2)}\in \sobw^{1,\infty}(-h,0;H_{0})$ and $\psi^{(1)},\psi^{(2)}\in \sobw^{1,\infty}(-h,0;H_{1})$ satisfying \eqref{eq:ValphaReq}.\\
If one wants to define the right-hand side on histories of traces, i.e.\ the inhomogeneity is of the form
\begin{equation*}
  F(t,\varphi,\psi) \coloneq \widetilde{F}\big(t, \pi_{\mathcal{BD}(G)}(\varphi), \pi_{\mathcal{BD}(D)}(\psi)\big),
\end{equation*}
then the estimate
\begin{multline}
  \label{eq:Lipschitz2}
  \norm[\big]{\widetilde{F}(t, \varphi^{(1)},\psi^{(1)}) \!- \!\widetilde{F}(t,\varphi^{(2)},\psi^{(2)})}_{H} \\
  \leq C_{\alpha} \big(\norm{\varphi^{(1)} \!- \!\varphi^{(2)}}_{\mathcal{BD}(G)} + \norm{\psi^{(1)}\!-\!\psi^{(2)}}_{\mathcal{BD}(D)}\big)
\end{multline}
has to hold for all $\varphi^{(j)} \in \sobw^{1,\infty}\big(-h,0;\mathcal{BD}(G)\big)$ and $\psi^{(j)} \in \sobw^{1,\infty}\big(-h,0;\mathcal{BD}(D)\big)$, $j \in \{1,2\}$, satisfying \cref{eq:ValphaReq}.
\begin{theorem}
  \label{th:localExistDN}
  Let $M,M',N$ satisfy the assumptions of \cref{th:Picard} and of \cref{th:PicardHigh} for $k=1$. Let $g\colon [0,\infty) \!\times\! \leb^{2}(-h,0;H_{0}) \!\times\! \leb^{2}(-h,0;H_{1})\to \mathcal{BD}(G)$ be
  \begin{itemize}[leftmargin = 4ex]
    \item almost uniformly Lipschitz-continuous in the sense of \cref{eq:Lipschitz1},
    \item satisfy $L_{G}g \in \operatorname{dom}(\partial_{t}M)$,
    \item regularity preserving and
    \item satisfy $g(\argdot,0,0)\in \leb^{2}_{\rho}(0,\infty;H_{0} \!\times\! H_{1})$.
  \end{itemize}
  Let $\Phi \in \sobh^{1}(-h,0;H_{0}\!\times\! H_{1})$ with $\norm{\Phi'}_{\infty}<\infty$ and let $Z$ be an $\sobh^{1}_{\rho}(-h,\infty;H_{0} \!\times\! H_{1})$-extension of $\Phi$ satisfying
  \begin{enumerate}[label=\roman*), leftmargin=5ex]
    \item $Z\vert_{(-h,0]}=\Phi$ and
    \item $Z\vert_{[0,\infty)}\in \mathcal{C}^{1}(0,\infty;H_{0} \!\times\! H_{1})$ with $\norm{Z'(0)}_{H_{0}\times H_{1}}\leq \norm{\Phi'}_{\infty}$.
  \end{enumerate}
  Then the problem
  \begin{equation*}
    \left\{
      \begin{aligned}
        \left[\partial_{t}M + N + \begin{pmatrix} 0 & D \\ \mathring{G} & 0 \end{pmatrix}\right]
        \begin{pmatrix} u \\ v \end{pmatrix} &= \begin{pmatrix} f_{1} \\ f_{2} \end{pmatrix},\\
        \pi_{\mathcal{BD}(G)}v &= g\big(\argdot, (I_{\rho}u + Z_{1})_{(\argdot)}, (I_{\rho}v + Z_{2})_{(\argdot)}\big),
      \end{aligned}
    \right.
  \end{equation*}
  for $f_{1}\in \sobh^{1}_{\tilde{\rho}}(0,\infty;H_{0})$ and $f_{2}\in \sobh^{1}_{\tilde{\rho}}(0,\infty;H_{1})$ for some $\tilde{\rho}>0$ satisfying
  \begin{equation*}
    \left[[\partial_{t}M + N]\begin{pmatrix} L_{G}g \\ 0 \end{pmatrix}\right](0)
    + \begin{pmatrix} (GL_{g}g)(0) \\ 0 \end{pmatrix}
    = \begin{pmatrix} f_{1}(0) \\ f_{2}(0) \end{pmatrix},
  \end{equation*}
  admits a unique local solution $u \in \sobh^{1}(0,T;H_{0} \!\times\! H_{1})$.
\end{theorem}

\begin{proof}
  By assumption, the problem can be translated into the shifted and lifted problem from \cref{th:ShiftingLifting} (i),
  \begin{equation*}
    \left[\partial_{t}M + N + \begin{pmatrix} 0 & D \\ \mathring{G} & 0 \end{pmatrix}\right]
    \begin{pmatrix} r \\ v \end{pmatrix}
    = \begin{pmatrix} f_{1} \\ f_{2} \end{pmatrix}
    - \left[\partial_{t}M + N + \begin{pmatrix} 0 & D \\ G & 0 \end{pmatrix}\right]
    \begin{pmatrix} L_{G}g \\ 0 \end{pmatrix}.
  \end{equation*}
  For an application of \cref{th:localExistIns}, we realize that the assumption that $t \mapsto L_{G}g(t, u_{(t)},v_{(t)})$ is regularity preserving remains unaffected by lifting or projecting, because of tensor product structure:\\
  Indeed, by assumption, for $u \!\in\! \sobh^{1}_{\rho}(-h,\infty;H_{0})$ with $\norm{u'}_{\infty}\!<\!\infty$ and $v \!\in\! \sobh^{1}_{\rho}(-h,\infty;H_{1}\!)$ with $\norm{v'}_{\infty}<\infty$, we have $t\mapsto g(t, u_{(t)},v_{(t)})$ as an element in $\sobh^{1}_{\rho}\big(0,\infty;\mathcal{BD}(G)\big)$. Now using $L_{G}\partial_{t} = \partial_{t}L_{G}$, $t\mapsto L_{G}g(t,u_{(t)},v_{(t)})$ being in $\sobh^{1}_{\rho}(0,\infty;H_{0})$ follows.
\end{proof}

In a similar fashion, local existence and uniqueness results for the remaining cases can be formulated:
\begin{enumerate}[label = (\roman*), leftmargin = 5ex]
  \item For an inhomogeneity $I_{\rho}g(\argdot, u_{(\argdot)},v_{(\argdot)})$ corresponding to \cref{th:localExistOut}.
  \item For a Lipschitz condition adhering to \cref{eq:Lipschitz2}.
  \item For a Neumann boundary datum $\pi_{\mathcal{BD}(D)}v = d(t,u_{(t)},v_{(t)})$, appealing to part (ii) of \cref{th:ShiftingLifting}.
\end{enumerate}

In \cref{app:AbstractBD} we elaborate on the connection between abstract boundary data spaces and classical trace spaces and hint at versions of \cref{th:localExistDN} for the cases
\begin{enumerate}[label = (\roman*), leftmargin = 5ex]
  \item $A = \begin{psmallmatrix} 0 & \operatorname{div} \\ \operatorname{grad} & 0 \end{psmallmatrix}$ on $H = \leb^2(\Omega) \!\times\! \leb^2(\Omega)^{d}$,
  \item $A = \begin{psmallmatrix}  0 & \operatorname{div}T \operatorname{grad} \\ \iota & 0 \end{psmallmatrix}$ on $H = \leb^2(\Omega) \!\times\! \sobh^{1}(\Omega)$ and
  \item $A = \begin{psmallmatrix} 0 & -\operatorname{curl} \\ \operatorname{curl} & 0 \end{psmallmatrix}$ on $H = \leb^2(\Omega)^{3} \!\times\! \leb^2(\Omega)^{3}$.
\end{enumerate}

\subsection{State-dependent delay in extended state spaces}
\label{subsec:ExtendedState}
In the previous section, boundary conditions were directly included into a system in form of an inhomogeneity via a shifting and lifting approach. In this subsection, a different strategy is pursued: We extend the state space to incorporate trace maps into the spatial operator $A$. This approach allows the formulation of boundary conditions as part of the system itself. To that end, the formalism of abstract $\operatorname{div}$-$\operatorname{grad}$ systems from \cite{PicardSeidlerTrostorffWaurick2016} is used, the central concepts of which we briefly recall.
\begin{definition}[{\cite[def.~1.1]{PicardSeidlerTrostorffWaurick2016}}]
  Let $C\colon X_{0}\supseteq X_{1} \to Y$, be a densely defined, closed, linear operator with domain $X_{1}$ between Hilbert spaces $X_{0}$ and $Y$. The operator $A \coloneq \begin{psmallmatrix} 0 & -C^{\ast} \\ C & 0 \end{psmallmatrix}$ defines an {\em abstract $\operatorname{div}$-$\operatorname{grad}$ system} if $Y$ is a direct sum of Hilbert spaces $Y_{k}$, i.e.
  \begin{equation*}
    Y = \bigoplus_{1 \leq k \leq n}Y_{k}.
  \end{equation*}
\end{definition}
The boundary condition encoded into an abstract $\operatorname{div}$-$\operatorname{grad}$ system $A$ can be described via the formal adjoint.
\begin{definition}
  Let $X_{0}$ and $X_{1}$ be Hilbert spaces and $C\colon X_{0}\supseteq \operatorname{dom}(C) \to X_{1}$ be linear, densely defined and closed. Set
  \begin{equation*}
    C^{\diamond}\colon X_{1}\to \operatorname{dom}(C)',\quad (C^{\diamond}\varphi)(x)\coloneq \dualprod{\varphi}{C x}_{X_{1}} \quad\text{for}\ \varphi \in X_{1},x\in \operatorname{dom}(C)\text{.}
  \end{equation*}
\end{definition}
Using the formal adjoint $^{\diamond}$, the adjoint of $C$ or equivalently, the exact boundary conditions encoded into the domain of $C^{\ast}$ can be characterized.
\begin{lemma}[{\cite[lem.~1.4]{PicardSeidlerTrostorffWaurick2016}}]
  Let $X_{0},X_{1},Y$ be Hilbert spaces, $X_{1}\subseteq X_{0}$ dense and $C\colon X_{0}\supseteq X_{1} \to Y$ a closed, linear operator. Then the adjoint $C^{\ast}$ is densely defined and given by
  \begin{equation*}
    C^{\ast} = C^{\diamond} \cap (Y \oplus X_{0}),
  \end{equation*}
  which is the same as to say
  \begin{equation*}
    C^{\ast} = \dset{(y,x) \in Y \oplus X'_{1}}{C^{\diamond}y = x \,\text{and}\ x\in X_{0}}.
  \end{equation*}
\end{lemma}
An abstract $\operatorname{div}$-$\operatorname{grad}$ system $A = \begin{psmallmatrix} 0 & - C^{\ast} \\ C & 0 \end{psmallmatrix}$ can be defined from any densely defined, closed $C_{0}\colon X_{0}\supseteq \operatorname{dom}(C_{0}) \to Y$.
\begin{proposition}[{\cite[prop.~1.9]{PicardSeidlerTrostorffWaurick2016}}]
  Let $X_{0},Y_{0},\dots ,Y_{n}$ be Hilbert spaces, $C_{0}\colon X_{0}\supseteq \operatorname{dom}(C_{0})\to Y_{0}$ densely defined and closed. Denote $X_{1}\coloneq \big(\operatorname{dom}(C_{0}), \sqrt{\norm{\argdot}^{2} + \norm{C\argdot}^{2}}\big)$ and let $C_{k} \in \mathcal{L}(X_{1},Y_{k})$, $k \in \{1,\dots, n\}$. Then
  \begin{equation*}
    C = \begin{pmatrix} C_{0} \\ \vdots \\ C_{n} \end{pmatrix} \colon X_{0}\supseteq \operatorname{dom}(C_{0}) \to \bigoplus_{k \in \{0,\dots, n\}}Y_{k}
  \end{equation*}
  generates an abstract $\operatorname{div}$-$\operatorname{grad}$ system.
\end{proposition}

\subsection{Examples}
\label{subsubsec:ExtendedStateExample}
Let $\Omega \subseteq \mathbb{R}^{d}$ be open and bounded with Lipschitz boundary $\partial\Omega$ and unit outward normal field $\nu$. We first give an overview over some instances of abstract $\operatorname{div}$-$\operatorname{grad}$ systems, mostly without proof. For the definition of the trace operators involved see \cref{subsec:Traces}.

\subsubsection{$\operatorname{div}$ and $\operatorname{grad}$}
\begin{enumerate}[leftmargin = 4ex]
  \item\label{it:DivGradVersion1} The choices
        \begin{alignat*}{3}
          C_{0}\coloneq &&\operatorname{grad} &\colon \leb^2(\Omega) \supseteq \sobh^{1}(\Omega) &&\to \leb^2(\Omega)^{d},\\
          C_{1}\coloneq &&\gamma_{0} &\colon \leb^2(\Omega) \supseteq \sobh^{1}(\Omega) &&\to \leb^2(\partial\Omega),
        \end{alignat*}
        define the operator
        \begin{equation*}
          A = \begin{pmatrix} 0 & - \begin{pmatrix} \operatorname{grad} \\ \gamma_{0} \end{pmatrix}^{\!\!\ast}\\
            \begin{pmatrix} \operatorname{grad} \\ \gamma_{0} \end{pmatrix} & 0
          \end{pmatrix},
        \end{equation*}
        which forms an abstract $\operatorname{div}$-$\operatorname{grad}$ system on the state space
        \begin{equation*}
          H = \leb^2(\Omega) \times \leb^2(\Omega)^{d} \times \leb^2(\partial\Omega).
        \end{equation*}
        From $\begin{psmallmatrix} \mathring{\operatorname{grad}} \\ 0 \end{psmallmatrix} \subseteq C$ we can infer $C^{\ast} \subseteq \begin{pmatrix} -\operatorname{div} & 0 \end{pmatrix}$. Easy calculations show that
        \begin{align*}
          \begin{pmatrix} \operatorname{grad} \\ \gamma_{0} \end{pmatrix}^{\!\!\ast} \begin{pmatrix} q \\ \eta \end{pmatrix} &= -\operatorname{div}q,\\
          \dualprod{\operatorname{div}q}{\varphi}_{\leb^2(\Omega)} &= - \dualprod{q}{\operatorname{grad}\varphi}_{\leb^2(\Omega)} + \dualprod{\eta}{\gamma_{0} \varphi}_{\leb^2(\partial\Omega)}
        \end{align*}
        for $\varphi \in \sobh^{1}(\Omega)$. Hence $\eta = \gamma_{\nu}q$ describes the boundary condition attached to the operator $A$.
  \item\label{it:DivGradVersion2} The choices
        \begin{alignat*}{3}
          C_{0}\coloneq && \operatorname{div} &\colon \leb^2(\Omega)^{d} \supseteq \hat{\sobh}(\operatorname{div};\Omega) &&\to \leb^2(\Omega),\\
          C_{1}\coloneq && \gamma_{\nu} &\colon \leb^2(\Omega)^{d} \supseteq \hat{\sobh}(\operatorname{div};\Omega) &&\to \leb^2(\partial\Omega),
        \end{alignat*}
        define the abstract $\operatorname{div}$-$\operatorname{grad}$ system
        \begin{equation*}
          A = \begin{pmatrix} 0 & - \begin{pmatrix} \operatorname{div} \\ \gamma_{\nu} \end{pmatrix}^{\!\!\ast}\\
            \begin{pmatrix} \operatorname{div} \\ \gamma_{\nu} \end{pmatrix} & 0
          \end{pmatrix}
        \end{equation*}
        on the state space
        \begin{equation*}
          H = \leb^2(\Omega)^{d} \times \leb^2(\Omega) \times \leb^2(\partial\Omega).
        \end{equation*}
        In this context, $\hat{\sobh}(\operatorname{div};\Omega)$ denotes those $\sobh(\operatorname{div};\Omega)$-functions $\varphi$ with normal trace $\gamma_{\nu}\varphi \in \leb^{2}(\partial\Omega)\subseteq \sobh^{-\frac{1}{2}}(\partial\Omega)$, equipped with the inner product
        \begin{equation*}
          \dualprod{\varphi}{\psi}_{\hat{\sobh}(\operatorname{div};\Omega)}
          \coloneq \dualprod{\operatorname{div}\varphi}{\operatorname{div}\psi}_{\leb^{2}(\Omega)}
          + \dualprod{\varphi}{\psi}_{\leb^{2}(\Omega)^{d}}
          + \dualprod{\gamma_{\nu}\varphi}{\gamma_{\nu}\psi}_{\leb^{2}(\partial\Omega)}.
        \end{equation*}
        Simple calculations show that
        \begin{align*}
          \begin{pmatrix} \operatorname{\operatorname{div}} \\ \gamma_{\nu} \end{pmatrix}^{\!\!\ast} \begin{pmatrix} q \\ \eta \end{pmatrix} &= -\operatorname{grad}q,\\
          \dualprod{\operatorname{grad}q}{\varphi}_{\leb^2(\Omega)} &= - \dualprod{q}{\operatorname{div}\varphi}_{\leb^2(\Omega)} + \dualprod{\eta}{\gamma_{0} \varphi}_{\leb^2(\partial\Omega)}
        \end{align*}
        for $\varphi \in \hat{\sobh}^{1}(\operatorname{div};\Omega)$. Hence $\eta = \gamma_{0}q$ describes the boundary condition attached to the operator $A$.
\end{enumerate}

\subsubsection{$\operatorname{div} T \operatorname{grad}$ and $\iota$}
Here we consider an elliptic divergence form operator $\operatorname{div}T \operatorname{grad}$ with $T\colon \Omega\to \mathbb{R}^{d\times d}$ being a measurable and bounded function of sym-metric matrices satisfying $T \geq c > 0$.
\begin{enumerate}[leftmargin = 4ex]\setcounter{enumi}{2}
  \item\label{it:LaplaceVersion1} First let
        \begin{align*}
          X &\coloneq \dset{u \in \mathring{\sobh}^{1}(\Omega)}{T \operatorname{grad}u \in \sobh(\operatorname{div};\Omega)},\\
          \hat{X} &\coloneq \dset{u \in X}{\gamma_{\nu}T \operatorname{grad}u \in \leb^{2}(\partial\Omega)},
        \end{align*}
        where $\hat{X}$ is equipped with
        \begin{multline*}
          \qquad\dualprod{u}{v}_{\hat{X}}\coloneq \dualprod{\operatorname{div}T \operatorname{grad}u}{\operatorname{div}T \operatorname{grad}v}_{\leb^{2}(\Omega)}\\
          + \dualprod{T \operatorname{grad}u}{\operatorname{grad}v}_{\leb^{2}(\Omega)^{d}}
          + \dualprod{\gamma_{\nu}T \operatorname{grad}u}{\gamma_{\nu}T \operatorname{grad}v}_{\leb^{2}(\partial\Omega)}.
        \end{multline*}
        The choices
        \begin{alignat*}{4}
          C_{0}\coloneq &&\operatorname{div}T \operatorname{grad} &\colon \mathring{\sobh}^{1}(\Omega) &&\supseteq \hat{X} &&\to \leb^2(\Omega),\\
          C_{1}\coloneq &&\gamma_{n}\coloneq \gamma_{\nu}T \operatorname{grad} &\colon \leb^2(\Omega) &&\supseteq \hat{X} &&\to \leb^2(\partial\Omega),
        \end{alignat*}
        lead to the operator
        \begin{equation*}
          A = \begin{pmatrix} 0 & - \begin{pmatrix} \operatorname{div}T \operatorname{grad} \\ \gamma_{n} \end{pmatrix}^{\!\!\ast}\\
            \begin{pmatrix} \operatorname{div}T \operatorname{grad} \\ \gamma_{n} \end{pmatrix} & 0
          \end{pmatrix}
        \end{equation*}
        defining an abstract $\operatorname{div}$-$\operatorname{grad}$ system on the state space
        \begin{equation*}
          H = \mathring{\sobh}^{1}(\Omega) \times \leb^2(\Omega) \times \leb^2(\partial\Omega),
        \end{equation*}
        where $\mathring{\sobh}^{1}(\Omega)$ is equipped with the inner product
        \begin{equation*}
          \dualprod{x}{y}_{\mathring{\sobh}^{1}(\Omega)} \coloneq \dualprod{T \operatorname{grad}x}{\operatorname{grad}y}_{\leb^2(\Omega)^{d}}.
        \end{equation*}
        For $\varphi \in \hat{X}$ and $(u,v)\in \sobh^{1}(\Omega) \times \leb^2(\partial\Omega)$ we can calculate
        \begin{align*}
          \quad&\dualprod[\bigg]{\begin{pmatrix}\operatorname{div}T \operatorname{grad} \\ \gamma_{n}\end{pmatrix}\varphi}{\begin{pmatrix}u \\ v\end{pmatrix}}_{\leb^2(\Omega)\times \leb^2(\partial\Omega)}\\
          &= \dualprod[\big]{\operatorname{div}T \operatorname{grad} \varphi}{u}_{\leb^2(\Omega)}
            \!+\! \dualprod[\big]{\gamma_{n}\varphi}{v}_{\leb^2(\partial\Omega)}\\
          &= - \dualprod[\big]{T \operatorname{grad}\varphi}{\operatorname{grad}u}_{\leb^2(\Omega)^{d}}
            \!+\! \dualprod[\big]{\gamma_{\nu}T \operatorname{grad}\varphi}{\gamma_{0}u}_{\leb^2(\partial\Omega)}
            \!+\! \dualprod[\big]{\gamma_{\nu}T \operatorname{grad}\varphi}{v}_{\leb^2(\partial\Omega)}\\
          &= - \dualprod{\iota \varphi}{u}_{\mathring{\sobh}^{1}(\Omega)}
            \!+\! \dualprod{\gamma_{\nu}T \operatorname{grad}\varphi}{\gamma_{0}u}_{\leb^2(\partial\Omega)}
            \!+\! \dualprod{\gamma_{\nu}T \operatorname{grad}\varphi}{v}_{\leb^2(\partial\Omega)},
        \end{align*}
        where $\iota\colon \leb^2(\Omega)\supseteq \sobh^{1}(\Omega) \to \leb^2(\Omega)$ is the unbounded inclusion. Thus, the boundary condition contained in $A$ is $-v = \gamma_{0}u$. We also refer to \cite[thm.~B.1]{Aigner2026} for the $\leb^{2}(\partial\Omega)$-density of the Neumann trace.
  \item\label{it:LaplaceVersion2} The choices
        \begin{alignat*}{4}
          C_{0}\coloneq && \iota &\colon \leb^2(\Omega) &&\supseteq \sobh^{1}(\Omega) &&\to \leb^2(\Omega),\\
          C_{1}\coloneq &&\gamma_{0} &\colon \leb^2(\Omega) &&\supseteq \sobh^{1}(\Omega) &&\to \leb^2(\partial\Omega),
        \end{alignat*}
        lead to the operator
        \begin{equation*}
          A = \begin{pmatrix} 0 & - \begin{pmatrix} \iota \\ \gamma_{0} \end{pmatrix}^{\!\!\ast}\\
            \begin{pmatrix} \iota \\ \gamma_{0} \end{pmatrix} & 0
          \end{pmatrix},
        \end{equation*}
        which defines an abstract $\operatorname{div}$-$\operatorname{grad}$ system on the state space
        \begin{equation*}
          H = \leb^2(\Omega) \times \sobh^{1}(\Omega) \times \leb^2(\partial\Omega),
        \end{equation*}
        where $\sobh^{1}(\Omega)$ is equipped with
        \begin{equation*}
          \dualprod{u}{v}_{\sobh^{1}} \coloneq \dualprod{T \operatorname{grad}u}{\operatorname{grad}v}_{\leb^{2}(\Omega)^{d}} + \dualprod{\gamma_{0}u}{\gamma_{0}v}_{\leb^{2}(\partial\Omega)}.
        \end{equation*}
        Here we calculate for $\varphi \in \sobh^{1}(\Omega)$, $u \in \hat{X}$ and $v\in \leb^2(\partial\Omega)$,
        \begin{align*}
          &\dualprod[\bigg]{\begin{pmatrix} \iota \\ \gamma_{0}\end{pmatrix}\varphi}{\begin{pmatrix} u \\ v \end{pmatrix}}_{\sobh^{1}(\Omega)\times \leb^2(\partial\Omega)}\\
          &\quad= \dualprod[\big]{\varphi}{u}_{\sobh^{1}(\Omega)} + \dualprod[\big]{\gamma_{0}\varphi}{v}_{\leb^2(\partial\Omega)}\\
          &\quad= \dualprod[\big]{T\operatorname{grad}\varphi}{\operatorname{grad}u}_{\leb^2(\Omega)^{d}} + \underbrace{\dualprod[\big]{\gamma_{0}\varphi}{\gamma_{0}u}_{\leb^2(\partial\Omega)}}_{=0} + \dualprod[\big]{\gamma_{0}\varphi}{v}_{\leb^2(\partial\Omega)}\\
          &\quad= -\dualprod[\big]{\varphi}{\operatorname{div}T \operatorname{grad}u}_{\leb^2(\Omega)} + \dualprod[\big]{\gamma_{0}\varphi}{\gamma_{\nu}T \operatorname{grad}u}_{\leb^2(\partial\Omega)} + \dualprod[\big]{\gamma_{0}\varphi}{v}_{\leb^2(\partial\Omega)}.
        \end{align*}
        Thus, the induced boundary condition on $A$ is $-v = \gamma_{\nu}T \operatorname{grad}u = \gamma_{n}u$.
\end{enumerate}

\subsubsection{$\operatorname{curl}$ and $\operatorname{curl}$}
Let $\Omega \subseteq \mathbb{R}^{3}$ be an open, bounded Lipschitz domain.
\begin{enumerate}[leftmargin = 4ex]\setcounter{enumi}{4}
  \item\label{it:MaxwellVersion1} First let
        \begin{equation*}
          \hat{\sobh}(\operatorname{curl};\Omega) \coloneq \dset{\varphi \in \sobh(\operatorname{curl};\Omega)}{\gamma_{\tau}\varphi \in \leb^{2}_{\tau}(\partial\Omega)},
        \end{equation*}
        equipped with
        \begin{equation*}
          \dualprod{u}{v}_{\hat{\sobh}^{1}(\operatorname{curl};\Omega)} \coloneq \dualprod{u}{v}_{\leb^{2}(\Omega)^{3}} + \dualprod{\operatorname{curl}u}{\operatorname{curl}v}_{\leb^{2}(\Omega)^{3}} + \dualprod{\pi_{\tau}u}{\pi_{\tau}v}_{\leb^{2}_{\tau}(\partial\Omega)}.
        \end{equation*}
        The choices
        \begin{alignat*}{4}
          C_{0}\coloneq &&-\operatorname{curl} &\colon \leb^2(\Omega)^{3} &&\supseteq \hat{\sobh}(\operatorname{curl};\Omega) &&\to \leb^2(\Omega)^{3},\\
          C_{1}\coloneq && \gamma_{\tau} &\colon \leb^2(\Omega)^{3} &&\supseteq \hat{\sobh}(\operatorname{curl};\Omega) &&\to \leb^2_{\tau}(\partial\Omega),
        \end{alignat*}
        lead to the operator
        \begin{equation*}
          A = \begin{pmatrix}
            \begin{pmatrix} 0 & 0 & 0 \\ 0 & 0 & 0 \\ 0 & 0 & 0 \end{pmatrix} &
            - \begin{pmatrix} - \operatorname{curl} \\ \gamma_{\tau} \end{pmatrix}^{\ast} \\
            \begin{pmatrix} -\operatorname{curl} \\ \gamma_{\tau} \end{pmatrix} &
            \begin{pmatrix} \begin{pmatrix} 0 & 0 & 0 \\ 0 & 0 & 0 \\ 0 & 0 & 0 \end{pmatrix}
              & \begin{pmatrix} 0 & 0 & 0 \\ 0 & 0 & 0 \\ 0 & 0 & 0 \end{pmatrix}\\
              \begin{pmatrix} 0 & 0 & 0 \\ 0 & 0 & 0 \\ 0 & 0 & 0 \end{pmatrix}
              & \begin{pmatrix} 0 & 0 & 0 \\ 0 & 0 & 0 \\ 0 & 0 & 0 \end{pmatrix} \end{pmatrix}
          \end{pmatrix}
        \end{equation*}
        defining an abstract $\operatorname{div}$-$\operatorname{grad}$ system on the state space
        \begin{equation*}
          H = \leb^2(\Omega)^{3} \times \leb^2(\Omega)^{3} \times \leb^2_{\tau}(\partial\Omega).
        \end{equation*}
        A calculation similar to \cite[sec.~2.3]{PicardSeidlerTrostorffWaurick2016} shows that $\begin{psmallmatrix} - \operatorname{curl} \\ \gamma_{\tau} \end{psmallmatrix}^{\ast} \subseteq \begin{pmatrix} -\operatorname{curl} & 0 \end{pmatrix}$ and that the boundary condition carried by the operator $A$ is
        \begin{equation*}
          \forall (u,v,w) \in \operatorname{dom}(A)\colon w = - \pi_{\tau}v.
        \end{equation*}
  \item\label{it:MaxwellVersion2} Alternatively, the choices
        \begin{alignat*}{4}
          C_{0}\coloneq && -\operatorname{curl} &\colon \leb^2(\Omega)^{3} &&\supseteq \hat{\sobh}(\operatorname{curl};\Omega) &&\to \leb^2(\Omega)^{3},\\
          C_{1}\coloneq && \pi_{\tau} &\colon \leb^2(\Omega)^{3} &&\supseteq \hat{\sobh}(\operatorname{curl};\Omega) &&\to \leb^2(\partial\Omega)^{3},
        \end{alignat*}
        lead to the operator
        \begin{equation*}
          A = \begin{pmatrix}
            \begin{pmatrix} 0 & 0 & 0 \\ 0 & 0 & 0 \\ 0 & 0 & 0 \end{pmatrix} &
            - \begin{pmatrix} - \operatorname{curl} \\ \pi_{\tau} \end{pmatrix}^{\ast} \\
            \begin{pmatrix} -\operatorname{curl} \\ \pi_{\tau} \end{pmatrix} &
            \begin{pmatrix} \begin{pmatrix} 0 & 0 & 0 \\ 0 & 0 & 0 \\ 0 & 0 & 0 \end{pmatrix}
              & \begin{pmatrix} 0 & 0 & 0 \\ 0 & 0 & 0 \\ 0 & 0 & 0 \end{pmatrix}\\
              \begin{pmatrix} 0 & 0 & 0 \\ 0 & 0 & 0 \\ 0 & 0 & 0 \end{pmatrix}
              & \begin{pmatrix} 0 & 0 & 0 \\ 0 & 0 & 0 \\ 0 & 0 & 0 \end{pmatrix} \end{pmatrix}
          \end{pmatrix},
        \end{equation*}
        which defines an abstract $\operatorname{div}$-$\operatorname{grad}$ system on the state space
        \begin{equation*}
          H = \leb^2(\Omega)^{3} \times \leb^2(\Omega)^{3} \times \leb^2_{\tau}(\partial\Omega),
        \end{equation*}
        cf.~\cite[sec.~2.3]{PicardSeidlerTrostorffWaurick2016}. There it is also shown that $\begin{psmallmatrix} - \operatorname{curl} \\ \pi_{\tau} \end{psmallmatrix}^{\ast} \subseteq \begin{pmatrix} -\operatorname{curl} & 0 \end{pmatrix}$ and that the boundary condition carried by $A$ is
        \begin{equation*}
          \forall (u,v,w)\in \operatorname{dom}(A)\colon w = -\gamma_{\tau}v.
        \end{equation*}
\end{enumerate}

\subsubsection{Wentzell boundary conditions}
Let $\Omega \subseteq \mathbb{R}^{d}$ be open and bounded with Lip-schitz boundary $\partial\Omega$ and unit outward normal field $\nu$. Assume that $\partial\Omega$ is a manifold, where it is possible to define the covariant derivative $\operatorname{grad}_{\partial\Omega}$ as an operator from $\leb^2(\partial\Omega)$ to $\leb^2_{\tau}(\partial\Omega)$. In this segment we want to briefly recapitulate a more complex instance of an abstract $\operatorname{div}$-$\operatorname{grad}$ system from \cite[sec.~2.2]{PicardSeidlerTrostorffWaurick2016}.
\begin{enumerate}[leftmargin = 4ex]\setcounter{enumi}{6}
  \item\label{it:Wentzell} First let
        \begin{equation*}
          X \coloneq \Big(\operatorname{dom}(\operatorname{grad}_{\partial\Omega}\gamma_{0}), \sqrt{\norm{\operatorname{grad}_{\partial\Omega}\gamma \argdot}_{\leb^2(\partial\Omega)}^{2} + \norm{\argdot}^{2}_{\sobh^{1}(\Omega)}}\Big).
        \end{equation*}
        Then one can define
        \begin{alignat*}{5}
          &C_{0}\coloneq &\operatorname{grad}&&\colon \leb^{2}(\Omega) &&\supseteq X &&\to \leb^{2}(\Omega)^{d},\\
          &C_{1}\coloneq &\gamma_{0}&&\colon \leb^{2}(\Omega) &&\supseteq X &&\to \leb^{2}(\partial\Omega),\\
          &C_{2}\coloneq &\operatorname{grad}_{\partial\Omega}\gamma_{0}&&\colon \leb^{2}(\Omega) &&\supseteq X &&\to \leb^{2}(\partial\Omega).
        \end{alignat*}
        As shown in \cite[sec.~2.2]{PicardSeidlerTrostorffWaurick2016}, the operator
        \begin{equation*}
          A = \begin{pmatrix}
            \begin{pmatrix} 0 \end{pmatrix} &
            - \begin{pmatrix} \operatorname{grad} \\ \gamma_{0} \\ \operatorname{grad}_{\partial\Omega}\gamma_{0} \end{pmatrix}^{\!\!\ast} \\
            \begin{pmatrix} \operatorname{grad} \\ \gamma_{0} \\ \operatorname{grad}_{\partial\Omega}\gamma_{0} \end{pmatrix} &
        \begin{pmatrix} 0 & 0 & 0 \\ 0 & 0 & 0 \\ 0 & 0 & 0 \end{pmatrix}
        \end{pmatrix}
        \end{equation*}
        defines an abstract $\operatorname{div}$-$\operatorname{grad}$ system on the state space
        \begin{equation*}
          H = \leb^2(\Omega) \times \leb^2(\Omega)^{d} \times \leb^2(\Omega) \times \leb_{\tau}^2(\partial\Omega).
        \end{equation*}
        Moreover, it can be shown that
        \begin{equation*}
          \begin{pmatrix} \operatorname{grad} \\ \gamma_{0} \\ \operatorname{grad}_{\partial\Omega}\gamma_{0} \end{pmatrix}^{\!\!\ast} \subseteq \begin{pmatrix} - \operatorname{div} & 0 & 0 \end{pmatrix}
        \end{equation*}
        and the boundary condition encoded into $A$ is
        \begin{equation*}
          \forall (u,v,\eta_{1},\eta_{2}) \in \operatorname{dom}(A)\colon \gamma_{\nu}v + \eta_{1} = - \operatorname{grad}_{\partial\Omega}^{\diamond}\eta_{2}.
        \end{equation*}
\end{enumerate}

\subsection{Applications}
\begin{example}[heat equation]
  The heat equation can be formulated with any of the formulations \eqref{it:DivGradVersion1}, \eqref{it:DivGradVersion2} or \eqref{it:Wentzell}. Formalism \eqref{it:DivGradVersion1} lends itself to the description of Robin boundary conditions with possibly nonautonomous delay in the Neumann trace; formalism \eqref{it:Wentzell} to the description of Wentzell--Robin boundary conditions with the Laplace--Beltrami operator. Some more details regarding those models are provided in \cref{app:Applications}. Here, we discuss the model utilizing formulation \eqref{it:DivGradVersion2} instead, i.e.
  \begin{equation}
    \label{eq:HeatEq}
    \left[\partial_{t}\begin{pmatrix} 0 & 0 & 0 \\ 0 & 1 & 0 \\ 0 & 0 & m \end{pmatrix}
    + \begin{pmatrix} a^{-1} & 0 & 0 \\ 0 & 0 & 0 \\ 0 & 0 & n \end{pmatrix}
    + \begin{pmatrix} 0 & - \begin{pmatrix} \operatorname{div} \\ \gamma_{\nu} \end{pmatrix}^{\!\!\ast}\\
    \begin{pmatrix} \operatorname{div} \\ \gamma_{\nu} \end{pmatrix} & 0
    \end{pmatrix}\right]
    \!\begin{pmatrix} u \\ v \\ w \end{pmatrix}
    \!= \!\begin{pmatrix} f \\ g \\ h \end{pmatrix},
  \end{equation}
  where $m,n \in \mathcal{L}\big(\leb^2_{\rho}(\mathbb{R};\leb^2(\partial\Omega))\big)$. The full set of equations described by the system reads
  \begin{align*}
    a^{-1}u + \operatorname{grad}v &= f,\\
    \partial_{t}v + \operatorname{div}u &= g,\\
    \partial_{t}mw + nw + \gamma_{\nu}u &= h,\\
    w + \gamma_{0}v &= 0,
  \end{align*}
  where the last equation describes the boundary condition carried by $A$. Hence, this system desribes a heat equation, with boundary condition
  \begin{equation*}
    -\partial_{t}m \gamma_{0}v - n\gamma_{0}v + \gamma_{\nu}u = h.
  \end{equation*}
  In the case where we have a strong solution to \cref{eq:HeatEq}, it is justified to substitute $u$ into the last expression, appealing to the differential equation $a^{-1}u + \operatorname{grad}v = f$, and we obtain
  \begin{equation}
    \label{eq:HeatBdy}
    -\partial_{t}m\gamma_{0}v - n\gamma_{0}v + \gamma_{\nu}T(f- \operatorname{grad}v) = h.
  \end{equation}
  The choices $f = 0$ and $m = 0$ yield
  \begin{equation*}
    n\gamma_{0}v + \gamma_{\nu}T\operatorname{grad}v = -h,
  \end{equation*}
  which is a Robin boundary condition with possibly nonautonomously delayed Dirichlet trace $\gamma_{0}v$. We now compare this approach with the example in \cite{Nicaise2009}, where the following one-dimensional heat equation was studied:
  \begin{equation}
    \label{eq:HeatNicaise}
    \left\{
      \begin{aligned}
        u_{t}(x, t) - \alpha u_{xx}(x, t) &= 0, &(x,t) \in (0,\pi) \!\times\! \mathbb{R},\\
        u(0, t) &= 0, &t \in \mathbb{R},\\
        u_{x}(\pi, t) &= -\mu_{0}u(\pi, t) - \mu_{1}u\big(\pi, t -\tau(t)\big), &t \in \mathbb{R},\\
        u(x, 0) &= u_{0}(x), &x\in (0,\pi),\\
        u\big(\pi, t - \tau(0)\big) &= f_{0}\big(t - \tau(0)\big), & t\in (0,\tau(0)),
      \end{aligned}
    \right.
  \end{equation}
  where $\alpha>0$, $\mu_{0},\mu_{1}\geq 0$ are constants and $\tau \in \sobw^{2,\infty}_{\mathrm{loc}}(\mathbb{R})$ satisfies
  \begin{itemize}[leftmargin = 4ex]
    \item $0 < \tau_{0}\leq \tau$ and
    \item $\tau' < 1$.
  \end{itemize}
  In our setup \eqref{eq:HeatEq} we generalize the situation immediately to bounded, open Lip-schitz domains $\Omega\subseteq \mathbb{R}^{d}$ with an elliptic divergence form operator $\operatorname{div} T \operatorname{grad}$ and a delay functional $\tau$ satisfying \cref{ass:Tau}. We also replace $\mu_{0}$ and $\mu_{1}$ with $b_{0},b_{1}\in \leb^{\infty}(\mathbb{R})$ and with the choice $n = b_{1} + b_{2}S_{\tau}$, where $S_{\tau}$ is the shift operator, cf.~\cref{def:Shift}, we obtain the boundary condition
  \begin{equation*}
    b_{1}(t)\gamma_{0}v(t) + b_{2}(t)\gamma_{0}v(t-\tau(t)) + \gamma_{\nu}T \operatorname{grad}v(t) = - h.
  \end{equation*}
  Well-posedness of problem \eqref{eq:HeatEq} then follows verbatim as in \cref{subsubsec:Wave1}. For state-dependence of $h$ we additionally require $b_{1},b_{2}\in \sobw^{1,\infty}(\mathbb{R})$. We comment, that in \cref{subsubsec:Wave1}, we require $\tau' < a <1$, but the assumption $\tau \in \sobw^{2,\infty}_{\mathrm{loc}}(\mathbb{R})$ assures continuity of $\tau'$ and hence the stricter assumption $\tau' < a < 1$ holds on any finite interval $[0,T]$. This suffices for our purpose, because for state-dependent delay we are only interested in local well-posedness.\\
  We comment on the boundary conditions of problem \eqref{eq:HeatNicaise}: There, the delayed boun-dary condition is only in effect on a part of the boundary. This is possible on a domain $\Omega \subseteq \mathbb{R}^{d}$ as well by adequately restricting the differential operators to a suitable domain, e.g.
  \begin{equation*}
    \widetilde{\operatorname{grad}}\coloneq \operatorname{grad}\vert_{D_{1}}\quad\text{and}\quad
    \widetilde{\operatorname{div}}\coloneq \operatorname{div}\vert_{D_{2}},
  \end{equation*}
  with
  \begin{align*}
    D_{1} &\coloneq \{\varphi \in \sobh^{1}(\Omega)\colon (\gamma_{0}\varphi)\vert_{\Gamma} = 0\},\\
    D_{2} &\coloneq \{\varphi \in \sobh(\operatorname{div};\Omega)\colon (\gamma_{\nu}\varphi)\vert_{\Gamma} = 0\},
  \end{align*}
  where $\Gamma \subseteq \partial\Omega$ is a part of the boundary, that is itself a Lipschitz domain. The third component of the state-space then needs to be appropriately swapped and the state-space changed to
  \begin{equation*}
    H = \leb^2(\Omega)^{d} \times \leb^2(\Omega) \times \leb^2(\partial\Omega \!\setminus \!\Gamma).
  \end{equation*}
  We leave the details to the reader, but refer to \cite{Aigner2026} for an approach addressing a split boundary condition that illustrates the idea in the context of the wave equation.
\end{example}

\begin{remark}[Wentzell boundary conditions]
  The choices $f = 0$ and $m = 1$ in \cref{eq:HeatBdy} lead to the boundary condition
  \begin{equation*}
    -\partial_{t}\gamma_{0}v - n \gamma_{0}v - \gamma_{\nu}T \operatorname{grad} v = h.
  \end{equation*}
  An appeal to the equation $\partial_{t}v = g - \operatorname{div}u = g + \operatorname{div}T \operatorname{grad}v$ yields the boundary condition
  \begin{equation*}
    - \gamma_{0}g - \gamma_{0}\operatorname{div} T \operatorname{grad} v - n\gamma_{0}v - \gamma_{\nu}T \operatorname{grad} v = h,
  \end{equation*}
  which corresponds to a Wentzell--Robin type boundary condition, cf.~\cite{Goldstein2006,KunzeMuiPloss2026}, but possibly includes a nonautonomous delay in the Dirichlet trace because of the pre-sence of $n$. Note that in \cite{KunzeMuiPloss2026}, the boundary conditions are nonlocal; however this situation can be addressed by including off-diagonal entries in the operator
  \begin{equation*}
    \partial_{t}\begin{pmatrix} 0 & 0 & 0 \\ 0 & 1 & 0 \\ 0 & 0 & m \end{pmatrix}
    + \begin{pmatrix} a^{-1} & 0 & 0 \\ 0 & n_{11} & n_{12} \\ 0 & n_{21} & n_{22} \end{pmatrix}
  \end{equation*}
  in \cref{eq:HeatEq}.
\end{remark}

The wave equation can be formulated using either of the formulations \eqref{it:DivGradVersion1}, \eqref{it:DivGradVersion2}, \eqref{it:LaplaceVersion1} or \eqref{it:LaplaceVersion2}, where the prior two correspond to the Lagrange formulation, the latter to the Dirac formulation of the problem.

\begin{example}[wave equation]
  \label{ex:Wave}
  In formalism \eqref{it:LaplaceVersion1}, the wave equation reads
  \begin{equation}
    \label{eq:WaveComp}
    \left[\!\partial_{t}\!\!\begin{pmatrix} 1 & 0 & 0 \\ 0 & 1 & 0 \\ 0 & 0 & m \end{pmatrix}
    \!\!+ \!\!\begin{pmatrix} 0 & 0 & 0 \\ 0 & 0 & 0 \\ 0 & 0 & n \end{pmatrix}
    \!\!- \!\!\begin{pmatrix} 0 & - \!\begin{pmatrix} \operatorname{div}T \operatorname{grad} \\ \gamma_{n} \end{pmatrix}^{\!\!\!\ast}\\
    \!\!\begin{pmatrix} \operatorname{div}T \operatorname{grad} \\ \gamma_{n} \end{pmatrix} & 0
    \end{pmatrix}\!\!\right]
    \!\!\!\begin{pmatrix} \!u\! \\ \!v\! \\ \!w\! \end{pmatrix}
    \!\!= \!\!\begin{pmatrix} \!f\! \\ \!g\! \\ \!h\! \end{pmatrix}\!\!,
  \end{equation}
  where again $m,n \in \mathcal{L}\big(\leb^2_{\rho}(\mathbb{R};\leb^2(\partial\Omega))\big)$.
  The full system of equations contained in \eqref{eq:WaveComp} reads
  \begin{align*}
    \partial_{t}u - v &= f,\\
    \partial_{t}v - \operatorname{div}T \operatorname{grad}u &= g,\\
    \partial_{t}mw + nw - \gamma_{n}u &= h,\\
    w + \gamma_{0}v &= 0,
  \end{align*}
  where the last equation corresponds to the boundary condition encoded in the operator $A$. Combining the last two equations results in the boundary condition
  \begin{equation}
    \label{eq:BoundaryWave}
    \partial_{t}m\gamma_{0}v + n\gamma_{0}v + \gamma_{n}u = -h,
  \end{equation}
  which corresponds to a Robin boundary condition for the choice $m=0$. We compare our model to the subject of \cite{NicaisePignottiValein2011}, i.e.\ the problem
  \begin{equation}
    \label{eq:WavePignotti}
    \left\{
      \begin{aligned}
        u_{tt}(x, t) - \Delta u(x, t) &= 0 &(x,t) \!\in \!\Omega \!\times \!(0,\infty),\\
        u(x, t) &= 0 &(x,t) \!\in \!\Gamma_{D} \!\times \!(0, \infty),\\
        \tfrac{\partial u}{\partial \nu}(x, t) &= -\mu_{1}u_{t}(x, t) \!- \!\mu_{2}u_{t}\big(x, t \!- \!\tau(t)\big) &(x,t) \!\in \!\Gamma_{N}\!\times \!(0,\infty),\\
        u(x, 0) &= u_{0}(x) & x\!\in \!\Omega,\\
        u_{t}(x, 0) &= u_{1}(x) & x \!\in \!\Omega,\\
        u_{t}(x, t -\tau(0)) &= f_{0}\big(x, t \!- \!\tau(0)\big) &(x,t) \!\in \!\Gamma_{N} \!\times \!\big(0,\tau(0)\big).
      \end{aligned}
    \right.
  \end{equation}
  \Cref{eq:WavePignotti} was studied for $\Omega\subseteq \mathbb{R}^{d}$ of class $\mathcal{C}^{2}$ with split boundary $\partial\Omega = \Gamma_{D}\dot\cup \Gamma_{N}$ and $\Gamma_{N} \neq \emptyset \neq \Gamma_{D}$. There, it was assumed that
  \begin{itemize}[leftmargin = 4ex]
    \item $\tau \in \sobw^{2,\infty}_{\mathrm{loc}}([0,\infty))$,
    \item $\tau' \leq d < 1$ for some $d \in \mathbb{R}$ and
    \item $\mu_{1},\mu_{2} \in [0,\infty)$ satisfy $\mu_{2} \leq \mu_{1}\sqrt{1-d}$.
  \end{itemize}
  Then two cases were studied in \cite{NicaisePignottiValein2011}:
  \begin{enumerate}[label = (\roman*), leftmargin = 4ex]
    \item $0 < \tau_{0}\leq \tau$ or
    \item $\tau \in W^{3,\infty}_{\mathrm{loc}}([0,\infty))$ and $-\tfrac{d}{1-d}\leq \tau' \leq d <1$.
  \end{enumerate}
  We generalize this scenario immediately to the situation discussed in \cref{subsubsec:Wave1}, i.e.\ $\Omega$ being a bounded, open Lipschitz domain; the Laplacian swapped with an elliptic divergence form operator $\operatorname{div}T \operatorname{grad}$ and $\mu_{1}$ and $\mu_{2}$ replaced with $b_{1},b_{2} \in \leb^{\infty}(\mathbb{R})$ respectively. For the delay functional $\tau \colon \mathbb{R}\to [0,h]$ we require \cref{ass:Tau}. The choices $m = 0$ and $n = (b_{1} + b_{2}S_{\tau})$ in \cref{eq:BoundaryWave}, where $S_{\tau}$ is the shift operator, cf.~\cref{def:Shift}, lead to a description of the wave equation with delayed Robin boundary condition
  \begin{equation*}
    \gamma_{n}u(t) + b_{1}(t)\gamma_{0}u_{t}(t) + b_{2}(t)\gamma_{0}u_{t}(t - \tau(t)) = -h
  \end{equation*}
  in the form of \cref{eq:WaveComp}. Verbatim as in \cref{subsubsec:Wave1} one can arrive at the coercivity estimates
  \begin{equation}
    \label{eq:Coercivity1}
    \Re \dualprod[\big]{(b_{1}+b_{2}S_{\tau})\varphi}{\varphi}_{\leb^2_{\rho}(\mathbb{R};\leb^2(\partial\Omega))}
    \geq \big(\inf b_{1} - \norm{b_{2}}_{\infty}\tfrac{\e^{-\rho \tau_{0}}}{\sqrt{1-d}}\Big)\norm{\varphi}_{\leb^{2}_{\rho}(\mathbb{R};\leb^{2}(\partial\Omega))}^{2},
  \end{equation}
  \begin{multline}
    \label{eq:Coercivity2}
    \;\Re \dualprod[\big]{(b_{1} + b_{2}S_{\tau})\varphi}{\varphi}_{\sobh^{1}_{\rho}(\mathbb{R};\leb^2(\partial\Omega))}\\
    \geq \big(\inf b_{1} \!- \!\norm{b_{2}(1 \!- \!\tau')}_{\infty} \tfrac{\e^{-\rho \tau_{0}}}{\sqrt{1-d}} - \tfrac{1}{\rho}\norm{b_{1}'}_{\infty} \!- \!\tfrac{1}{\rho}\norm{b_{2}'}_{\infty}\tfrac{\e^{-\rho \tau_{0}}}{\sqrt{1-d}}\big)\norm{\varphi'}_{\leb^{2}_{\rho}(\mathbb{R};\leb^{2}(\partial\Omega))}^{2},
  \end{multline}
  where $0 \leq \tau_{0} \coloneq \inf \tau$. For the situation considered in \cite{NicaisePignottiValein2011}, \cref{eq:Coercivity1} already suffices, since $h = 0$ there. We stress that for well-posedness in the model \eqref{eq:WaveComp}, for the case
  \begin{enumerate}[label = (\roman*), leftmargin = 4ex]
    \item $\tau_{0} > 0$, the condition $\inf b_{1} > 0$ (corresponding to $\mu_{1}>0$ in \cite{NicaisePignottiValein2011}) suffices.
    \item $\tau_{0} = 0$, we require
          \begin{equation*}
            \inf b_{1} > \tfrac{\norm{b_{2}}_{\infty}}{\sqrt{1-d}},
          \end{equation*}
          which corresponds to the condition $\mu_{2}< \mu_{1}\sqrt{1-d}$.
  \end{enumerate}
  For well-posedness in $\sobh^{1}_{\rho}$ and to cover the case of state-dependently delayed $h$, we additionally require $b_{1},b_{2} \in \sobw^{1,\infty}$; \cref{eq:Coercivity2} then holds for any $b_{1}$ with $\inf b_{1}>0$ for large enough $\rho$. Well-posedness of problem \eqref{eq:WaveComp} is assured under the assumptions of \cref{th:localExistIns} for suitable initial conditions.\\
  In \cite{NicaisePignottiValein2011}, the delayed boundary condition is only in effect on a part $\Gamma_{N}\subseteq \partial\Omega$ of the boundary. It is however only a minor technical nuisance to accommodate that situation into the model \eqref{eq:WaveComp}. One simply needs to suitably restrict operators and spaces involved to contain the stipulated Dirichlet boundary condition on $\Gamma_{D}$.\\
  We further remark that this setup generalizes the findings of \cite{SilgaBayili2021,SilgaBayiliZabsonre2022}, because there the coefficients $b_{1},b_{2}$ need to be positive multiplication operators. Our findings however are not more general than the more general scenario considered in \cite{NicaisePignottiValein2011}, where $b_{1}$ and $b_{2}$ are nonlinear operators.
\end{example}

\begin{remark}[Wentzell--Robin boundary conditions]
  The choice $m = 1$ in \cref{eq:BoundaryWave} allows the substitution of the differential equation $\partial_{t}v = \operatorname{div}T \operatorname{grad} u + g$ into the boundary condition for strong solutions and with the choice $g = 0$ yields
  \begin{equation*}
    \gamma_{0}\operatorname{div}T \operatorname{grad}u + n \gamma_{0}u_{t} + \gamma_{\nu} T \operatorname{grad}u = -h.
  \end{equation*}
  The case $n = 0$ then corresponds to Wentzell--Robin boundary conditions, cf.~\cite{Goldstein2006}.
\end{remark}
The formulations of the wave equation making use of formalisms \eqref{it:DivGradVersion1}, \eqref{it:DivGradVersion2} and \eqref{it:LaplaceVersion2} are cursorily discussed in \cref{app:Applications}.

\begin{example}[Maxwell equations]
  Maxwell's equations can be posed as a system using formulation \eqref{it:MaxwellVersion1} or \eqref{it:MaxwellVersion2}. Choosing the former, we obtain
  \begin{equation*}
    \left[\partial_{t}
    \begin{pmatrix} \epsilon & 0 & 0 \\ 0 & \mu & 0 \\ 0 & 0 & m \end{pmatrix}
    + \begin{pmatrix} \sigma & 0 & 0 \\ 0 & 0 & 0 \\ 0 & 0 & n \end{pmatrix}
    + \begin{pmatrix} 0 & - \begin{pmatrix} -\operatorname{curl} \\ \gamma_{\tau} \end{pmatrix}^{\!\!\ast}\\
    \begin{pmatrix} -\operatorname{curl} \\ \gamma_{\tau} \end{pmatrix} & 0
    \end{pmatrix}
    \right]
    \begin{pmatrix} u \\ v \\ w \end{pmatrix}
    = \begin{pmatrix} f \\ g \\ h \end{pmatrix},
  \end{equation*}
  where again $m,n \in \mathcal{L}\big(\leb^{2}_{\rho}(\mathbb{R};\leb^{2}_{\tau}(\partial\Omega))\big)$. This system describes Maxwell's equations with the boundary condition
  \begin{equation*}
    \partial_{t}mw + nw + \gamma_{\tau}u = h
  \end{equation*}
  as part of the system and the boundary condition
  \begin{equation*}
    w + \pi_{\tau}v = 0.
  \end{equation*}
  encoded in the spatial operator $A$.  Together they make up the boundary condition
  \begin{equation*}
    -\partial_{t}m \pi_{\tau}v - n\pi_{\tau}v + \gamma_{\tau}u = h.
  \end{equation*}
  For $m = 0$ this corresponds to classic Leontovich/impedance boundary conditions.\\
  Setup \eqref{it:MaxwellVersion2} similarly leads to the treatment of Leontovich boundary conditions as already described in \cite{PicardSeidlerTrostorffWaurick2016}, but in that instance with an operator $n$ acting on the twisted tangential trace $\gamma_{\tau}$ instead of the tangential trace $\pi_{\tau}$.
\end{example}

\section{Conclusion}

Well-posedness results addressing boundary value problems with state-dependent delay are few and far between. The majority of structural investigations into boundary value problems specifically seems to have been made in systems theory in the context of boundary control problems, e.g.\ \cite{Gantouh2026-1,Gantouh2026-2}. There, one studies an abstract Cauchy problem
\begin{equation*}
  \partial_{t}z(t) = Az(t) + BL(z_{(t)})
\end{equation*}
with a suitable initial condition. $A$ generates a strongly continuous semigroup on an extended state space, $B$ is the (boundary) control operator and $L$ a possibly nonlinear function receiving the history $z_{(t)}$ as input. Usually, $L$ is at least Lipschitz-continuous, but is often assumed to be a linear operator in the form of an integral over a measure of bounded variation. This is in some sense reminiscent of the theory developed in \cite{Batkai2005}, where delay equations were treated via perturbation theory for semigroups. This is also the idea in \cite{Diekmann1995}, where the technically more challenging $\odot$-calculus for semigroups is employed, see also \cite{Spek2020} for an application. All of these approaches have in common that they cannot accommodate state-dependent delay, which in its typical form appears as a delayed evaluation of the form
\begin{equation*}
  F(t,u_{(t)}) = u\big(t - \tau(u(t))\big).
\end{equation*}
The only attempts at structural solution theory for delay differential equations addressing state-dependent delay seem to stem from Hern\'{a}ndez, e.g.\ \cite{Hernandez2006}, and Rezounenko, e.g.\ \cite{Rezounenko2007}, as well as their coauthors. Neither of their approaches relate to delayed boundary value problems. In addition, their techniques are curated to take advantage of regularity results for the spatial operator $A$, which is usually the Dirichlet--Laplacian. Summarily, little general well-posedness theory seems to exist for boundary value problems with (state-dependent) delay, as the majority of the literature on delay differential equations is concerned with questions regarding the dynamic of solutions (e.g.\ stability, existence of attractors, bifurcation) of problems for which well-posedness is not an issue. For that reason, the articles the newly obtained results here were compared against do not relate to state-dependent delay at all, but to boundary value problems for the heat and wave equations with nonautonomous delay. These comparisons demonstrate the strength of the nonautonomous solution theory, which allows to address the boundary value problems for the heat and wave equation studied in said comparison articles in a systematic framework.\\
In the context of evolutionary equations, delayed partial differential equations were already considered in \cite{Trostorff2015}, but in the context of constant delay. Well-posedness theory for evolutionary equations with state-dependent delay was developed in \cite{AignerWaurick2026} and extended to include boundary value problems in this publication. The proposed techniques are highly applicable and general in scope. The approach presented here provides plenty opportunity for future topical research on specific applications, for instance by spelling out the precise assumptions on delay functionals, forcing terms and prehistories at hand.

\section*{Acknowledgements}
I, as the author of this article, received funding in form of a graduate student stipend of the federated German state Saxony.\\
I sincerely thank my supervisor Marcus Waurick for the many helpful suggestions during the genesis of this article and for pointing out the concept of abstract $\operatorname{div}$-$\operatorname{grad}$ systems. Additionally, I thank Nathanael Skrepek for sparking my interest in boundary value problems in the first place and the many helpful discussions during my visit in Enschede.\\
No data was used for this research.

\appendix
\section{Well-posedness for nonautonomous material laws}
\label{app:WellPosedness}

Here, full details for the proofs in \cref{subsec:WellPosedness} are provided. To that end, we recall two critical facts: The first one is the key ingredient for the solution theory in \cite{AignerWaurick2026,Waurick2023} and takes the form of a Lipschitz estimate for the delay operator.
\begin{lemma}[{\cite[thm.~2.2]{Waurick2023}, \cite[lem.~2.1]{Aigner2024}}]
  \label{th:ThetaIsLipschitz}
  For $0<T\leq \infty$ and $\rho>0$ let
  \begin{equation*}
    \Theta\colon\,\leb^2_{\rho}(-h,T;H)\rightarrow \leb^2_{\rho}\bigl(0,T;\leb^2(-h,0;H)\bigr)\text{,}\quad
    u\mapsto (t\mapsto u_{(t)})\text{.}
  \end{equation*}
  Then $\norm{\Theta}\leq \frac{1}{\sqrt{2\rho}}$.
\end{lemma}
The second fact is an integrability condition for the right-hand side $F$ from \cref{eq:EvolEqSDD}, mandated by the fact that evolutionary equations are posed on the entire real line $\mathbb{R}$. Here, as in \cref{subsec:WellPosedness}, $\pi_{\alpha}\colon \sobh^{1}(-h,0;H)\to V_{\alpha}$ denotes the metric projection onto the subset $V_{\alpha}$ from \cref{def:AlmLC}.
\begin{lemma}[{\cite[lem.~3.4]{AignerWaurick2026}}]
  \label{th:integrability}
  Let $\rho > 0$, $v\in \sobh^{1}_{\rho}(-h,\infty;H)$ and let $F \colon [0,\infty) \!\times\! \leb^2(-h,0;H) \to H$ be almost uniformly Lipschitz-continuous. Suppose $F(\argdot,0)\in \leb^2_{\rho}(0,\infty;H)$. Then for all $\alpha>0$, $F\bigl(\argdot, \pi_{\alpha}(v_{(\argdot)})\bigr)\in \leb^2_{\rho}(0,\infty;H)$.
\end{lemma}

We now want to show the existence of a fixed point for the map
\begin{align}
  \begin{aligned}
    \label{eq:FPPproj}
    \Gamma_{\rho,\alpha}\colon \, \leb^2_{\rho}(0,\infty;H)&\rightarrow \leb^2_{\rho}(0,\infty;H)\text{,}\\
    v &\mapsto S_{\rho}f + S_{\rho}F\bigl(\argdot,\pi_{\alpha}((v+Z)_{(\argdot)})\bigr),
  \end{aligned}
\end{align}
where we assume that $F$ satisfies $F(\argdot,0)\in \leb^2_{\rho}(0,\infty;H)$ and the function on the right is therefore in $\leb^2_{\rho}(0,\infty;H)$ by \cref{th:integrability}. We now provide full proofs for both cases considered in \cref{subsec:WellPosedness} and make use of the following notation:
\begin{equation*}
  \sobh^{1}_{0,\rho}(0,\infty;H)\coloneq \dset{w \in \sobh^{1}_{\rho}(0,\infty;H)}{w(0)=0}.
\end{equation*}

\subsubsection*{The case $\widetilde{F}(v)=I_{\rho}\widetilde{G}(v)$}
\begin{theorem}[local existence and uniqueness for \cref{eq:FPPIntOut}]
  \label{th:localExistOutApp}
  Let $G$ satisfy assumption \ref{ass:A} and let $\Phi$ and $Z$ satisfy assumption \ref{ass:B}. Let $A$, $M$, $M'$ and $N$ satisfy the assumptions of \cref{th:Picard} and \cref{th:PicardHigh} for $k = 1$ and $k=2$. Let $f \in \sobh_{\tilde{\rho}}^{2}(0,\infty;H)\cap \sobh^{1}_{0,\tilde{\rho}}(0,\infty;H)$ for some $\tilde{\rho} >0$ and $G(0,\Phi) = - f'(0)$. Then \cref{eq:FPPIntOut} has a unique local solution $u\in \sobh^{2}(0,T;H)$.
\end{theorem}

\begin{proof}
  Let $\alpha > \norm{\Phi'}_{\infty}$ and $\rho > \max\bigl\{\tfrac{L_{\alpha}^{2}}{2 c^{2}},\tilde{\rho}\bigr\}$, where $L_{\alpha}$ is the Lipschitz-constant of $G$ associated to the parameter $\alpha$, cf.~\cref{def:AlmLC}, and $c$ is the minimum of the coercivity bounds from \cref{th:PicardHigh} (for $k = 1$ and $k = 2$) and \cref{th:Picard}. We first show the existence of a fixed point of
  \begin{multline*}
    v = S_{\rho}\bigl[f + I_{\rho}G\bigl(\argdot,\pi_{\alpha}((v+Z)_{(\argdot)})\bigr)\bigr] = S_{\rho}\bigl[f + I_{\rho}\widetilde{G}_{\alpha}(v)\bigr]\eqcolon \Gamma_{\rho,\alpha}(v)\text{,}\\
    v\in \sobh^{1}_{0,\rho}(0,\infty;H)\text{,}
  \end{multline*}
  by showing that $\Gamma_{\rho,\alpha}$ satisfies the requirements of the contraction mapping principle:
\begin{itemize}[leftmargin=3ex]
  \item $\Gamma_{\rho,\alpha}$ is a self-mapping:\\
        Since $\pi_{\alpha}$ maps into $V_{\alpha}$ and $G$ satisfies assumption \ref{ass:A} iii), \cref{th:integrability} ensures that $\widetilde{G}_{\alpha}(v)$ defines an element of $\leb^2_{\rho}(\mathbb{R};H)$. As $I_{\rho}\colon \leb^2_{\rho}\to \sobh^{1}_{\rho}$, we have $I_{\rho}\widetilde{G}_{\alpha}(v) \in \sobh^{1}_{\rho}(\mathbb{R};H)$. Since $f\in \sobh^{1}_{0,\rho}(\mathbb{R};H)$, $t\mapsto \bigl[f + I_{\rho}\widetilde{G}_{\alpha}(v)\bigr]$ defines an element of $\sobh^{1}_{0,\rho}(0,\infty;H)$. Because of \cref{th:PicardHigh} for $k=1$, the expression $\Gamma_{\rho,\alpha}(v)$ defines an element of $\sobh^{1}_{\rho}(\mathbb{R};H)$. Appealing to causality of $S_{\rho}$ we infer $\Gamma_{\rho,\alpha}(v)\in \sobh^{1}_{0,\rho}(0,\infty;H)$.
  \item $\Gamma_{\rho,\alpha}$ is Lipschitz-continuous:\\
        We calculate for $v,w \in \sobh^{1}_{0,\rho}(0,\infty;H)$:
        \begin{align*}
          &\norm{\Gamma_{\rho,\alpha}v-\Gamma_{\rho,\alpha}w}_{\sobh^{1}_{\rho}(\mathbb{R};H)}^{2}\\
          &\quad= \norm{S_{\rho}I_{\rho}\widetilde{G}_{\alpha}(v) - S_{\rho}I_{\rho}\widetilde{G}_{\alpha}(w)}_{\sobh^{1}_{\rho}}^{2}\\
          \intertext{Using well-posedness in $\sobh^{1}_{\rho}$ appealing to \cref{th:PicardHigh} we estimate}
          &\quad\leq \norm{S_{\rho}}^{2}\norm{I_{\rho}\widetilde{G}_{\alpha}(v) - I_{\rho}\widetilde{G}_{\alpha}(w)}_{\sobh^{1}_{\rho}}^{2}\\
          &\quad\leq \tfrac{1}{c^{2}}\norm{\partial_{t}I_{\rho}\widetilde{G}_{\alpha}(v) - \partial_{t}I_{\rho}\widetilde{G}_{\alpha}(w)}_{\leb^2_{\rho}}^{2}\\
          \intertext{Using almost uniform Lipschitz-continuity of $G$ (note that $\alpha > \norm{\Phi'}_{\infty}$) and the Lipschitz-continuity of the projection $\pi_{\alpha}$, we can estimate further}
          &\quad\leq \tfrac{1}{c^{2}}\medint\int_0^\infty \norm[\big]{G\bigl(t,\pi_{\alpha}((v+Z)_{(t)})\bigr)
            -G\bigl(t,\pi_{\alpha}((w+Z)_{(t)})\bigr)}_{H}^{2}\e^{-2\rho t}\dx[t]\\
          &\quad\leq \tfrac{L_{\alpha}^{2}}{c^{2}}\medint\int_0^\infty \norm[\big]{\pi_{\alpha}\bigl((v+Z)_{(t)}\bigr)
            -\pi_{\alpha}\bigl((w+Z)_{(t)}\bigr)}_{\sobh^{1}(-h,0;H)}^{2}\e^{-2\rho t}\dx[t]\\
          &\quad\leq \tfrac{L_{\alpha}^{2}}{c^{2}}\medint\int_{0}^{\infty}\norm[\big]{(v+Z)_{(t)}
            -(w+Z)_{(t)}}_{\sobh^{1}(-h,0;H)}^{2}\e^{-2\rho t}\dx[t]\\
          &\quad= \tfrac{L_{\alpha}^{2}}{c^{2}}\medint\int_0^\infty \norm{v_{(t)}
            -w_{(t)}}_{\sobh^{1}(-h,0;H)}^{2}\e^{-2\rho t}\dx[t]\\
          &\quad=\tfrac{L_{\alpha}^{2}}{c^{2}}\norm{\Theta v
            - \Theta w}_{\leb^2_{\rho}(0,\infty;\sobh^{1}(-h,0;H))}^{2}\\
          \intertext{An application of the Lipschitz-continuity of the delay operator $\Theta$ in the form of \cref{th:ThetaIsLipschitz} yields the desired result:}
          &\quad\leq \tfrac{L_{\alpha}^{2}}{c^{2}} \tfrac{1}{2\rho}
            \norm{v-w}_{\sobh_{\rho}^{1}(-h,\infty;H)}^{2}\text{.}
        \end{align*}
  \item $\Gamma_{\rho,\alpha}$ is a contraction on $\sobh^{1}_{0,\rho}(0,\infty;H)$:\\
        This follows from the last estimate and the assumption $\rho > \tfrac{L_{\alpha}^{2}}{2 c^{2}}$.
\end{itemize}
The contraction mapping principle hence assures a solution $w$ of
\begin{equation*}
  w = \Gamma_{\rho,\alpha}(w)\text{,} \quad w\in \sobh^{1}_{0,\rho}(0,\infty;H)\text{.}
\end{equation*}
To obtain a local solution of \cref{eq:FPPIntOut}, we need to argue that at least up to some $T>0$, we have $\pi_{\alpha}\bigl((w + Z)_{(t)}\bigr) = (w + Z)_{(t)}$ for $0 \leq t \leq T$.\\
First recall that $f\in \sobh^{2}_{\rho}(0,\infty;H)$, and because $G$ is regularity preserving, $\widetilde{G}_{\alpha}(w) \in \sobh^{1}_{\rho}(0,\infty;H)$ and consequently $I_{\rho}\widetilde{G}_{\alpha}(w) \in \sobh^{2}_{\rho}(0,\infty;H)$. Since $f'(0) + G(0,\Phi) = 0$, the sum $f + I_{\rho}\widetilde{G}_{\alpha}(w) \in \sobh^{2}_{\rho}(\mathbb{R};H)$. Well-posedness in $\sobh^{2}_{\rho}(\mathbb{R};H)$ appealing to \cref{th:PicardHigh} for $k =2$ then assures that in fact $w \in \sobh^{2}_{\rho}(\mathbb{R};H)$. From causality of $S_{\rho}$ we infer that $\mathrm{spt}(w) \subseteq [0,\infty)$. Because $w \in \sobh^{2}_{\rho}(\mathbb{R};H)$; in particular this implies $w'(t) = 0$ for $t\leq 0$ and hence $w' \in \sobh^{1}_{0,\rho}(0,\infty;H)$.\\
Since $w = 0$ on $(-\infty,0]$ and $\Phi \in V_{\alpha}$, it only needs to be shown that $w' + Z'$ is bounded in $\norm{\argdot}_{\infty}$-norm by $\alpha$ up to some positive time $T>0$. Since $Z'$ is bounded on any interval $[0,c]$, $c>0$, and $\norm{Z'(0)}_{H}< \alpha$ by assumption \ref{ass:B} ii), $\norm{Z'}_{\mathcal{C}([0,\tilde{T}];H)}<\alpha$ holds up to some positive $\tilde{T}>0$. Since $w'$ is continuous and $w'(0)=0$, up to some positive $0<T\leq \tilde{T}$, $w'+Z'$ remains bounded in $\norm{\argdot}_{\infty}$-norm by $\alpha$. This establishes existence of a local solution.\\
For uniqueness, we let $v \in \leb^2_{\rho}(0,\infty;H)$ be a local solution of \cref{eq:FPPIntOut}, that is by definition, a fixed point of $\Gamma_{\hat{\rho},\alpha}$ from \eqref{eq:FPPproj} for some $\hat{\rho}>0$, i.e.,
\begin{align*}
  \begin{aligned}
    \Gamma_{\hat{\rho},\alpha}\colon \, \leb^2_{\hat{\rho}}(0,\infty;H)&\rightarrow \leb^2_{\hat{\rho}}(0,\infty;H)\text{,}\\
    v &\mapsto S_{\hat{\rho}}f + S_{\hat{\rho}}I_{\hat{\rho}}\widetilde{G}_{\alpha}(v)\text{.}
  \end{aligned}
\end{align*}
As proven at the start of the proof, a fixed point $v$ is an element of $\sobh^{1}_{0,\hat{\rho}}(0,\infty;H)$ by virtue of the shape of the right-hand side. Uniqueness of the fixed point now follows from the contraction mapping principle and eventual independence of the weight parameters $\hat{\rho}$ and $\rho$, appealing to \cref{th:Picard}.
\end{proof}

\begin{remark}[maximal existence interval]
  \label{rmk:maxExistIntOut}
  A closer inspection of the proof reveals that the unique local solution $u = v + Z$ of \cref{eq:FPPIntOut} exists up to the smallest $T>0$ at which $\lim_{t\uparrow T}\norm{u'}_{\mathcal{C}([0,t];H)}=+\infty$:\\
  Indeed, in the proof of \cref{th:localExistOut}, the solution of the fixed point problem with $\pi_{\alpha}$,
  \begin{equation*}
    v = \Gamma_{\rho, \alpha}(v) = S_{\rho}\bigl[f + I_{\rho}G\bigl(\argdot, \pi_{\alpha}\bigl((v+Z)_{(\argdot)}\bigr)\bigr)\bigr]\text{,} \qquad v\in \sobh^{1}_{0,\rho}(0,\infty;H)
  \end{equation*}
  is a local solution of \cref{eq:FPPIntOut} for as long as the $\norm{\argdot}_{\infty}$-norm of $v' + Z'$ remains bounded by $\alpha$, say up to $T_{0}$. Increasing the parameter $\alpha$ to $\alpha + 1$ yields a solution $w$ of $w = \Gamma_{\rho, \alpha +1}(w)$, such that $w'+Z'$ remains bounded by $\alpha +1$ up to $T_{1}>T_{0}$ by virtue of the continuity of $w'+Z'$ (cf.\ proof of \cref{th:localExistOutApp}). We observe that the solutions $v$ (corresponding to $\alpha$) and $w$ (corresponding to $\alpha + 1$) agree up to $T_{0}$ by virtue of unique solvability. The claim follows by iteration and the fact, that $\norm{Z'}_{\mathcal{C}([0,t];H)}<\infty$ for all $t>0$.
\end{remark}

\subsubsection*{The case $\widetilde{F}(v)=\widetilde{G}(I_{\rho}v)$}
\begin{theorem}[local existence and uniqueness for \cref{eq:FPPIntIns}]
  Let $G$ satisfy assumption \ref{ass:A} and let $\Phi$ and $Z$ satisfy  assumption \ref{ass:B}. Let $M, M', N$ and $A$ satisfy the assumptions of \cref{th:Picard} and \cref{th:PicardHigh} for $k=1$. Let $f\in \sobh^{1}_{\tilde{\rho}}(0,\infty;H)$ for some $\tilde{\rho}>0$ and $f(0) = -G\bigl(0, \Phi\bigr)$. Then \cref{eq:FPPIntIns} has a unique local solution $u \in \sobh^{1}(0,T;H)$.
\end{theorem}
\begin{proof}
  Let $\alpha > \norm{\Phi'}_{\infty}$ and $\rho > \max \bigl\{1, \tilde{\rho}, \tfrac{2 c^{2}}{L_{\alpha}^{2}(h^{2}+1)}\bigr\}$, where $L_{\alpha}$ is the Lipschitz-constant of $G$ associated to the parameter $\alpha$, cf.~\cref{def:AlmLC}, and $c$ is the minimum of the coercivity constants from \cref{th:Picard} and \cref{th:PicardHigh} for $k=1$. We start with a preliminary observation:\\
  Let $-\infty < a < b < \infty$ and $f \in \leb^2(a,b;H)$. Then we can estimate:
  \begin{align}
    \begin{aligned}
      \medint\int_{a}^{b} \norm[\Big]{\medint\int_{a}^{t}f(s)\dx[s]}^{2}\dx[t]
      &\leq \medint\int_{a}^{b} (t-a) \medint\int_{a}^{t}\norm{f(s)}^{2}\dx[t]
      = \medint\int_{a}^{b} \medint\int_{s}^{b}\underbrace{(t-a)}_{\leq b-a} \norm{f(s)}^{2}\dx[t]\dx[s]\\[-1.5ex]
      &\leq (b-a)^{2}\medint\int_{a}^{b}\norm{f(s)}^{2}\dx[s]\text{.}
    \end{aligned} \tag{$\square$}
  \end{align}
  We begin the actual proof by first showing the existence of a fixed point of
  \begin{equation*}
    v = S_{\rho}\bigl[f + G\bigl(\argdot,\pi_{\alpha}\bigl((I_{\rho}v+Z)_{(\argdot)}\bigr)\bigr)\bigr] = S_{\rho}\bigl[f+ \widetilde{G}_{\alpha}(I_{\rho}v)\bigr]\eqcolon \Gamma_{\rho,\alpha}(v)
    \quad v\in \leb^2_{\rho}(0,\infty;H)\text{,}
  \end{equation*}
  by verifying that $\Gamma_{\rho,\alpha}$ satisfies the requirements of the contraction mapping principle:
\begin{itemize}[leftmargin=3ex]
  \item $\Gamma_{\rho,\alpha}$ is a self-mapping:\\
        This follows from \cref{th:integrability}, the assumptions on $f$ and $G$ and that the solution operator is a bounded linear operator on $\leb^2_{\rho}$ (cf.~\cref{th:Picard}).
  \item $\Gamma_{\rho,\alpha}$ is Lipschitz-continuous:\\
        For $v,w \in \leb^2_{\rho}(0,\infty;H)$ we can compute:
        \begin{align*}
          &\norm{\Gamma_{\rho,\alpha}v-\Gamma_{\rho,\alpha}w}_{\leb^2_{\rho}(0,\infty;H)}^{2}\\
          &\quad= \norm[\big]{S_{\rho}\widetilde{G}_{\alpha}(I_{\rho}v) -S_{\rho}\widetilde{G}_{\alpha}(I_{\rho}w)}_{\leb^2_{\rho}(0,\infty;H)}^{2}\\
          &\quad\leq \norm{S_{\rho}}^{2} \medint\int_0^\infty \norm[\big]{G\bigl(t,\pi_{\alpha}\bigl((I_{\rho}v+Z)_{(t)}\bigr)\bigr) -G\bigl(t,\pi_{\alpha}\bigl((I_{\rho}w+Z)_{(t)}\bigr)\bigr)}_{H}^{2}\e^{-2\rho t}\dx[t]\\
          \intertext{Using almost uniform Lipschitz-continuity of $G$ and Lipschitz-continuity of $\pi_{\alpha}$ we estimate further:}
          &\quad\leq \tfrac{L_{\alpha}^{2}}{c^{2}}\medint\int_0^\infty \norm[\big]{\pi_{\alpha}\bigl((I_{\rho}v+Z)_{(t)}\bigr) -\pi_{\alpha}\bigl((I_{\rho}w+Z)_{(t)}\bigr)}_{\sobh^{1}(-h,0;H)}^{2}\e^{-2\rho t}\dx[t]\\
          &\quad\leq \tfrac{L_{\alpha}^{2}}{c^{2}}\medint\int_{0}^{\infty}\norm[\big]{(I_{\rho}v+Z)_{(t)} -(I_{\rho}w+Z)_{(t)}}_{\sobh^{1}(-h,0;H)}^{2}\e^{-2\rho t}\dx[t]\\
          \intertext{Since $I_{\rho}$ and the delay operator interchange, we can equate:}
          &\quad=  \tfrac{L_{\alpha}^{2}}{c^{2}} \medint\int_{0}^{\infty}\norm[\Big]{\smallint_{-h}^{\argdot}\bigl(v_{(t)}\bigr) -\smallint_{-h}^{\argdot}\bigl(w_{(t)}\bigr)}_{\sobh^{1}(-h,0;H)}^{2}\e^{-2\rho t}\dx[t]\\
          &\quad=  \tfrac{L_{\alpha}^{2}}{c^{2}} \medint\int_{0}^{\infty}\Big[\norm[\big]{v_{(t)} -w_{(t)}}_{\leb^{2}(-h,0;H)}^{2} +\norm[\Big]{\smallint_{-h}^{\argdot}\bigl(v_{(t)} -w_{(t)}\bigr)}_{\leb^{2}(-h,0;H)}^{2}\Big]\e^{-2\rho t}\dx[t]\\
          \intertext{We make use of the preliminary observation ($\square$) and obtain:}
          &\quad\leq  \tfrac{L_{\alpha}^{2}}{c^{2}} \medint\int_{0}^{\infty}\Big[\norm[\big]{v_{(t)} -w_{(t)}}_{\leb^{2}(-h,0;H)}^{2} +h^{2}\norm[\big]{v_{(t)} -w_{(t)}}_{\leb^{2}(-h,0;H)}^{2}\Big]\e^{-2\rho t}\dx[t]\\
          &\quad= \tfrac{L_{\alpha}^{2}}{c^{2}} (h^{2} + 1) \medint\int_{0}^{\infty} \norm{v_{(t)}-w_{(t)}}_{\leb^{2}(-h,0;H)}^{2}\e^{-2\rho t}\dx[t]\\
          &\quad= \tfrac{L_{\alpha}^{2}}{c^{2}}(h^{2} + 1) \norm{\Theta v -\Theta w}_{\leb^{2}_{\rho}(0,\infty;\leb^{2}(-h,0;H))}^{2}\\
          \intertext{Now we can appeal to Lipschitz-continuity of $\Theta$, cf.~\cref{th:ThetaIsLipschitz}:}
          &\quad\leq \tfrac{L_{\alpha}^{2}}{c^{2}}(h^{2}+1)\tfrac{1}{2\rho} \norm{v-w}_{\leb^{2}_{\rho}(-h,\infty;H)}^{2}\text{.}
        \end{align*}
  \item $\Gamma_{\rho,\alpha}$ is a contraction:\\
        This follows from the last estimate and the assumption $\rho > \max \bigl\{1,\tfrac{2 c^{2}}{L_{\alpha}^{2}(h^{2}+1)}\bigr\}$.
\end{itemize}
The contraction mapping principle therefore provides a solution $w$ of
\begin{equation*}
  w = \Gamma_{\rho,\alpha}(w)\text{,}\qquad w\in \leb^2_{\rho}(0,\infty;H)\text{.}
\end{equation*}
We again argue, that up to some $T>0$, we have $\pi_{\alpha}\bigl((I_{\rho}w + Z)_{(t)}\bigr) = \bigl(I_{\rho}w + Z\bigr)_{(t)}$.\\
Because of the integral term $I_{\rho}$, it suffices to verify that the term $w + Z'$ is bounded in $\norm{\argdot}_{\infty}$-norm by $\alpha$ up to some positive time $T>0$:
Appealing to the regularity preserving property of $G$ and the fact that $f\in \sobh^{1}_{\rho}(0,\infty;H)$ by assumption, we know $f + G\bigl(\argdot,\pi_{\alpha}\bigl((I_{\rho}w + Z)_{(\argdot)}\bigr)\bigr) \in \sobh^{1}_{\rho}(0,\infty;H)$. We further observe:
\begin{align*}
  \lim_{t\downarrow 0} G\bigl(t,\pi_{\alpha}\bigl((I_{\rho}w+Z)_{(t)}\bigr)\bigr)
  &= G\bigl(0, \pi_{\alpha}\bigl((I_{\rho}w + Z)_{(0)}\bigr)\bigr)\\[-1.5ex]
  &= G\bigl(0, \pi_{\alpha}\bigl((0 + Z)_{(0)} \bigr)\bigr)\\
  &= G(0, \Phi )\text{.}
\end{align*}
The consistency condition therefore assures
\begin{equation*}
  f + G\bigl(\argdot,\pi_{\alpha}\bigl((I_{\rho}w+Z)_{(\argdot)}\bigr)\bigr) \in \sobh^{1}_{0,\rho}(0,\infty;H)\text{.}
\end{equation*}
By virtue of \cref{th:PicardHigh}, the solution $w$ of $\Gamma_{\rho,\alpha}(w) = w$ is in fact in $\sobh^{1}_{\rho}$ and because $S_{\rho}$ is causal, we can infer $w\in \sobh^{1}_{0,\rho}(0,\infty;H)$. Hence, $\norm{w(t)}_{H}< \alpha - \norm{\Phi'}_{\infty}$ up to some positive time $\tilde{T}>0$. Since $Z' \in \mathcal{C}([0,\tilde{T}];H)$ and $\norm{Z'(0)}_{H}\leq \norm{\Phi'}_{\infty}$ by assumption, continuity provides a positive $0<T\leq \tilde{T}$ such that $\norm{w + Z'}_{\mathcal{C}([-h,T];H)}\leq \alpha$.\\
Uniqueness of the solution follows similarly as in the proof of \cref{th:localExistOut} from the uniqueness of the fixed point of $\Gamma_{\rho,\alpha}$ by appealing to the contraction mapping principle and eventual independence of the weight parameter $\rho$ appealing to \cref{th:Picard}.
\end{proof}

\begin{remark}[maximal existence interval]
  \label{rmk:maxExistIntIns}
  A closer inspection of the proof reveals that the unique local solution $u = I_{\rho}v + Z$ of \cref{th:localExistIns} exists up to the smallest $T>0$ at which $\lim_{t\uparrow T}\norm{u}_{\mathcal{C}([0,t];H)}=+\infty$:\\
  Indeed, in the proof of \cref{th:localExistIns}, the solution of the fixed point theorem with $\pi_{\alpha}$,
  \begin{equation*}
    v = \Gamma_{\rho, \alpha}(w) = S_{\rho}\bigl[f + G\bigl(\argdot, \pi_{\alpha}\bigl((I_{\rho}v +Z)_{(\argdot)}\bigr)\bigr)\bigr]\text{,} \qquad v \in \leb^2_{\rho}(0,\infty;H)\text{,}
  \end{equation*}
  is a local solution of \cref{eq:FPPIntIns} for as long as the $\norm{\argdot}_{\infty}$-norm of $v$ remains bounded by $\alpha$. The claim follows as in the case of \cref{eq:FPPIntOut}, cf.~\cref{rmk:maxExistIntOut}, replacing $v'$ and $w'$ with $v$ and $w$ respectively.
\end{remark}

\begin{remark}
  The following differences to \cite[thm.~4.1\&4.2]{AignerWaurick2026}, i.e.\ the autonomous versions of the above theorems, should be noted:
  \begin{itemize}[leftmargin = 4ex]
    \item In the nonautonomous case, regularity theory for the solution operator in the form $\partial_{t}S_{\rho}\supseteq S_{\rho}\partial_{t}$ is not available and compensated for by the coercivity assumptions in $\sobh^{k}_{\rho}$. These are automatically satisfied in the autonomous case and therefore represent no additional requirement in the autonomous case.
    \item In \cref{th:localExistOutApp}, well-posedness in $\sobh^{2}_{\rho}$ in addition to well-posedness in $\sobh^{1}_{\rho}$ was required. This is a strong assumption mandated by the mathematical difficulty of not being able to interchange $S_{\rho}$ and $\partial_{t}$ in the nonautonomous case and infer $w'(0) = 0$ that way in the proof of \cref{th:localExistOutApp}. It is possible to calculate the commutator
          \begin{equation*}
            \big[\partial_{t},S_{\rho}\big] = \partial_{t}S_{\rho}I_{\rho}(\partial_{t}M'+N')I_{\rho}S_{\rho}\partial_{t},
          \end{equation*}
          where $M' = [\partial_{t},M]$ and $N' = [\partial_{t},N]$, but this unfortunately proves unhelpful in showing $w'(0) = 0$, which is required for the continuity argument in the proof.
  \end{itemize}
\end{remark}

\section{Convolutions}
\label{app:Convolutions}
This segment serves as a brief reminder on convolution operators arising from a real-valued kernel $\mu$; more specifically, assume that
\begin{itemize}[leftmargin = 4ex]
  \item $\mu \colon \mathbb{R}_{\geq 0}\to \mathbb{R}_{\geq 0}$ is a locally absolutely continuous function,
  \item $\medint\int_{0}^{\infty}\mu(s)\dx[s] \eqcolon \overline{\mu} <1$ and
  \item $\medint\int_{0}^{\infty}\abs{\mu'(s)}\dx[s] < \infty$.
\end{itemize}

To study convolution operators, it is useful to utilize Fourier methods. Since we use the exponentially weighted spaces $\leb^{2}_{\rho}(\mathbb{R};H)$, we employ the {\em Fourier-Laplace transform} $\mathcal{L}_{\rho}$, which is the unitary extension of
\begin{align}
  \label{def:FourierLaplace}
  \begin{aligned}
  \mathring{\mathcal{C}}^{\infty}(\mathbb{R};H) &\to \mathcal{C}^{\infty}(\mathbb{R};H)\text{,}\\
  \phi &\mapsto \Bigl(t\mapsto \tfrac{1}{\sqrt{2\pi}} \medint\int_{\mathbb{R}} \e^{-(\iu t + \rho)s} \phi(s)\dx[s]\Bigr)
  \end{aligned}
\end{align}
to a map $\mathcal{L}_{\rho}\colon \leb^{2}_{\rho}(\mathbb{R};H)\to \leb^{2}(\mathbb{R};H)$, cf.~\cite[ch.~5.2]{STW2022}. It adheres to standard calculation rules; in particular,
\begin{enumerate}[label = (\roman*), leftmargin = 5ex]
  \item\label{it:Useful1} $\forall f\in \leb^{2}_{\rho}(\mathbb{R},H)\;\forall k \in \leb^{1}_{\rho}(\mathbb{R})\colon \flt(f\ast k) = \sqrt{2\pi} \flt(f)\flt(k)$.
  \item\label{it:Useful2} $\forall f \in \leb^{2}_{\rho}(\mathbb{R};H)\colon \overline\flt(f) = \mathcal{L}_{-\rho}(f(-\argdot))$.
  \item\label{it:Useful3} $\forall k \in \leb^{1}_{\rho}(\mathbb{R};H)\colon \norm{\flt(k)}_{\infty}\leq (2\pi)^{-\nicefrac{1}{2}}\norm{k}_{\leb^{1}_{\rho}}$.
  \item\label{it:Useful4} $\forall k \in \sobh^{1}_{\rho}(\mathbb{R};H)\colon \flt(k') = (\iu \mathrm{m} + \rho)\flt(k)$.\footnote{Here, $\mathrm{m}$ denotes the multiplication with the argument operator.}
\end{enumerate}

\begin{lemma}
  Let $\rho >0$. Then $(1 - \mu \ast)$ and $(1 - \mu \ast)^{-1}$ define elements of $\mathcal{L}\big(\leb^{2}_{\rho}(\mathbb{R};H)\big)$. Moreover, there exists $c>0$ such that
  \begin{alignat*}{2}
    \forall \varphi \in \leb^{2}_{\rho}(\mathbb{R};H)\colon &\Re \dualprod[\big]{(1-\mu\ast)\varphi}{\varphi}_{\leb^{2}_{\rho}(\mathbb{R};H)}&&\geq c \norm{\varphi}_{\leb^{2}_{\rho}(\mathbb{R};H)}^{2},\\
    &\Re \dualprod[\big]{(1-\mu\ast)^{-1}\varphi}{\varphi}_{\leb^{2}_{\rho}(\mathbb{R};H)}&&\geq \tfrac{c}{\norm{(1-\mu\ast)}^{2}_{\mathcal{L}(\leb^{2}_{\rho}(\mathbb{R};H))}} \norm{\varphi}_{\leb^{2}_{\rho}(\mathbb{R};H)}^{2}.
  \end{alignat*}
\end{lemma}

\begin{proof}
  Since $\mathrm{spt}(\mu)\subseteq [0,\infty)$, it is clear that $(1 - \mu\ast) \in \mathcal{L}\big(\leb^{2}_{\rho}(\mathbb{R};H)\big)$; in particular, this implies $(1 - \mu\ast)^{\ast} \in \mathcal{L}(\leb^{2}_{\rho}(\mathbb{R};H))$. We show the first coercivity estimate: Let $\varphi \in \leb^{2}_{\rho}(\mathbb{R};H)$,
  \begin{align*}
    \Re \dualprod[\big]{(1-\mu\ast)\varphi}{\varphi}_{\leb^{2}_{\rho}}
    &= \norm{\varphi}_{\leb^{2}_{\rho}}^{2} - \Re \dualprod[\big]{\flt(\mu \ast \varphi)}{\flt(\varphi)}_{\leb^{2}}\\
    &\overset{\ref{it:Useful1}}{=} \norm{\varphi}_{\leb^{2}_{\rho}}^{2} - \Re \sqrt{2\pi}\dualprod[\big]{\flt(\mu) \flt(\varphi)}{\flt(\varphi)}_{\leb^{2}}\\
    &\overset{\ref{it:Useful3}}{\geq} \norm{\varphi}_{\leb^{2}_{\rho}}^{2}(1- \norm{k}_{\leb^{1}_{\rho}})
      \geq \underbrace{(1-\overline{\mu})}_{\eqcolon c>0}\norm{\varphi}_{\leb^{2}_{\rho}}^{2}.
  \end{align*}
  Fact \ref{it:Useful2} and a similar calculation prove
  \begin{equation*}
    \Re \dualprod[\big]{(1-\mu\ast)^{\ast}\varphi}{\varphi}_{\leb^{2}_{\rho}} \geq c \norm{\varphi}_{\leb^{2}_{\rho}}^{2}.
  \end{equation*}
  Consequently, $(1 - \mu\ast)$ is continuously invertible by \cite[prop.~6.3.1]{STW2022}. The second claimed coercivity follows easily by calculation.
\end{proof}

\begin{lemma}
  Let $\rho>0$. Then $\operatorname{ran}(\mu\ast) \subseteq \sobh^{1}_{\rho}(\mathbb{R};H)$ and for any $u \in \leb^{2}_{\rho}(\mathbb{R};H)$, $\partial_{t}(\mu \ast u) = \mu(0)u + \mu'\ast u$.
\end{lemma}
\begin{proof}
  Let $\varphi \in \mathring{\mathcal{C}}^{\infty}(\mathbb{R})$. We test with $\varphi'$,
  \begin{align*}
    \medint\int_{-\infty}^{\infty} (\mu\ast u)(t)\varphi'(t)\dx[t]
    &= \medint\int_{-\infty}^{\infty}\medint\int_{-\infty}^{t}\mu(t-s)u(s)\dx[s]\varphi'(t)\dx[t]\\
    &= \medint\int_{-\infty}^{\infty} - \bigg[\mu(0) u(t) + \medint\int_{-\infty}^{t}\mu'(t-s)u(s)\dx[s]\bigg]\varphi(t)\dx[t]\\
    &= -\medint\int_{-\infty}^{\infty}\mu(0)u(t)\varphi(t) + (\mu'\ast u)(t)\varphi(t)\dx[t].
  \end{align*}
  Hence, the distributional derivative of $(\mu\ast u)$ is $\mu(0)u + \mu'\ast u$. This derivative defines a $\leb^{2}_{\rho}$-function assuming that $\mu' \in \leb^{1}(0,\infty)$.
\end{proof}

\begin{lemma}
  \label{th:FLTMu}
  For $\rho > 0$ large enough and for any $\psi \in \leb^{2}_{\rho}(\mathbb{R};H)$,
  \begin{equation*}
    \flt \big((1-\mu\ast)^{-1}\psi\big) = \frac{\flt(\psi)}{(1-\sqrt{2\pi})\flt(\mu)}.
  \end{equation*}
\end{lemma}
\begin{proof}
  Let $\varphi \in \leb^{2}_{\rho}(\mathbb{R};H)$ and set $\psi \coloneq (1- \mu \ast)\varphi$. Then
  \begin{equation*}
    \flt(\psi) = \flt(\varphi) - \sqrt{2\pi}\flt(\mu)\flt(\varphi) = \big(1-\sqrt{2\pi}\flt(\mu)\big)\flt(\varphi).
  \end{equation*}
  Note that $\flt(\mu)$ is pointwise bounded by $(2\pi)^{-\nicefrac{1}{2}}\norm{\mu}_{\leb^{1}_{\rho}(\mathbb{R};H)}$, which converges to $0$ as $\rho \to \infty$. From the same calculation we infer $\flt(\varphi) = \big(1-\sqrt{2\pi}\flt(\mu)\big)^{-1}\flt(\psi)$.
\end{proof}

\begin{lemma}
  \label{th:Convolution}
  For $\rho> 0$ large enough, $\operatorname{ran}\big((1-\mu\ast)^{-1}\big)\subseteq \sobh^{1}_{\rho}(\mathbb{R};H)$. Moreover, there exists $c>0$ such that
  \begin{equation*}
    \forall \varphi \in \leb^{2}_{\rho}(\mathbb{R};H)\colon \Re \dualprod[\big]{\partial_{t}(1-\mu\ast)^{-1}\varphi}{\varphi}_{\leb^{2}_{\rho}(\mathbb{R};H)}\geq (\rho - c)\norm{\varphi}_{\leb^{2}_{\rho}(\mathbb{R};H)}^{2}.
  \end{equation*}
\end{lemma}
\begin{proof}
  We initially observe:
  \begin{align*}
    \flt(\mu) &= \medint\int_{0}^{\infty} \e^{-(\iu t +\rho)s}\mu(s)\dx[s]\\
              &= -\tfrac{1}{\iu t + \rho}\medint\int_{0}^{\infty}-(\iu t + \rho)\e^{-(\iu t +\rho)s}\mu(s)\dx[s]\\
              &= -\tfrac{1}{\iu t + \rho}\big[\e^{-(\iu t + \rho)s}\mu(s)\big]_{s=0}^{\infty}
                + \tfrac{1}{\iu t + \rho}\medint\int_{0}^{\infty}\e^{-(\iu t +\rho)s}\mu'(s)\dx[s]\\
              &= \tfrac{1}{\iu t + \rho}\flt(\mu') + \tfrac{\mu(0)}{\iu t + \rho}.
  \end{align*}
  Realizing that $\abs[\big]{\sqrt{2\pi}\flt(\mu)(t)}<1$ for all $t \in \mathbb{R}$ for $\rho >0$ large enough, one can estimate for $t \in \mathbb{R}$,
  \begin{align*}
    \tfrac{\iu t + \rho}{1 - \sqrt{2\pi}\flt(\mu)(t)}
    &= (\iu t + \rho)\sum_{k=0}^{\infty}\big(\sqrt{2\pi}\flt(\mu)(t)\big)^{k}\\
    &= \iu t + \rho + \sum_{k=1}^{\infty}(\iu t + \rho)\big(\sqrt{2\pi}\flt(\mu)(t)\big)^{k}\\
    &= \iu t + \rho + \sqrt{2\pi}\big(\flt(\mu')(t) + \mu(0)\big)\sum_{k=0}^{\infty}\big(\sqrt{2\pi}\flt(\mu)(t)\big)^{k}.
  \end{align*}
  Consequently,
  \begin{align*}
    \Re \tfrac{\iu t + \rho}{1 - \sqrt{2\pi}\flt(\mu)(t)}
    &\geq \rho - \Re\bigg[\sqrt{2\pi}\big(\flt(\mu')(t) + \mu(0)\big)\sum_{k=0}^{\infty}\big(\sqrt{2\pi}\flt(\mu)(t)\big)^{k}\bigg]\\
    &\geq \rho - \big(\norm{\mu'}_{\leb^{1}_{\rho}} + \sqrt{2\pi}\abs{\mu(0)}\big)\tfrac{1}{1 - \norm{\mu'}_{\leb^{1}_{\rho}}}.
  \end{align*}
  Note that the latter term is bounded for $\rho \to \infty$, since
  \begin{equation*}
    \big(\norm{\mu'}_{\leb^{1}_{\rho}} + \sqrt{2\pi}\abs{\mu(0)}\big)\frac{1}{1 - \norm{\mu'}_{\leb^{1}_{\rho}}} \xrightarrow[]{\rho \to \infty} \sqrt{2\pi}\abs{\mu(0)}.
  \end{equation*}
  Finally, this allows the verification of the desired coercivity estimate,
  \begin{align*}
    \Re \dualprod[\big]{\partial_{t,\rho}(1-\mu\ast)^{-1}f}{f}_{\leb^{2}_{\rho}(\mathbb{R};H)}
    &= \Re \dualprod[\bigg]{\frac{(\iu \mathrm{m} + \rho)}{1 - \sqrt{2\pi}\flt(\mu)}\flt(f)}{\flt(f)}_{\leb^{2}(\mathbb{R};H)}\\
    &= \medint\int_{-\infty}^{\infty} \Re \tfrac{\iu t + \rho}{1 - \sqrt{2\pi}\flt(\mu)(t)}\norm{\flt(f)(t)}_{H}^{2}\dx[t]\\
    &\geq \bigg[\rho -\underbrace{\frac{\norm{\mu'}_{\leb^{1}_{\rho}} + \sqrt{2\pi}\abs{\mu(0)}}{1 - \norm{\mu'}_{\leb^{1}_{\rho}}}}_{\eqcolon c}\bigg]\norm{\flt{f}}_{\leb^{2}(\mathbb{R};H)}^{2}\\
    &= (\rho - c) \norm{f}_{\leb^{2}_{\rho}(\mathbb{R};H)}^{2}\qedhere
  \end{align*}
\end{proof}

\section{Abstract Dirichlet and Neumann boundary conditions}
\label{app:AbstractBD}
Let $\Omega \subseteq \mathbb{R}^{d}$ be a bounded, open Lipschitz domain. Let
\begin{enumerate}[label = (\roman*), leftmargin = 5ex]
  \item $\operatorname{grad}\colon \leb^{2}(\Omega) \supseteq \sobh^{1}(\Omega) \to \leb^{2}(\Omega)^{d}$ be the weak gradient.
  \item $\operatorname{div}\colon \leb^{2}(\Omega)^{d}\supseteq \sobh (\operatorname{div};\Omega)\to \leb^{2}(\Omega)$ be the weak divergence.
  \item $\operatorname{curl}\colon \leb^{2}(\Omega)^{3}\supseteq \sobh (\operatorname{curl};\Omega)\to \leb^{3}(\Omega)^{3}$ be the weak rotation.
\end{enumerate}
Let
\begin{equation*}
  \mathring{\operatorname{grad}} \coloneq - \operatorname{div}^{\ast},\quad
  \mathring{\operatorname{div}} \coloneq - \operatorname{grad}^{\ast}\quad\text{and}\quad
  \mathring{\operatorname{curl}} \coloneq \operatorname{curl}^{\ast}.
\end{equation*}
Then the following examples befit the setting of \cref{subsec:AbstractBD}:
\begin{enumerate}[label = (\roman*), leftmargin = 5ex]
  \item $A = \begin{psmallmatrix} 0 & \operatorname{div} \\ \operatorname{grad} & 0 \end{psmallmatrix}$ on $H = \leb^2(\Omega) \!\times\! \leb^2(\Omega)^{d}$.
  \item $A = \begin{psmallmatrix}  0 & \operatorname{div}T \operatorname{grad} \\ \iota & 0 \end{psmallmatrix}$ on $H = \leb^2(\Omega) \!\times\! \sobh^{1}(\Omega)$.\\
        Here, let $T\colon \Omega \to \mathbb{R}^{d\times d}$ be bounded, symmetric and satisfy $c^{-1} \leq T \leq c$ for some $c>0$. The minimal version of $\operatorname{div}T \operatorname{grad}$ is $\operatorname{div} T \mathring{\operatorname{grad}}$ and the minimal version of $\iota\colon \leb^{2}(\Omega)\supseteq \sobh^{1}(\Omega)\to \leb^{2}(\Omega)$ is the restriction to $\mathring{\sobh}^{1}(\Omega)$. For details on these operators cf.~\cite[sec.~6.1]{STW2022}.
  \item $A = \begin{psmallmatrix} 0 & \operatorname{curl} \\ \operatorname{curl} & 0 \end{psmallmatrix}$ on $H = \leb^2(\Omega)^{3}\!\times\! \leb^2(\Omega)^{3}$.
\end{enumerate}
More examples are possible, e.g.
\begin{enumerate}[label = (\roman*), leftmargin = 5ex]
        \setcounter{enumi}{3}
  \item $A = \begin{psmallmatrix} 0 & \operatorname{Div} \\ \operatorname{Grad} & 0 \end{psmallmatrix}$ on $H = \leb^2(\Omega)^{d} \!\times\! \leb^2(\Omega)^{d\times d}_{\mathrm{sym}}$, where $\operatorname{Div}$ is the row-wise taken weak divergence of a matrix and $\operatorname{Grad}$ is the symmetric gradient of a vector.
\end{enumerate}
However, we confine to the first three instances.

\subsection{Reminder on traces}
\label{subsec:Traces}
Before explaining the connection between abstract boundary data spaces and classical trace spaces, a brief reminder on traces is presented. Let $\Omega \subseteq \mathbb{R}^{d}$ be a bounded, open Lipschitz domain. It is well-known that
\begin{itemize}[leftmargin = 4ex]
  \item the map
        \begin{equation*}
          \mathcal{C}(\overline{\Omega}) \to \mathcal{C}(\partial\Omega), \quad \varphi \mapsto \varphi\vert_{\partial\Omega},
        \end{equation*}
        has a bounded, linear and surjective extension
        \begin{equation*}
          \gamma_{0}\colon \sobh^{1}(\Omega) \to \sobh^{\frac{1}{2}}(\partial\Omega),
        \end{equation*}
        called the {\em Dirichlet trace}.
  \item the map
        \begin{equation*}
          \mathcal{C}(\overline{\Omega})^{d} \to \mathcal{C}(\partial\Omega), \quad \varphi \mapsto \dualprod{\nu}{\varphi\vert_{\partial\Omega}}_{\mathbb{C}^{n}},
        \end{equation*}
        where $\nu$ is the unit outward normal, has a bounded, linear and surjective extension
        \begin{equation*}
          \gamma_{\nu} \colon \sobh(\operatorname{div};\Omega) \to \sobh^{-\frac{1}{2}}(\partial\Omega),
        \end{equation*}
        called the {\em normal trace}.
  \item the map
        \begin{equation*}
          \mathcal{C}(\overline{\Omega})^{3}\to \mathcal{C}(\partial\Omega)^{3}, \quad \varphi \mapsto \nu \times \varphi\vert_{\partial\Omega},
        \end{equation*}
        has a bounded, linear and surjective extension
        \begin{equation*}
          \gamma_{\tau}\colon H(\operatorname{curl};\Omega) \to \mathcal{V}_{\tau}^{\times},
        \end{equation*}
        called the {\em twisted tangential trace}. $\mathcal{V}_{\tau}^{\times}$ is the completion of
        \begin{equation*}
           \dset{ \nu \times \gamma_{0}f}{f \in \sobh^{1}(\Omega)^{3}} \eqcolon \mathrm{M}^{\times}\subseteq \leb_{\tau}^{2} (\partial\Omega) \coloneq \dset{\varphi \in \leb^2(\partial\Omega)^{3}}{\dualprod{\nu}{\varphi}_{\mathbb{C}^{3}} = 0}
        \end{equation*}
        w.r.t.\ the range norm
        \begin{equation*}
          \norm{\varphi}_{\mathcal{V}_{\tau}^{\times}} \coloneq \inf_{\substack{f \in \sobh^{1}(\Omega)^{3}\\ \nu \times \gamma_{0}f = \varphi}} \norm{\varphi}_{\sobh(\operatorname{curl};\Omega)}, \quad\varphi \in \mathrm{M}^{\times}.
        \end{equation*}
  \item the map
        \begin{equation*}
          \mathcal{C}(\overline{\Omega})^{3}\to \mathcal{C}(\partial\Omega)^{3}, \quad \varphi \mapsto \big(\nu \times \varphi\vert_{\partial\Omega}\big)\times \nu,
        \end{equation*}
        has a bounded, linear and surjective extension
        \begin{equation*}
          \pi_{\tau}\colon \sobh(\operatorname{curl};\Omega) \to \mathcal{V}_{\tau},
        \end{equation*}
        called the {\em tangential trace}. $\mathcal{V}_{\tau}$ is the completion of $\dset{(\nu \times \gamma_{0}g)\times \nu}{g \in \sobh^{1}(\Omega)^{3}} \eqcolon \mathrm{M} \subseteq \leb^2_{\tau}(\partial\Omega)$ w.r.t.\ the range norm
        \begin{equation*}
          \norm{\varphi}_{\mathcal{V}_{\tau}} \coloneq \inf_{\substack{g \in \sobh^{1}(\Omega)^{3}\\ (\nu \times \gamma_{0}g)\times \nu = \varphi}} \norm{g}_{\sobh(\operatorname{curl};\Omega)}, \quad\varphi \in \mathrm{M}.
        \end{equation*}
\end{itemize}
All these traces are well-defined, bounded, linear and surjective maps with a conti-nuous right inverse. For proofs of these facts regarding $\gamma_{0}$ we refer to \cite[sec.~2.4\&2.5]{Necas2012}, regarding $\gamma_{\nu}$ to \cite[app.~A]{Zwart2015} and for $\pi_{\tau}$ and $\gamma_{\tau}$ we refer to \cite[app.~B]{SkrepekWaurick2024}. To explain the connection to the concept of abstract boundary data spaces we elaborate on the identification of aforementioned boundary spaces with boundary data spaces.
\begin{itemize}[leftmargin = 4ex]
  \item For the operator $\operatorname{grad}\colon \leb^2(\Omega) \supseteq \sobh^{1}(\Omega) \to \leb^2(\Omega)^{d}$ we trivially have $\sobh^{1}(\Omega) \cong \sobh(\abs{\operatorname{grad}} + \iu)$ and we can identify
        \begin{equation*}
          \mathcal{BD}(\operatorname{grad}) \cong \sobh^{\frac{1}{2}}(\partial\Omega),
        \end{equation*}
        because the operator
        \begin{equation*}
          \gamma \coloneq \gamma_{0}\circ \pi_{\mathcal{BD}(\operatorname{grad})}^{\ast}\colon \mathcal{BD}(\operatorname{grad}) \to \sobh^{\frac{1}{2}}(\partial\Omega)
        \end{equation*}
        is a Banach space isomorphism with inverse $\gamma^{-1} = \pi_{\mathcal{BD}(\operatorname{grad})} \circ R$, where $R$ is the right inverse of $\gamma_{0}$, appealing to \cite[cor.~4.4]{Trostorff2014}.
  \item For the operator $\operatorname{div}\colon \leb^2(\Omega)^{d}\supseteq \sobh(\operatorname{div};\Omega) \to \leb^2(\Omega)$ we trivially have $\sobh(\operatorname{div};\Omega) \cong \sobh(\abs{\operatorname{div}} + \iu)$ and we can identify
        \begin{equation*}
          \mathcal{BD}(\operatorname{div}) \cong \sobh^{-\frac{1}{2}}(\partial\Omega),
        \end{equation*}
        cf.~\cite[cor.~4.6 \& prop.~4.7]{Trostorff2014}.
  \item For the operator $\operatorname{curl}\colon \leb^2(\Omega)^{3} \supseteq \sobh(\operatorname{curl};\Omega) \to \leb^2(\Omega)^{3}$ we have $\sobh(\operatorname{curl};\Omega) \cong \sobh(\abs{\operatorname{curl}} + \iu)$ and additionally, we can identify
        \begin{equation*}
          \mathcal{BD}(\operatorname{curl}) \cong \mathcal{V}_{\tau} \cong \mathcal{V}_{\tau}^{\times},
        \end{equation*}
        cf.~\cite[sec.~2.3.1]{PicardSeidlerTrostorffWaurick2016}.
\end{itemize}
We now briefly discuss the three cases outlined at the start of this section.

\subsection{$\operatorname{div}$ and $\operatorname{grad}$}\phantom{.}\\
The operators
\begin{alignat*}{3}
  \mathring{D}&\coloneq \mathring{\operatorname{div}} &&\subseteq \operatorname{div} &= - \mathring{\operatorname{grad}}^{\ast},\\
  \mathring{G}&\coloneq \mathring{\operatorname{grad}} &&\subseteq \operatorname{grad} &= - \mathring{\operatorname{div}}^{\ast},
\end{alignat*}
satisfy the assumptions of \cref{subsec:AbstractBD}. Hence, we can identify the abstract boundary data spaces
\begin{equation*}
  \mathcal{BD}(\operatorname{grad}) \cong \sobh^{\frac{1}{2}}(\partial\Omega) \quad \text{and} \quad
  \mathcal{BD}(\operatorname{div}) \cong \sobh^{-\frac{1}{2}}(\partial\Omega)
\end{equation*}
and use \cref{th:ShiftingLifting} to impose pure Dirichlet and Neumann boundary conditions on the corresponding evolutionary equation.
\begin{enumerate}[label = (\roman*), leftmargin = 4ex]
  \item A pure normal trace $\gamma_{\nu}v = d(\argdot, u_{(\argdot)},v_{(\argdot)})$ corresponds to the problem
        \begin{equation*}
          \bigg[\partial_{t}M \!+ \!N \!+ \!\begin{pmatrix} 0 & \mathring{\operatorname{div}} \\ \operatorname{grad} & 0\end{pmatrix}\!\!\bigg]
          \!\!\begin{pmatrix} u \\ r \end{pmatrix}
          \!= \!\begin{pmatrix} f \\ g \end{pmatrix}
          \!- \!\big[\partial_{t}M \!+ \!N \big]\!\!\begin{pmatrix} 0 \\ L_{\operatorname{div}}d \end{pmatrix} \!- \!\begin{pmatrix} \operatorname{div}L_{\operatorname{div}}d \\ 0 \end{pmatrix}
        \end{equation*}
        with $L_{\operatorname{div}}d$ being a lifting of boundary condition $d$.
  \item A pure Dirichlet trace $\gamma_{0}u = g(\argdot, u_{(\argdot)},v_{(\argdot)})$ corresponds to the problem
        \begin{equation*}
          \bigg[\partial_{t}M \!+ \!N \!+ \!\begin{pmatrix} 0 & \operatorname{div} \\ \mathring{\operatorname{grad}} & 0\end{pmatrix}\!\!\bigg]
          \!\!\begin{pmatrix} u \\ r \end{pmatrix}
          \!= \!\begin{pmatrix} f \\ g \end{pmatrix}
          \!- \!\big[\partial_{t}M \!+ \!N\big]\!\!\begin{pmatrix} L_{\operatorname{grad}}g \\ 0 \end{pmatrix} \!- \!\begin{pmatrix} 0 \\ \operatorname{grad}L_{\operatorname{grad}}g\end{pmatrix}
        \end{equation*}
        with $L_{\operatorname{grad}}g$ being a lifting of boundary condition $g$.
\end{enumerate}

\begin{example}[Heat equation]
  The heat equation is given as
  \begin{equation*}
    \partial_{t}u + \operatorname{div}a \operatorname{grad} u = f
  \end{equation*}
  with $a\colon \Omega \to \mathbb{R}^{d\times d}$ being bounded, measurable and positive definite, $\Re a \geq c > 0$. Superimposing a stipulated pure Neumann boundary condition
  \begin{equation*}
    \gamma_{\nu}\operatorname{grad}u = d\big(t, u_{(t)}\big),
  \end{equation*}
  the system can be formalized as
  \begin{multline*}
    \bigg[\partial_{t}\underbrace{\begin{pmatrix} 1 & 0 \\ 0 & 0 \end{pmatrix}}_{= M}
    + \begin{pmatrix} 0 & 0 \\ 0 & a^{-1} \end{pmatrix}
    + \begin{pmatrix} 0 & \mathring{\operatorname{div}} \\ \operatorname{grad} & 0\end{pmatrix}\bigg]
    \begin{pmatrix} u \\ v \end{pmatrix}\\
    = \begin{pmatrix} f \\ g \end{pmatrix}
    - \bigg[\partial_{t}\begin{pmatrix} 1 & 0 \\ 0 & 0 \end{pmatrix}
    + \begin{pmatrix} 0 & 0 \\ 0 & a^{-1} \end{pmatrix}\bigg]
    \begin{pmatrix} 0 \\ L_{\operatorname{div}}d \end{pmatrix}
    - \begin{pmatrix} \operatorname{div}L_{\operatorname{div}}d \\ 0 \end{pmatrix}.
  \end{multline*}
  Here, the block operator structure of $M$ guarantees that $(0, L_{\operatorname{div}}d) \in \operatorname{dom}(\partial_{t}M)$; in fact, $(0, L_{\operatorname{div}}d) \in \ker (M)$ and the equation simplifies to
  \begin{equation*}
    \bigg[\partial_{t}\begin{pmatrix} 1 & 0 \\ 0 & 0 \end{pmatrix}
    + \begin{pmatrix} 0 & 0 \\ 0 & a^{-1} \end{pmatrix}
    + \begin{pmatrix} 0 & \mathring{\operatorname{div}} \\ \operatorname{grad} & 0\end{pmatrix}\bigg]
    \begin{pmatrix} u \\ v \end{pmatrix}
    = \begin{pmatrix} f - \operatorname{div}L_{\operatorname{div}}d \\ g - a^{-1}L_{\operatorname{div}}d \end{pmatrix}.
  \end{equation*}
  Hence, assuming that $d$ adheres to adequate assumptions from \cref{subsubsec:LipschitzCond}, one can state a version of \cref{th:localExistDN} after passing to a formulation with suitable initial condition.\\
  For the Dirichlet boundary condition $\gamma_{0}u = g$ one would need to make sure that the lifted boundary condition satisfies $L_{\operatorname{grad}}g \in \operatorname{dom}(\partial_{t})$, which is not automatically assured by the form of the system.
\end{example}

\begin{example}[Wave equation]
  The wave equation is given as
  \begin{equation*}
    \partial_{t}^{2}u - \operatorname{div} T \operatorname{grad} u = f
  \end{equation*}
  for some $T\colon \Omega \to \mathbb{R}^{d\times d}$ that is pointwise symmetric, bounded, measurable and positive definite, $T \geq c > 0$. The wave equation is a second-order equation in time for the displacement and can be posed as a first order system for the velocity and stress. The Dirac formulation/port-Hamiltonian formulation of the wave equation, cf.~\cite{Zwart2015}, is given as the system
  \begin{equation*}
    \bigg[\partial_{t}\underbrace{\begin{pmatrix} 1 & 0 \\ 0 & T^{-1} \end{pmatrix}}_{\eqcolon M}
    - \begin{pmatrix} 0 & \operatorname{div} \\ \operatorname{grad} & 0 \end{pmatrix}\bigg]
    \begin{pmatrix} u \\ v \end{pmatrix}
    = \begin{pmatrix} f \\ 0 \end{pmatrix}.
  \end{equation*}
  Here, the necessary assumption for \cref{th:ShiftingLifting}, $(-L_{\operatorname{grad}}g,0) \in dom(\partial_{t}M)$ for Dirichlet boundary conditions or $(0, -L_{\operatorname{div}}d) \in \operatorname{dom}(\partial_{t}M)$ for Neumann boundary conditions is not automatically satisfied and needs to be imposed.
\end{example}

\subsection{$\operatorname{div}T \operatorname{grad}$ and $\iota$}\phantom{.}\\
Here, we let $T\colon \Omega \to \mathbb{R}^{d\times d}$ be pointwise symmetric, bounded, measurable and positive definite, $T \geq c > 0$. Additionally let
\begin{equation*}
  \iota \colon \leb^2(\Omega) \supseteq \sobh^{1}(\Omega) \to \leb^2(\Omega), \quad \varphi \mapsto \varphi
\end{equation*}
denote the (unbounded) inclusion of $\sobh^{1}(\Omega)$ into $\leb^2(\Omega)$. We understand
\begin{align*}
  \operatorname{div}T \operatorname{grad} &\colon \sobh^{1}(\Omega)\supseteq \operatorname{dom}(\operatorname{div}T \operatorname{grad}) \to \leb^2(\Omega),\quad \varphi \mapsto \operatorname{div}T \operatorname{grad}\varphi\\
  \intertext{with domain}
  \operatorname{dom}(\operatorname{div}T \operatorname{grad}) &\coloneq \dset{\varphi \in \sobh^{1}(\Omega)}{T \operatorname{grad}\varphi \in \sobh(\operatorname{div};\Omega)}.
\end{align*}
Taking a simplified version of the system introduced in \cite{Aigner2026}, we define
\begin{equation*}
  A = \begin{pmatrix} 0 & \operatorname{div}T \operatorname{grad} \\ \iota & 0 \end{pmatrix}
\end{equation*}
and observe that $\bigl((\sobh^{\frac{1}{2}}(\partial\Omega)', \sobh^{\frac{1}{2}}(\partial\Omega)), \gamma_{0} + \gamma_{\nu}T \operatorname{grad}, \gamma_{0}\bigr)$ forms a boundary triple for $A$ on the state space $H = \sobh^{1}(\Omega) \!\times\! \leb^2(\Omega)$,\footnote{In \cite{Aigner2026}, boundary triples were defined for adjoints of skew-symmetric operators and the trace space is not identified with its dual space, cf.~\cite[def.~2.7]{Aigner2026}.} where $\sobh^{1}(\Omega)$ is equipped with the inner product
\begin{equation*}
  \dualprod{x}{y}_{\sobh^{1}(\Omega)} \coloneq \dualprod{x}{y}_{\leb^2(\Omega)} + \dualprod{\gamma_{0}x}{\gamma_{0}y}_{\leb^2(\partial\Omega)}.
\end{equation*}
Hence, with this formulation, via shifting and lifting one can natively formulate Dirichlet boundary or Robin boundary conditions. In the latter case one uses the ``Robin trace''
\begin{equation*}
  \gamma_{R}\coloneq \gamma_{0} + \gamma_{\nu}T \operatorname{grad} \colon \operatorname{dom}(\operatorname{div}T \operatorname{grad}) \to \sobh^{-\frac{1}{2}}(\partial\Omega)
\end{equation*}
and an adequate lifting thereof, i.e.\ a continuous right-inverse $R\colon \sobh^{-\frac{1}{2}}(\partial\Omega) \to \operatorname{dom}(\operatorname{div} T \operatorname{grad})$, to obtain the following version of the wave equation
\begin{equation*}
  \left[\partial_{t}\begin{pmatrix} 1 & 0 \\ 0 & 1 \end{pmatrix}
    - \begin{pmatrix} 0 & \operatorname{div} T \mathring{\operatorname{grad}} \\ \iota & 0 \end{pmatrix}\right]
  \begin{pmatrix} u \\ v \end{pmatrix}
    = \begin{pmatrix} f \\ g \end{pmatrix}
    - \partial_{t}\begin{pmatrix} 0 \\ Rr \end{pmatrix}
    + \begin{pmatrix} \operatorname{div}T \operatorname{grad} Rr \\ 0 \end{pmatrix},
\end{equation*}
where the boundary condition takes the form
\begin{equation*}
  \gamma_{R}u(t) = \gamma_{0}u(t) + \gamma_{\nu}T \operatorname{grad}u(t) = r\big(t, u_{(t)},v_{(t)}\big).
\end{equation*}
Note that, as already showcased in \cite[sec.~6.5]{AignerWaurick2026}, for the wave equation, the structural requirement imposed on the right-hand side of the shifted and lifted version to carry state-dependent delay of the form $F(t,u_{(t)},v_{(t)}) = \widetilde{F}\big(t,(I_{\rho}u)_{(t)}, (I_{\rho}v)_{(t)}\big)$ can always be assured if $F$ only depends on the history of the displacement $u_{(t)}$, using the fact that $v_{t} = u$, i.e.\ $F(t,u_{(t)}) = \widetilde{F}\big(t, (I_{\rho}v)_{(t)}\big)$.

\subsection{$\operatorname{curl}$ and $\operatorname{curl}$}\phantom{.}\\
Here, we define the operators
\begin{alignat*}{3}
  \mathring{D} &\coloneq \phantom{-}\mathring{\operatorname{curl}} &\subseteq \phantom{-}\operatorname{curl} &= \phantom{-}\mathring{\operatorname{curl}}^{\ast},\\
  \mathring{G} &\coloneq -\mathring{\operatorname{curl}} &\subseteq -\operatorname{curl} &= -\mathring{\operatorname{curl}}^{\ast},
\end{alignat*}
which satisfy the assumptions of \cref{subsec:AbstractBD} and we can pose pure Dirichlet and Neumann boundary conditions associated to the operator
\begin{equation*}
  A = \begin{pmatrix} 0 & -\operatorname{curl} \\ \operatorname{curl} & 0 \end{pmatrix}.
\end{equation*}
Let $\epsilon, \mu \colon \Omega \to \mathbb{R}^{3\times 3}$ be bounded, measurable, symmetric matrix valued functions denoting the eletric conductivity and the magnetic permeability respectively. Let $\sigma \colon \Omega \to \mathbb{R}^{3\times 3}$ be bounded and measurable and denote the electric conductivity. Let the assumptions
\begin{equation*}
  \epsilon + \Re \sigma \geq c \quad\text{and}\quad \mu \geq c
\end{equation*}
hold for some $c>0$. Then Maxwell's equations can be posed as the evolutionary equation
\begin{equation*}
  \left[\partial_{t}\!\begin{pmatrix} \epsilon & 0 \\ 0 & \mu \end{pmatrix}
  + \begin{pmatrix} \sigma & 0 \\ 0 & 0 \end{pmatrix}
  + \begin{pmatrix} 0 & -\operatorname{curl} \\ \operatorname{curl} & 0 \end{pmatrix}\right]
    \!\begin{pmatrix} E \\ H \end{pmatrix}
  = \begin{pmatrix} f \\ g \end{pmatrix},
\end{equation*}
cf.~\cite[p.~94ff]{STW2022}. By means of suitable continuous liftings $L \colon \mathcal{V}_{\tau}\to \sobh(\operatorname{curl};\Omega)$ or $L\colon \mathcal{V}_{\tau}^{\times}\to \sobh(\operatorname{curl};\Omega)$ of the tangential trace $\pi_{\tau}$ or the twisted tangential trace $\gamma_{\tau}$ respectively, the following boundary conditions can be accommodated:
\begin{itemize}[leftmargin = 4ex]
  \item $\gamma_{\tau}E = h(t, E_{(t)}, H_{(t)})$ (perfect conduction for $h = 0$) resulting in
        \begin{equation*}
          \left[\partial_{t}\!\begin{pmatrix} \epsilon & 0 \\ 0 & \mu \end{pmatrix}
            + \begin{pmatrix} \sigma & 0 \\ 0 & 0 \end{pmatrix}
            + \begin{pmatrix} 0 & -\operatorname{curl} \\ \mathring{\operatorname{curl}} & 0 \end{pmatrix}\right]
          \!\begin{pmatrix} E \\ H \end{pmatrix}
          = \begin{pmatrix} f - \sigma Lh\\ g - \operatorname{curl}Lh \end{pmatrix}
          - \begin{pmatrix} \partial_{t}\epsilon Lh \\ 0 \end{pmatrix}.
        \end{equation*}
  \item $\pi_{\tau}H = h(t, E_{(t)},H_{(t)})$ resulting in
        \begin{equation*}
          \left[\partial_{t}\!\begin{pmatrix} \epsilon & 0 \\ 0 & \mu \end{pmatrix}
            + \begin{pmatrix} \sigma & 0 \\ 0 & 0 \end{pmatrix}
            + \begin{pmatrix} 0 & -\mathring{\operatorname{curl}} \\ \operatorname{curl} & 0 \end{pmatrix}\right]
          \!\begin{pmatrix} E \\ H \end{pmatrix}
          = \begin{pmatrix} f + \operatorname{curl}Lh \\ g \end{pmatrix}
          - \begin{pmatrix} 0 \\ \partial_{t}\mu Lh \end{pmatrix}.
        \end{equation*}
\end{itemize}
Of course one needs to assure that the domain condition $Lh \in \operatorname{dom}(\partial_{t}\epsilon)$ or $Lh \in \operatorname{dom}(\partial_{t}\mu)$ is met respectively. Local well-posedness of both formulations is then ascertained under the assumptions of \cref{th:localExistDN} after passing to a formulation with suitable initial conditions.

\section{Initial conditions}
\label{app:InitialValues}
As cursorily explained in \cref{subsubsec:InitialValues}, the formulation of initial value problems for evolutionary equations is not immediately obvious. A detailed account of how to pose initial value problems for evolutionary equations can be found in \cite[ch.~3]{Trostorff2018}, cf.~also \cite[app.~B]{AignerWaurick2026} in the context of state-dependent delay. In general, the problem
\begin{equation*}
  \left\{
    \begin{aligned}
      \big[\partial_{t}M + N + A\big]u &= F(\argdot, u_{(\argdot)}), \quad t>0,\\
      u_{(0)} &= \Phi,
    \end{aligned}
  \right.
\end{equation*}
can be posed in an extrapolation space $\sobh^{-1}_{\rho}(\mathbb{R};H)\coloneq \sobh^{1}_{\rho}(\mathbb{R};H)'$ in the form
\begin{equation}
  \label{eq:EvoleqInitial}
  \big[\partial_{t} M + N + A\big]v = f + F(\argdot, (v + Z)_{(\argdot)}),
\end{equation}
where $Z$ is a continuation of the prehistory $\Phi$, $f \in \sobh^{-1}_{\rho}(\mathbb{R};H)$ and the unknown $v$ will have support only on $[0,\infty)$. For details we again refer to \cite[ch.~3]{Trostorff2018}. To make use of the theory developed in \cref{sec:EvolEq} and \cite{AignerWaurick2026}, one needs $f \in \leb^{2}_{\rho}(\mathbb{R};H)$ though. It is not clear how and when one can arrive at a suitable $f$ in the general case, because $M$ and $N$ might include highly nonlocal operators, such as convolutions. For the strategy for local operators we refer to \cite[sec.~5]{AignerWaurick2026}. Here, we focus on the matter of extending the prehistory $\Phi$, which is required for both the formulation of \cref{eq:EvoleqInitial}, as well as the formulation of the consistency conditions mandated by \cref{th:localExistIns} and \cref{th:localExistOut} in the case of nonautonomous delay of the form $t \mapsto t - \tau(t)$ for a delay functional $\tau\colon \mathbb{R}\to [0,h]$ adhering to \cref{ass:Tau}. We additionally assume:
\begin{enumerate}[label = (\roman*), leftmargin = 5ex]
  \item\label{it:Regularity} $\Phi \in \sobw^{1,\infty}(-h,0;H)$ and $\Phi$ is continuously differentiable at $t = 0$.
  \item\label{it:LeftBoundary} $\Phi(-h) = 0$.
\end{enumerate}

\begin{remark}
  Assumption \ref{it:Regularity} is mandated by the phenomenon that without a Lipschitz-continuous prehistory, unique solvability of a differential equation with state-dependent delay is not guaranteed, cf.~\cite[ex.~4.3]{Waurick2023}. It is well understood that the role of $t=0$ is special and we want to make sure that $\Phi$ can be evaluated at $t=0$, hence the second part of the assumption.\\
  \Cref{it:LeftBoundary} seems like a restriction at first glance, but it really is not: If $\Phi(-h)\neq 0$, the backward time horizon can be increased to $\tilde{h} > h$ and $(-h,\Phi(-h))$ and $(-\tilde{h},0)$ can be connected as smoothly as possible (in the sense that the first point is the endpoint of the graph of the function $\Phi$ and the second one is the endpoint of $(-\infty,-\tilde{h}] \!\times\! \{0\}$). This backward continuation $\widetilde{\Phi}$ of the prehistory $\Phi$ can take the place of $\Phi$ without issues, since the function $F$ does not ``see'' the backward continuation at all.\\
  We additionally repeat that the \cref{ass:Tau} imply that $t\mapsto t - \tau(t)$ is injective. Hence, there exists a unique $t_{0}>0$ such that $t_{0} = \tau(t_{0})$.
\end{remark}

Employing the same strategy as in \cite{AignerWaurick2026}, we can define the affine extension of $\Phi$ by
\begin{equation*}
  Z\colon \mathbb{R}\to H, \quad t\mapsto \begin{cases}
    \Phi(0) + t \Phi'(0) &t\geq 0,\\
    \Phi(t) &t<0.
  \end{cases}
\end{equation*}
Clearly, $Z \in \sobh^{1}_{\rho}(\mathbb{R};H)$. An application of the shift $S_{\tau}$ from \cref{def:Shift} yields
\begin{equation*}
  S_{\tau}Z \colon \mathbb{R}\to H,\quad t\mapsto \begin{cases}
    \Phi(0) + \big(t-\tau(t)\big)\Phi'(0) & t\geq t_{0},\\
    \Phi(t-\tau(t)) & t<t_{0}.
  \end{cases}
\end{equation*}
If one wishes to solve a second order (in time) problem, such as the wave equation, the prehistory needs to be smoother and we require
\begin{enumerate}[label = (\roman*), leftmargin = 5ex]
  \item $\Phi \in \sobw^{2,\infty}(-h,0;H)$ with $\Phi'$ continuously differentiable at $t = 0$ and
  \item $\tau \in \sobw^{2,\infty}(\mathbb{R})$.
\end{enumerate}
To illustrate the point and elaborate on the formulation of the consistency condition, consider the wave equation with nonautonomous delay in the velocity as in \cref{subsubsec:Wave1}, i.e.
\begin{equation*}
  \bigg[\partial_{t}\begin{pmatrix} T^{-1} & 0 \\ 0 & 1 \end{pmatrix}
  +\begin{pmatrix} 0 & 0 \\ 0 & b_{1} + b_{2}S_{\tau} \end{pmatrix}
  - \begin{pmatrix} 0 & \mathring{\operatorname{grad}} \\ \operatorname{div} & 0 \end{pmatrix}\bigg]
\begin{pmatrix} u \\ v \end{pmatrix}
= \begin{pmatrix} 0 \\ F(\argdot, v_{(\argdot)}) \end{pmatrix}.
\end{equation*}
A reasonable continuation of the prehistory $\Phi$ is
\begin{equation*}
  Z\colon \mathbb{R}\to H, \quad t\mapsto \begin{cases}
    \Phi(0) + t \Phi'(0) + \tfrac{1}{2}t^{2}\Phi''(0) &t\geq 0,\\
    \Phi(t) &t<0.
  \end{cases}
\end{equation*}
After transitioning to a formulation for initial conditions, the consistency condition for the second component of the state reads
\begin{equation*}
  \Phi''(0) - \operatorname{div}T \operatorname{grad} \Phi(0) + b_{1}(0)\Phi'(0) + b_{2}(0)\big(1-\tau'(0)\big)\Phi'(-\tau(0)) = F(0,\Phi).
\end{equation*}
Of course this condition mandates $\Phi(0) \in \operatorname{dom}(\operatorname{div}T \operatorname{grad})$.

\

In the case of boundary value problems, cf.~\cref{sec:BdyDelay}, there is the additional complication of the lifting $L$ from the boundary space to the state space. We only exemplify this in the case of \cref{ex:Wave}, where a prehistory $\Phi$ with values in the boundary space $\leb^2(\partial\Omega)$ is required. Because of block operator structure, it suffices to consider only the third component of \cref{eq:WaveComp} to obtain the consistency condition
\begin{align*}
  \big[(b_{1}+b_{2}S_{\tau})Z - \gamma_{\nu}Z\big](0) &= h(0,u_{(0)},v_{(0)},w_{(0)}),\\
  \intertext{which translates to}
  b_{1}(0)\Phi(0) + b_{2}(0)\Phi(-\tau(0)) - \gamma_{\nu}u(0) &= h(0,u_{(0)},v_{(0)},\Phi).
\end{align*}
Note that because our solution theory provides solutions in $\sobh^{1}_{\rho}$, evaluating $u$ at $t=0$ is not a problem. Under the assumption $u(0) \in \sobh(\operatorname{div};\Omega)$ we obtain
\begin{equation*}
  \big(b_{1}(0)- \gamma_{\nu})\Phi(0) + b_{2}(0)\Phi\big(-\tau(0)\big) = h(0,\Phi),
\end{equation*}
provided that $h$ has no further dependence on histories.

\section{Further applications}
\label{app:Applications}

In the final appendix (some) additional details regarding possible alternative formulations of heat and wave equation by means of the formalism from \cref{subsec:ExtendedState} are provided; underlining the versatility of the presented approach.

\subsection{Heat equation}
Making use of formalism \eqref{it:DivGradVersion1}, the corresponding formulation of the heat equation reads
\begin{equation*}
  \left[\partial_{t}\begin{pmatrix} 1 & 0 & 0 \\ 0 & 0 & 0 \\ 0 & 0 & m \end{pmatrix}
    + \begin{pmatrix} 0 & 0 & 0 \\ 0 & a^{-1} & 0 \\ 0 & 0 & n \end{pmatrix}
    + \begin{pmatrix} 0 & - \begin{pmatrix} \operatorname{grad} \\ \gamma_{0} \end{pmatrix}^{\!\!\ast}\\
    \begin{pmatrix} \operatorname{grad} \\ \gamma_{0} \end{pmatrix} & 0
    \end{pmatrix}\right]
    \begin{pmatrix} u \\ v \\ w \end{pmatrix}
    = \begin{pmatrix} f \\ g \\ h \end{pmatrix},
\end{equation*}
where $m,n \in \mathcal{L}\big(\leb^{2}_{\rho}(\mathbb{R};\leb^2(\partial\Omega))\big)$. The full set of equations encoded by the system is
\begin{align*}
  \partial_{t}u + \operatorname{div}v &= f,\\
  a^{-1}v + \operatorname{grad}u &= g,\\
  \partial_{t}mw + nw + \gamma_{0}u &= h,\\
  -\gamma_{\nu}v &= w,
\end{align*}
where the last equation is the boundary condition encoded in $A$. In other words, this system corresponds to the heat equation with the boundary condition
\begin{equation*}
  -\partial_{t}m\gamma_{\nu}v - n\gamma_{\nu}v + \gamma_{0}u = h.
\end{equation*}
For regular enough solutions, i.e. $(u,v,w) \in \operatorname{dom}(\partial_{t}M + N + A)$, we can appeal to the differential equation $a^{-1}v = g - \operatorname{grad}u$ and substitute $v$ for the case $g=0$ to obtain
\begin{equation*}
  \partial_{t}m\gamma_{\nu}a\operatorname{grad}u + n \gamma_{\nu}a\operatorname{grad}u + \gamma_{0}u = h.
\end{equation*}
The choice $m = 0$ leads to the Robin boundary condition
\begin{equation*}
  n \gamma_{\nu}a\operatorname{grad}u + \gamma_{0}u = h,
\end{equation*}
which includes a possible nonautonomous delay in the Neumann trace.

\

As the concluding example for the heat equation, we want to recall the formulation of Wentzell--Robin boundary conditions for the heat equation in \cite[sec.~2.2]{PicardSeidlerTrostorffWaurick2016}. There, one uses formalism \eqref{it:Wentzell} to pose the heat equation as the system
\begin{multline*}
  \left[\partial_{t}\underbrace{\begin{pmatrix} m_{1} & 0 & 0 & 0 \\ 0 & m_{2} & 0 & 0 \\ 0 & 0 & m_{3} & 0 \\ 0 & 0 & 0 & m_{4}\end{pmatrix}}_{\eqcolon M}
    + \underbrace{\begin{pmatrix} n_{1} & 0 & 0 & 0 \\ 0 & n_{2} & 0 & 0 \\ 0 & 0 & n_{3} & 0 \\ 0 & 0 & 0 & n_{4}\end{pmatrix}}_{\eqcolon N}\right.\\
    \left.+ \begin{pmatrix}
    \begin{pmatrix} 0 \end{pmatrix} &
    - \begin{pmatrix} \operatorname{grad} \\ \gamma_{0} \\ \operatorname{grad}_{\partial\Omega}\gamma_{0} \end{pmatrix}^{\!\!\ast} \\
    \begin{pmatrix} \operatorname{grad} \\ \gamma_{0} \\ \operatorname{grad}_{\partial\Omega}\gamma_{0} \end{pmatrix} &
    \begin{pmatrix} 0 & 0 & 0 \\ 0 & 0 & 0 \\ 0 & 0 & 0 \end{pmatrix}
  \end{pmatrix}\right]
  \begin{pmatrix} p \\ v \\ \eta_{1} \\ \eta_{2} \end{pmatrix}
  = \begin{pmatrix} f \\ g \\ h_{1} \\ h_{2} \end{pmatrix}
\end{multline*}
in $\leb^{2}_{\rho}\big(\mathbb{R}; \leb^2(\Omega) \!\times\! \leb^2(\Omega)^{d} \!\times\! \leb^2(\Omega) \!\times\! \leb^2_{\tau}(\partial\Omega)\big)$ for sufficiently large $\rho$, where $M$ and $N$ satisfy the assumptions of \cref{th:Picard}.\\
The choices $m_{1} = 1$, $m_{2} = T^{-1}$, $m_{4} = 0$ and $n_{1}=0$, $n_{2}=0$, $n_{4}=1$ lead to the following system of equations:
\begin{align*}
  \partial_{t}p + \operatorname{div}v &= f,\\
  \partial_{t}T^{-1}v + \operatorname{grad}p &= g,\\
  \partial_{t}m_{3}\eta_{1} + n_{3}\eta_{1} + \gamma_{0}p &= h_{1},\\
  \eta_{2} + \operatorname{grad}_{\partial\Omega}\gamma_{0}p &= h_{2},\\
  \gamma_{\nu}v + \eta_{1} &= - \operatorname{grad}_{\partial\Omega}^{\diamond}\eta_{2},
\end{align*}
where again the last equation is the boundary condition encoded into $A$. For strong solutions one can now plug in both the fourth and fifth equation into the third one, resulting in
\begin{equation*}
  \big(\partial_{t}m_{3} + n_{3}\big)\big[- \operatorname{grad}_{\partial\Omega}^{\diamond}(h_{2} - \operatorname{grad}_{\partial\Omega}\gamma_{0}p) - \gamma_{\nu}v\big] +\gamma_{0}p = h_{1}.
\end{equation*}
With the choices $h_{2}=0$ and $m_{3}=0$ one obtains
\begin{equation*}
  n_{3}\operatorname{grad}_{\partial\Omega}^{\diamond}\operatorname{grad}_{\partial\Omega}\gamma_{0}p - n_{3}\gamma_{\nu}v + \gamma_{0}p = h_{1}.
\end{equation*}
It can be shown that $\operatorname{grad}_{\partial\Omega}^{\diamond} \operatorname{grad}_{\partial\Omega} \gamma_{0} = \Delta_{\mathrm{LB}}$, the Laplace--Beltrami operator. This boundary condition can therefore be viewed as a Wentzell--Robin boundary condition.

\subsection{Wave equation}
We first briefly spell out the versions of the wave equation obtained via
\begin{itemize}[leftmargin = 4ex]
  \item formalism \eqref{it:DivGradVersion1}. Here we obtain the model
        \begin{equation}
          \label{eq:Wave2}
          \left[\!\partial_{t}\!\begin{pmatrix} 1 & 0 & 0 \\ 0 & T^{-1} & 0 \\ 0 & 0 & m \end{pmatrix}
            \!\!+ \!\!\begin{pmatrix} 0 & 0 & 0 \\ 0 & 0 & 0 \\ 0 & 0 & n \end{pmatrix}
            \!\!- \!\!\begin{pmatrix} 0 & - \begin{pmatrix} \operatorname{grad} \\ \gamma_{0} \end{pmatrix}^{\!\!\ast}\\
              \!\begin{pmatrix} \operatorname{grad} \\ \gamma_{0} \end{pmatrix} & 0
            \end{pmatrix}\!\right]
          \!\!\begin{pmatrix} u \\ v \\ w \end{pmatrix}
          \!= \!\begin{pmatrix} f \\ g \\ h \end{pmatrix},
        \end{equation}
        where $m,n \in \mathcal{L}\big(\leb^{2}_{\rho}(\mathbb{R};\leb^{2}(\partial\Omega))\big)$. Combining both the boundary condition contained in the system as well as the boundary condition encoded in $A$, we obtain
        \begin{equation*}
          \partial_{t}m\gamma_{\nu}v + n\gamma_{\nu}v + \gamma_{0}u = -h.
        \end{equation*}
        For $m=1$, $n=0$ and $g = 0$ for strong solutions we obtain Robin boundary conditions without delay,
        \begin{equation*}
          \gamma_{\nu}T \operatorname{grad}u + \gamma_{0}u = -h.
        \end{equation*}
  \item formalism \eqref{it:DivGradVersion2}. Here we obtain the model
        \begin{equation}
          \label{eq:Wave3}
          \left[\partial_{t}\!\begin{pmatrix} T^{-1} & 0 & 0 \\ 0 & 1 & 0 \\ 0 & 0 & m \end{pmatrix}
            \!\!+ \!\!\begin{pmatrix} 0 & 0 & 0 \\ 0 & 0 & 0 \\ 0 & 0 & n \end{pmatrix}\\
            \!\!- \!\!\begin{pmatrix} 0 & - \begin{pmatrix} \operatorname{div} \\ \gamma_{\nu} \end{pmatrix}^{\!\!\ast}\\
              \begin{pmatrix} \operatorname{div} \\ \gamma_{\nu} \end{pmatrix} & 0
            \end{pmatrix}\right]
          \!\!\begin{pmatrix} u \\ v \\ w \end{pmatrix}
          \!= \!\begin{pmatrix} f \\ g \\ h \end{pmatrix},
        \end{equation}
        where $m,n \in \mathcal{L}\big(\leb^{2}_{\rho}(\mathbb{R};\leb^{2}(\partial\Omega))\big)$. The full system of equations contained in \eqref{eq:WaveComp} reads
        \begin{align*}
          \partial_{t}T^{-1}u - \operatorname{grad}v &= f,\\
          \partial_{t}v - \operatorname{div}u &= g,\\
          \partial_{t}mw + nw - \gamma_{\nu}w &= h,\\
          w + \gamma_{0}v &= 0,
        \end{align*}
        where the last equation corresponds to the boundary condition encoded in the operator $A$. Combining the last two equations results in the boundary condition
        \begin{equation*}
          \partial_{t}m\gamma_{0}v + n\gamma_{0}v + \gamma_{\nu}u = -h,
        \end{equation*}
        which corresponds to a Robin boundary condition for the choice $m=0$.
  \item formalism \eqref{it:LaplaceVersion2}. Here, the wave equation is formulated as the system
        \begin{equation}
          \label{eq:Wave4}
          \left[\partial_{t}\!\begin{pmatrix} 1 & 0 & 0 \\ 0 & 1 & 0 \\ 0 & 0 & m \end{pmatrix}
            \!+ \!\begin{pmatrix} 0 & 0 & 0 \\ 0 & 0 & 0 \\ 0 & 0 & n \end{pmatrix}\\
            \!- \!\begin{pmatrix} 0 & - \begin{pmatrix} \iota \\ \gamma_{0} \end{pmatrix}^{\!\!\ast}\\
              \begin{pmatrix} \iota \\ \gamma_{0} \end{pmatrix} & 0
            \end{pmatrix}\!\right]
          \!\begin{pmatrix} u \\ v \\ w \end{pmatrix}
          \!= \!\begin{pmatrix} f \\ g \\ h \end{pmatrix}
        \end{equation}
        with $m,n \in \mathcal{L}\big(\leb^2_{\rho}(\mathbb{R};\leb^2(\partial\Omega))\big)$. The system reads
        \begin{align*}
          \partial_{t}u - \operatorname{div}T \operatorname{grad} v &=f,\\
          \partial_{t}v - u &= g,\\
          \partial_{t}mw + nw -\gamma_{0}u &= h,\\
          w + \gamma_{n}v &= 0,
        \end{align*}
        where the last equation describes the boundary condition encoded into the spatial operator $A$. Combining the latter two equations produces the boundary condition carried by the system,
        \begin{equation*}
          \partial_{t}m\gamma_{n}v + n\gamma_{n}v + \gamma_{0}u = -h.
        \end{equation*}
        For $g = 0$ and autonomous $m$, one obtains generalized Robin boundary conditions
        \begin{equation*}
          m \gamma_{n}u + n\gamma_{n}v + \gamma_{0}u = -h.
        \end{equation*}
\end{itemize}

\bibliographystyle{abbrvurl}
\bibliography{references}

@article{AignerWaurick2026,
	author = {Bernhard Aigner and Marcus Waurick},
	doi = {https://doi.org/10.1016/j.na.2026.114088},
	issn = {0362-546X},
	journal = {Nonlinear Analysis},
	pages = {114088},
	title = {Evolutionary equations with state-dependent delay},
	volume = {269},
	year = {2026},
    }

@book{STW2022,
 author = {Seifert, Christian and Trostorff, Sascha and Waurick, Marcus},
 title = {Evolutionary equations. {Picard}'s theorem for partial differential equations, and applications},
 fseries = {Operator Theory: Advances and Applications},
 series = {Oper. Theory: Adv. Appl.},
 issn = {0255-0156},
 volume = {287},
 isbn = {978-3-030-89396-5; 978-3-030-89399-6; 978-3-030-89397-2},
 year = {2022},
 publisher = {Cham: Birkh{\"a}user},
 language = {English},
 doi = {10.1007/978-3-030-89397-2},
 zbMATH = {7414827},
 Zbl = {1497.35008}
}

@article{PicardSeidlerTrostorffWaurick2016,
 author = {Picard, Rainer and Seidler, Stefan and Trostorff, Sascha and Waurick, Marcus},
 title = {On abstract grad-div systems},
 fjournal = {Journal of Differential Equations},
 journal = {J. Differ. Equations},
 issn = {0022-0396},
 volume = {260},
 number = {6},
 pages = {4888--4917},
 year = {2016},
 language = {English},
 doi = {10.1016/j.jde.2015.11.033},
 zbMATH = {6535778},
 Zbl = {1353.35119}
}

@article{NicaisePignottiValein2011,
 author = {Nicaise, Serge and Pignotti, Cristina and Valein, Julie},
 title = {Exponential stability of the wave equation with boundary time-varying delay},
 fjournal = {Discrete and Continuous Dynamical Systems. Series S},
 journal = {Discrete Contin. Dyn. Syst., Ser. S},
 issn = {1937-1632},
 volume = {4},
 number = {3},
 pages = {693--722},
 year = {2011},
 language = {English},
 doi = {10.3934/dcdss.2011.4.693},
 zbMATH = {5869376},
 Zbl = {1215.35030}
}

@article{TrostorffWaurick2021,
 author = {Trostorff, Sascha and Waurick, Marcus},
 title = {Maximal regularity for non-autonomous evolutionary equations},
 fjournal = {Integral Equations and Operator Theory},
 journal = {Integral Equations Oper. Theory},
 issn = {0378-620X},
 volume = {93},
 number = {3},
 pages = {37},
 note = {Id/No 30},
 year = {2021},
 language = {English},
 doi = {10.1007/s00020-021-02645-5},
 zbMATH = {7367378},
 Zbl = {1467.35081}
}

@article{Spek2020,
	author = {Spek, Len and Kuznetsov, Yuri A. and van Gils, Stephan A.},
	date = {2020/12/09},
	doi = {10.1186/s13408-020-00098-5},
	id = {Spek2020},
	isbn = {2190-8567},
	journal = {The Journal of Mathematical Neuroscience},
	number = {1},
	pages = {21},
	title = {Neural field models with transmission delays and diffusion},
	volume = {10},
	year = {2020}
    }

@misc{Gantouh2026-2,
      title={A spectral-based {ISS} small-gain theorem for boundary control systems with infinite couplings},
      author={Yassine El Gantouh and Jun Zheng and Guchuan Zhu and Dingshi Li},
      year={2026},
      eprint={2604.11031},
      archivePrefix={arXiv},
      primaryClass={math.OC}
}

@article{Gantouh2026-1,
	author = {Yassine El Gantouh and Yang Liu},
	doi = {https://doi.org/10.1016/j.sysconle.2025.106307},
	issn = {0167-6911},
	journal = {Systems \& Control Letters},
	pages = {106307},
	title = {Well-posedness and stability of boundary delay equations},
	volume = {208},
	year = {2026}
}

@article{SilgaBayili2021,
	author = {Roland Silga and Gilbert Bayili},
    title = {Polynomial stability of the wave equation with distributed delay term on the dynamical control},
	doi = {10.1515/msds-2020-0134},
	journal = {Nonautonomous Dynamical Systems},
	lastchecked = {2026-05-27},
	number = {1},
	pages = {207--227},
	url = {https://doi.org/10.1515/msds-2020-0134},
	volume = {8},
	year = {2021}
    }

@article{SilgaBayiliZabsonre2022,
author = {Roland Silga and Gilbert Bayili and Issa Zabsonre},
doi = {https://doi.org/10.28919/ejma.2022.2.15},
journal = {Eur. J. Math. Appl.},
url = {https://doi.org/10.28919/ejma.2022.2.15},
volume = {2},
title = {Stabilization of the {S}chr\"{o}dinger equation with distributed delay in boundary feedback},
year = {2022}
}

@article{Trostorff2014,
	author = {Sascha Trostorff},
	doi = {https://doi.org/10.1016/j.jfa.2014.08.009},
	issn = {0022-1236},
	journal = {Journal of Functional Analysis},
	number = {8},
	pages = {2787-2822},
	title = {A characterization of boundary conditions yielding maximal monotone operators},
	volume = {267},
	year = {2014}
    }

@article{PicardTrostorffWaurick2016,
	author = {Picard, Rainer and Trostorff, Sascha and Waurick, Marcus},
	doi = {https://doi.org/10.1093/imamci/dnu035},
	issn = {0265-0754},
	journal = {IMA Journal of Mathematical Control and Information},
	month = {06},
	number = {2},
	pages = {257-291},
	title = {On a comprehensive class of linear control problems},
	volume = {33},
	year = {2016}
    }

@book{PicardMcGhee2011,
	address = {Berlin, New York},
	author = {Rainer Picard and Des McGhee},
	doi = {https://doi.org/10.1515/9783110250275},
	isbn = {9783110250275},
	lastchecked = {2026-05-30},
	publisher = {De Gruyter},
	title = {Partial {D}ifferential {E}quations --- {A} unified {H}ilbert space approach},
	year = {2011}
    }

@book{JacobZwart2012,
 author = {Jacob, Birgit and Zwart, Hans J.},
 title = {Linear port-{Hamiltonian} systems on infinite-dimensional spaces.},
 fseries = {Operator Theory: Advances and Applications},
 series = {Oper. Theory: Adv. Appl.},
 issn = {0255-0156},
 volume = {223},
 isbn = {978-3-0348-0398-4; 978-3-0348-0399-1},
 year = {2012},
 publisher = {Basel: Birkh{\"a}user},
 language = {English},
 doi = {10.1007/978-3-0348-0399-1},
 zbMATH = {6023275},
 Zbl = {1254.93002}
}

@phdthesis{Trostorff2018,
    author      = {Sascha Trostorff},
    title       = {Exponential Stability and Initial Value Problems for Evolutionary Equations},
    type        = {Habilitation thesis},
    school      = {University of Technology Dresden},
    year        = {2018},
    note        = {\url{https://nbn-resolving.org/urn:nbn:de:bsz:14-qucosa-236494}},
}

@article{Waurick2023,
 author = {Frohberg, Johanna and Waurick, Marcus},
 title = {State-dependent delay differential equations on {$\mathrm{H}^1$}},
 fjournal = {Journal of Differential Equations},
 journal = {J. Differ. Equations},
 issn = {0022-0396},
 volume = {410},
 pages = {737--771},
 year = {2024},
 language = {English},
 doi = {10.1016/j.jde.2024.08.009},
 zbMATH = {7923719},
 Zbl = {1555.34073}
}

@misc{Aigner2024,
      title={A quick guide to ordinary state-dependent delay differential equations},
      author={Bernhard Aigner and Marcus Waurick},
      year={2024},
      eprint={2410.20613},
      archivePrefix={arXiv},
      primaryClass={math.CA},
}

@article{Aigner2026,
	author = {Bernhard Aigner and Nathanael Skrepek},
	doi = {https://doi.org/10.1016/j.jmaa.2025.130241},
	issn = {0022-247X},
	journal = {Journal of Mathematical Analysis and Applications},
	number = {1},
	pages = {130241},
	title = {Well-posedness and stability of the {L}agrange representation of the n-{D} wave equation via boundary triples},
	volume = {557},
	year = {2026}
    }

@article {SkrepekWaurick2024,
    AUTHOR = {Skrepek, Nathanael and Waurick, Marcus},
     TITLE = {Semi-uniform stabilization of anisotropic {M}axwell's equations via boundary feedback on split boundary},
   JOURNAL = {J. Differential Equations},
  FJOURNAL = {Journal of Differential Equations},
    VOLUME = {394},
      YEAR = {2024},
     PAGES = {345--374},
      ISSN = {0022-0396,1090-2732},
   MRCLASS = {93D15 (35L03 35Q61 47A40 47F05 78A25 93C20)},
  MRNUMBER = {4725773},
       DOI = {10.1016/j.jde.2024.03.021}
}

@book {Behrndt2020,
    AUTHOR = {Behrndt, Jussi and Hassi, Seppo and de Snoo, Henk},
     TITLE = {Boundary value problems, {W}eyl functions, and differential
              operators},
    SERIES = {Monographs in Mathematics},
    VOLUME = {108},
 PUBLISHER = {Birkh\"{a}user/Springer, Cham},
      YEAR = {[2020] \copyright 2020},
     PAGES = {vii+772},
      ISBN = {978-3-030-36713-8; 978-3-030-36714-5},
   MRCLASS = {34-02 (34B05 34B20 34B24 34Lxx 35Pxx 47A70 47B25)},
  MRNUMBER = {3971207},
MRREVIEWER = {Julio\ H.\ Toloza},
       DOI = {10.1007/978-3-030-36714-5}
}

@misc{Pignotti2023,
      title={Exponential decay estimates for semilinear wave-type equations with time-dependent time delay},
      author={Cristina Pignotti},
      year={2023},
      eprint={2303.14208},
      archivePrefix={arXiv},
      primaryClass={math.AP}
}

@article{Pignotti2016,
 author = {Fragnelli, G. and Pignotti, C.},
 title = {Stability of solutions to nonlinear wave equations with switching time delay},
 fjournal = {Dynamics of Partial Differential Equations},
 journal = {Dyn. Partial Differ. Equ.},
 issn = {1548-159X},
 volume = {13},
 number = {1},
 pages = {31--51},
 year = {2016},
 language = {English},
 doi = {10.4310/DPDE.2016.v13.n1.a2},
 zbMATH = {6585214},
 Zbl = {1350.35135}
}

@article{Pignotti2022,
 author = {Paolucci, Alessandro and Pignotti, Cristina},
 title = {Well-posedness and stability for semilinear wave-type equations with time delay},
 fjournal = {Discrete and Continuous Dynamical Systems. Series S},
 journal = {Discrete Contin. Dyn. Syst., Ser. S},
 issn = {1937-1632},
 volume = {15},
 number = {6},
 pages = {1561--1571},
 year = {2022},
 language = {English},
 doi = {10.3934/dcdss.2022049},
 zbMATH = {7539681},
 Zbl = {1492.93142}
}

@article{Pignotti2021,
	author = {Paolucci, Alessandro and Pignotti, Cristina},
	date = {2021/12/01},
	doi = {10.1007/s00498-021-00292-0},
	id = {Paolucci2021},
	isbn = {1435-568X},
	journal = {Mathematics of Control, Signals, and Systems},
	number = {4},
	pages = {617--636},
	title = {Exponential decay for semilinear wave equations with viscoelastic damping and delay feedback},
	volume = {33},
	year = {2021}
    }

@article{Pignotti2017,
	author = {Pignotti, Cristina},
	date = {2017/12/01},
	doi = {10.1007/s10884-016-9545-3},
	id = {Pignotti2017},
	isbn = {1572-9222},
	journal = {Journal of Dynamics and Differential Equations},
	number = {4},
	pages = {1309--1324},
	title = {Stability results for second-order evolution equations with memory and switching time-delay},
	volume = {29},
	year = {2017}
    }

@book{Necas2012,
 author = {Ne{\v{c}}as, Jind{\v{r}}ich},
 title = {Direct methods in the theory of elliptic equations. {Transl}. from the {French}. {Editorial} coordination and preface by {{\v{S}}{\'a}rka} {Ne{\v{c}}asov{\'a}} and a contribution by {Christian} {G}. {Simader}},
 fseries = {Springer Monographs in Mathematics},
 series = {Springer Monogr. Math.},
 issn = {1439-7382},
 isbn = {978-3-642-10454-1; 978-3-642-10455-8},
 year = {2012},
 publisher = {Berlin: Springer},
 language = {English},
 doi = {10.1007/978-3-642-10455-8},
 zbMATH = {5702707},
 Zbl = {1246.35005}
}

@article{Zwart2015,
	author = {Mikael Kurula and Hans Zwart},
	doi = {10.1080/00207179.2014.993337},
	journal = {International Journal of Control},
	number = {5},
	pages = {1063--1077},
	publisher = {Taylor \& Francis},
	title = {Linear wave systems on n-{D} spatial domains},
	volume = {88},
	year = {2015}
    }

@article{Nicaise2009,
 author = {Nicaise, Serge and Valein, Julie and Fridman, Emilia},
 title = {Stability of the heat and of the wave equations with boundary time-varying delays},
 fjournal = {Discrete and Continuous Dynamical Systems. Series S},
 journal = {Discrete Contin. Dyn. Syst., Ser. S},
 issn = {1937-1632},
 volume = {2},
 number = {3},
 pages = {559--581},
 year = {2009},
 language = {English},
 doi = {10.3934/dcdss.2009.2.559},
 zbMATH = {5613322},
 Zbl = {1171.93029}
}

@article{KunzeMuiPloss2026,
 author = {Kunze, Markus and Mui, Jonathan and Plo{\ss}, David},
 title = {Elliptic operators with non-local {Wentzell}-{Robin} boundary conditions},
 fjournal = {Journal of Spectral Theory},
 journal = {J. Spectr. Theory},
 issn = {1664-039X},
 volume = {16},
 number = {1},
 pages = {197--242},
 year = {2026},
 language = {English},
 doi = {10.4171/jst/595},
 zbMATH = {8161850}
}

@article{Goldstein2006,
 author = {Goldstein, Gis{\`e}le Ruiz},
 title = {Derivation and physical interpretation of general boundary conditions},
 fjournal = {Advances in Differential Equations},
 journal = {Adv. Differ. Equ.},
 issn = {1079-9389},
 volume = {11},
 number = {4},
 pages = {457--480},
 year = {2006},
 language = {English},
 zbMATH = {5068811},
 Zbl = {1107.35010}
}

@book{Batkai2005,
 author = {B{\'a}tkai, Andr{\'a}s and Piazzera, Susanna},
 title = {Semigroups for delay equations},
 fseries = {Research Notes in Mathematics},
 series = {Res. Notes Math.},
 volume = {10},
 isbn = {1-56881-243-4},
 year = {2005},
 publisher = {Wellesley, MA: A K Peters},
 language = {English},
 zbMATH = {2215036},
 Zbl = {1089.35001}
}

@book{Diekmann1995,
 author = {Diekmann, Odo and van Gils, Stephan A. and Verduyn Lunel, Sjoerd M. and Walther, Hans-Otto},
 title = {Delay equations. {Functional}-, complex-, and nonlinear analysis},
 fseries = {Applied Mathematical Sciences},
 series = {Appl. Math. Sci.},
 issn = {0066-5452},
 volume = {110},
 isbn = {0-387-94416-8},
 year = {1995},
 publisher = {New York, NY: Springer-Verlag},
 language = {English},
 zbMATH = {770062},
 Zbl = {0826.34002}
}

@article{Hernandez2006,
	author = {Eduardo Hern{\'a}ndez and Andr{\'e}a Prokopczyk and Luiz Ladeira},
	doi = {https://doi.org/10.1016/j.nonrwa.2005.03.014},
	issn = {1468-1218},
	journal = {Nonlinear Analysis: Real World Applications},
	number = {4},
	pages = {510-519},
	title = {A note on partial functional differential equations with state-dependent delay},
	volume = {7},
	year = {2006}
    }

@article{Rezounenko2007,
	author = {Alexander V. Rezounenko},
	doi = {https://doi.org/10.1016/j.jmaa.2006.03.049},
	issn = {0022-247X},
	journal = {Journal of Mathematical Analysis and Applications},
	number = {2},
	pages = {1031-1045},
	title = {Partial differential equations with discrete and distributed state-dependent delays},
	volume = {326},
	year = {2007}
    }

@article{Trostorff2015,
author = {Trostorff, Sascha},
title = {Exponential stability of a second-order integro-differential equation with delay},
journal = {Proceedings in Applied Mathematics and Mechanics},
volume = {15},
number = {1},
pages = {699-700},
doi = {https://doi.org/10.1002/pamm.201510339},
year = {2015}
}

@incollection{Buchinger2024,
 author = {Buchinger, Andreas and Doherty, Michael},
 title = {On some impedance boundary conditions for a thermo-piezo-electromagnetic system},
 booktitle = {Systems theory and PDEs. Open problems, recent results, and new directions. Based on the first workshop on systems theory and PDEs, WOSTAP, Freiberg, Germany, July 2022},
 isbn = {978-3-031-64990-5; 978-3-031-64993-6; 978-3-031-64991-2},
 pages = {1--24},
 year = {2024},
 publisher = {Cham: Birkh{\"a}user},
 language = {English},
 doi = {10.1007/978-3-031-64991-2_1},
 zbMATH = {7949838},
 Zbl = {1554.35316}
}

@phdthesis{Doherty2024,
    author      = {Michael Doherty},
    title       = {A model for ultrasonic transducers in a high-temperature regime with boundary dynamics: an evolutionary equations approach},
    type        = {PhD thesis},
    school      = {University of Strathclyde Glasgow},
    year        = {2024},
    note        = {\url{https://stax.strath.ac.uk/concern/theses/8k71nh74g}
    }
}

\end{document}